%% file: Fourth_moments_I_arxiv_v3.tex
\documentclass[oneside]{amsart}
\usepackage{geometry}                
\usepackage{graphicx}
\usepackage{amsmath}
\usepackage{amssymb}
\usepackage{amsthm}
\usepackage{mathtools}
\usepackage{mathscinet}
\usepackage{tikz-cd}
\usepackage{dsfont}
\usepackage{color}
\usepackage{enumitem}
\usepackage{tensor}
\usepackage{nicefrac}
\usepackage{todonotes}
\usepackage{mathrsfs}
\usepackage{float}

\usepackage[]{hyperref}
\hypersetup{
	pdftitle={Theta functions, fourth moments of eigenforms, and the sup-norm problem I},
	pdfauthor={Ilya Khayutin},
	bookmarksnumbered=true,
	bookmarksopen=true,
	bookmarksopenlevel=1,
	bookmarksdepth=3,
	colorlinks=true,
	citecolor=blue,
	pdfstartview=Fit,
	pdfpagemode=UseOutlines,
	pdfpagelayout=TwoPageRight
}

\usepackage[protrusion=true]{microtype}
\DeclareSymbolFont{legacymaths}{OT1}{cmr}{m}{n}
\SetSymbolFont{legacymaths}{bold}{OT1}{cmr}{bx}{n}
\usepackage[vvarbb,upint]{newtxmath}
\usepackage[cal=boondoxo]{mathalfa}

\makeatletter
\let\orgdescriptionlabel\descriptionlabel
\renewcommand*{\descriptionlabel}[1]{%
  \let\orglabel\label
  \let\label\@gobble
  \phantomsection
  \protected@edef\@currentlabel{#1}%
  \let\label\orglabel
  \orgdescriptionlabel{#1}%
}
\def\paragraph{
	\@startsection{paragraph}{4}
	\z@{.5\linespacing\@plus.7\linespacing}{-.5em}%
	{\normalfont\itshape}}
\makeatother

\DeclareFontFamily{U}{mathx}{\hyphenchar\font45}
\DeclareFontShape{U}{mathx}{m}{n}{
      <5> <6> <7> <8> <9> <10>
      <10.95> <12> <14.4> <17.28> <20.74> <24.88>
      mathx10
      }{}
\DeclareSymbolFont{mathx}{U}{mathx}{m}{n}
\DeclareFontSubstitution{U}{mathx}{m}{n}
\DeclareMathAccent{\widecheck}{0}{mathx}{"71}

\newcommand{\dif}{\,\mathrm{d}}

\DeclareMathOperator{\Id}{Id}

\DeclareMathOperator{\covol}{covol}

\DeclareMathOperator{\Nr}{Nr}

\DeclareMathOperator{\Tr}{Tr}
\DeclareMathOperator{\Lie}{Lie}

\DeclareMathOperator{\sign}{sign}

\DeclareMathOperator{\diag}{diag}
\DeclareMathOperator{\Span}{Span}
\DeclareMathOperator{\Sym}{Sym}

\DeclareMathOperator{\supp}{supp}

\DeclareMathOperator{\Fr}{\mathcal{F}}

\newcommand{\ord}{\operatorname{ord}}

\renewcommand{\det}{\operatorname{det}} 

\newlength{\faktorheight}
\newcommand*{\dfaktor}[3]{
\mathchoice{
 \settototalheight{\faktorheight}{\ensuremath{#1}}%
  \raisebox{-0.5\faktorheight}{\ensuremath{#1}}
   \backslash
   \settototalheight{\faktorheight}{\ensuremath{#2}}%
  \raisebox{0.5\faktorheight}{\ensuremath{#2}}
  \slash
   \settototalheight{\faktorheight}{\ensuremath{#3}}%
  \raisebox{-0.5\faktorheight}{\ensuremath{#3}}
  }{
  \ensuremath{#1}
  \backslash
  \ensuremath{#2}
  \slash
  \ensuremath{#3}
  }
  {
    \ensuremath{#1}
    \backslash
    \ensuremath{#2}
    \slash
    \ensuremath{#3}
  }
  {
    \ensuremath{#1}
    \backslash
    \ensuremath{#2}
    \slash
    \ensuremath{#3}
  }
}

\newcommand*{\lfaktor}[2]{
\mathchoice{
 \settototalheight{\faktorheight}{\ensuremath{#1}}%
  \raisebox{-0.5\faktorheight}{\ensuremath{#1}}
	\backslash
   \settototalheight{\faktorheight}{\ensuremath{#2}}%
  \raisebox{0.5\faktorheight}{\ensuremath{#2}}
}
{
	\ensuremath{#1}
	\backslash
	\ensuremath{#2}
}
{
	\ensuremath{#1}
	\backslash
	\ensuremath{#2}
}
{
	\ensuremath{#1}
	\backslash
	\ensuremath{#2}
}
}

\newcommand*{\faktor}[2]{
\mathchoice{
 \settototalheight{\faktorheight}{\ensuremath{#1}}%
  \raisebox{0.5\faktorheight}{\ensuremath{#1}}
  \slash
   \settototalheight{\faktorheight}{\ensuremath{#2}}%
  \raisebox{-0.5\faktorheight}{\ensuremath{#2}}
}
{
	\ensuremath{#1}
	\slash
	\ensuremath{#2}
}
{
	\ensuremath{#1}
	\slash
	\ensuremath{#2}
}
{
	\ensuremath{#1}
	\slash
	\ensuremath{#2}
}
}

\newcommand*{\cyclic}[1]{
	\faktor{\mathbb{Z}}{#1\mathbb{Z}}
}

\newcommand{\sm}{\left(\begin{smallmatrix}}
\newcommand{\esm}{\end{smallmatrix}\right)}
\newcommand{\bpm}{\begin{pmatrix}}
\newcommand{\ebpm}{\end{pmatrix}}

\newtheorem{thm}{Theorem}
\numberwithin{thm}{section}
\newtheorem*{thm*}{Statement of Theorem}
\newtheorem*{thm-quote}{Theorem}
\newtheorem{lem}[thm]{Lemma}
\newtheorem{prop}[thm]{Proposition}
\newtheorem{cor}[thm]{Corollary}

\theoremstyle{definition}
\newtheorem{defi}[thm]{Definition}

\theoremstyle{remark}
\newtheorem{remark}[thm]{Remark}
\newtheorem*{remark*}{Remark}
\newtheorem*{acknowledgements}{Acknowledgements}

\title[Theta functions, fourth moments of eigenforms, and the sup-norm problem I]{Theta functions, fourth moments of eigenforms, and the sup-norm problem I}
\author[I. Khayutin and R. Steiner]{Ilya Khayutin and Raphael S.\ Steiner}
\address{Department of Mathematics, Northwestern University, Evanston IL 60203, USA}
\address{Department of Mathematics, ETH Z{\"u}rich, 8092 Z\"urich, CH}


\begin{document}

\begin{abstract} We give sharp point-wise bounds in the weight-aspect on fourth moments of modular forms on arithmetic hyperbolic surfaces associated to Eichler orders. Thereby, we strengthen a result of Xia and extend it to co-compact lattices. We realize this fourth moment by constructing a holomorphic theta kernel on $\mathbf{G} \times \mathbf{G} \times \mathbf{SL}_{2}$, for $\mathbf{G}$ an indefinite inner-form of $\mathbf{SL}_2$ over $\mathbb{Q}$, based on the Bergman kernel, and considering its $L^2$-norm in the Weil variable. The constructed theta kernel further gives rise to new elementary theta series for integral quadratic forms of signature $(2,2)$.
\end{abstract}
\maketitle

{
\hypersetup{linkcolor=blue}
\tableofcontents
}
\input{Intro_arxiv_v3}

\section{The Weil Representation and Theta Series}
\label{sec:weilreptheta}

\subsection{Inner-forms of \texorpdfstring{$\mathbf{SL}_2$}{SL2}}
Let $B$ be a quadratic central simple algebra over $\mathbb{Q}$ and for each place $v$ denote $B_v\coloneqq B\otimes \mathbb{Q}_v$. We define the affine algebraic group $\mathbf{G}$ over $\mathbb{Q}$ as representing the group functor
\begin{equation*}
\mathbf{G}(L)=\left\{x\in B\otimes L \mid \Nr(x)=1\right\}
\end{equation*}
for all $\mathbb{Q}$-algebras $L$. The group $\mathbf{G}$ is an inner-form of $\mathbf{SL}_2$, and all inner-forms of $\mathbf{SL}_2$ over $\mathbb{Q}$ arise this way.

Fix a maximal order $R\subset B$ and define $R_v$ to be the completion of $R$ in $B_v$. For each finite place $v$ the order $R_v$ is maximal in $B_v$. For $v<\infty$ set $K_v\coloneqq R_v^1 < \mathbf{G}(\mathbb{Q}_v)$ to be the group of norm $1$ elements in $R_v$.
If $B$ splits over $\mathbb{R}$ we fix once and for all an isomorphism $B_\infty\simeq \operatorname{Mat}_{2\times 2}(\mathbb{R})$ and use it to identify the two spaces.
We then set $K_\infty=\mathbf{SO}_2(\mathbb{R})$ if $B$ splits over $\mathbb{R}$ and $K_\infty=\mathbf{G}(\mathbb{R})$ otherwise. For almost all $v$ the group $K_v$ is a hyperspecial maximal compact subgroup of $\mathbf{G}(\mathbb{Q}_v)$. We define $\mathbf{G}(\mathbb{A})$ as the restricted direct product
\begin{equation*}
\mathbf{G}(\mathbb{A})\coloneq \left\{ (g_v)_v\in \prod_{v\leq \infty} \mathbf{G}(\mathbb{Q}_v) \mid g_v\in K_v \textrm{ for almost all } v \right\}\;.
\end{equation*}

\subsection{Normalization of Haar Measures.}
For a linear algebraic group $\mathbf{L}$ defined over $\mathbb{Q}$ we denote $[\mathbf{L}(\mathbb{A})]\coloneqq \lfaktor{\mathbf{L}(\mathbb{Q})}{\mathbf{L}(\mathbb{A})}$. Assume $[\mathbf{L}(\mathbb{A})]$ is of finite volume. We shall always integrate with respect to the probability Haar measure on $[\mathbf{L}(\mathbb{A})]$. Let $U<\mathbf{G}(\mathbb{A}_f)$ be a compact open subgroup. Then, $\mathbf{L}(\mathbb{R})$ acts on $[\mathbf{L}(\mathbb{A}))]_U\coloneqq\dfaktor{\mathbf{L}(\mathbb{Q})}{\mathbf{L}(\mathbb{A})}{U}$ with finitely many orbits \cite{BorelFiniteness}, and $[\mathbf{L}(\mathbb{A})]_U\simeq \bigsqcup_i \lfaktor{\Gamma_i}{\mathbf{L}(\mathbb{R})}$ with $\Gamma_i<\mathbf{G}(\mathbb{R})$ finitely many lattices. On  $[\mathbf{L}(\mathbb{A})]_U$ we integrate with respect to the push-forward of the probability Haar measure on $[\mathbf{L}(\mathbb{A})]$. This measure is evidently an $\mathbf{L}(\mathbb{R})$-invariant probability measure. If $[\mathbf{L}(\mathbb{A})]_U\simeq \lfaktor{\Gamma}{\mathbf{L}(\mathbb{R})}$ is a single $\mathbf{L}(\mathbb{R})$-orbit, then this measure is the probability Haar measure on $\lfaktor{\Gamma}{\mathbf{L}(\mathbb{R})}$.

On $\faktor{\mathbf{SL}_2(\mathbb{R})}{\mathbf{SO_2(\mathbb{R})}}$ and $\faktor{\mathbf{Spin}_3(\mathbb{R})}{\mathbf{SO}_2(\mathbb{R})}$ we fix the standard Haar measures corresponding to the volume form of Gaussian curvature $\pm 1$ on the hyperbolic plane and the $2$-sphere. 
We fix the  unique Haar measures on $\mathbf{SL}_2(\mathbb{R})$ and $\mathbf{Spin}_3(\mathbb{R})$ whose push-forward to the symmetric space coincides with the measure above.

On $\mathbf{SL}_2(\mathbb{Q}_p)$ and $\mathbf{PGL}_2(\mathbb{Q}_p)$ we fix the Haar measure giving volume $1$ to $\mathbf{SL}_2(\mathbb{Z}_p)$ and $\mathbf{PGL}_2(\mathbb{Z}_p)$ respectively. Let $\mathbb{D}_p$ be the unique ramified quaternion algebra over $\mathbb{Q}_p$ with ring of integers $\mathcal{O}(\mathbb{D}_p)$. Denote by $\mathbb{D}_p^{(1)}$ the group of norm $1$ element in $\mathbb{D}_p$. We fix the Haar measures on $\mathbb{D}_p^{(1)}$, $\lfaktor{\mathbb{Q}_p^\times}{\mathbb{D}_p^\times}$ that give volume $1$ to the compact open subgroups $\mathcal{O}(\mathbb{D}_p)\cap \mathbb{D}_p^{(1)}$, $\lfaktor{\mathbb{Z}_p^\times}{\mathcal{O}(\mathbb{D}_p)^\times}$ respectively. These choices fix a Haar measure $m_{\mathbf{G}(\mathbb{Q}_p)}$ on $\mathbf{G}(\mathbb{Q}_p)$ for all primes $p$.

The product of the local Haar measures $m_{\mathbf{G}(\mathbb{Q}_p)}$ at all primes $p$ induce a Haar measure on $\mathbf{G}(\mathbb{A}_f)=\prod'_p \mathbf{G}(\mathbb{Q}_p)$, which we call the \emph{unnormalized} Haar measure on $\mathbf{G}(\mathbb{A}_f)$. Similarly we call the product of the fixed Haar measure on $\mathbf{G}(\mathbb{R})$ with the unnormalized Haar measure on $\mathbf{G}(\mathbb{A}_f)$, the \emph{unnormalized} Haar measure on $\mathbf{G}(\mathbb{A})$. The unnormalized Haar measure on $\mathbf{G}(\mathbb{A})$ is necessarily proportional to the covolume $1$ measure, but they are \emph{not} equal. Our local measure normalization forces $m_{\mathbf{G}(\mathbb{Q}_p)}(K_p)=1$ for all primes $p$, hence the volume of $[\mathbf{G}(\mathbb{A})]$ with respect to the unnormalized measure is not $1$, but rather the sum of the volumes $m_{\mathbf{G}(\mathbb{R})}(\lfaktor{R_i^{(1)}}{\mathbf{G}(\mathbb{R})})$ for  orders $R_i\subset B$ representing all the classes in the class set\footnote{That is, all orders everywhere locally conjugate to $R$ by a norm $1$ element, where two orders are equivalent if they are  globally conjugate by a rational norm $1$ element.} of $R$, where $R$ is the maximal order from above.  Denote by $\varrho_{\mathbf{G}}$ the volume of $[\mathbf{G}(\mathbb{A})]$ with respect to the unnormalized measure.
In the indefinite case the class number is $1$ and the volume is $\varrho_{\mathbf{G}}=\frac{\pi}{3} \varphi(D_B)$ \cite[Theorem 39.1.2]{VoightQA}. Exactly, the same formula holds in the definite case, due to the Eichler mass formula \cite[Theorem 25.1.1]{VoightQA}. We henceforth fix the Haar measure $m_{\mathbf{G}(\mathbb{A}_f)}$ on $\mathbf{G}(\mathbb{A}_f)$ to be the measure induced by $\varrho_{\mathbf{G}}^{-1} \prod_p m_{\mathbf{G}(\mathbb{Q}_p)}$. The product $m_{\mathbf{G}(\mathbb{R})}\times m_{\mathbf{G}(\mathbb{A}_f)}$ is the co-volume $1$ Haar measure. The same discussion applies mutatis mutandi to $\mathbf{SL}_2$.

Note that we have several normalizations of the Haar measure on $\mathbf{G}(\mathbb{R})$. When integrating over a quotient by a lattice $\lfaktor{\Gamma}{\mathbf{G}(\mathbb{R})}$ we always use the co-volume $1$ Haar measure. When integrating over $\mathbf{G}(\mathbb{R})$ we use the standard measure $m_{\mathbf{G}(\mathbb{R})}$ which is not a co-volume $1$ measure in general. The discrepancy is accounted for by the factor $\varrho_{\mathbf{G}} ^{-1}$ in the Haar measure of $\mathbf{G}(\mathbb{A}_f)$. The same discussion applies to $\mathbf{SL}_2$. 

\subsection{Local Weil Representation}
In this section, the field $F=\mathbb{Q}_v$ is a completion of $\mathbb{Q}$ at a place $v$, then $B_v$ is a quadratic central simple algebra over $F$, i.e.\ $B_v=\operatorname{Mat}_{2\times 2}(F)$ or $B_v$ is the unique quadratic division algebra over $F$. Denote by $x\mapsto \tensor[^\iota]{x}{}$ the canonical involution on $B_v$. In the split case, the involution sends a matrix to its adjugate. Denote the reduced norm on $B_v$ by $\Nr$ and the reduced trace by $\Tr$. We shall also fix a unitary additive character $\psi_v\colon F\to \mathbb{C}^\times$.
In this section, we recall the construction and elementary properties of the Weil representation.

The vector space $B_v$ is endowed with an additive Haar measure. For an integrable function $M\colon B_v\to\mathbb{C}$, we define the Fourier transform by
\begin{equation*}
\Fr{M}(x)=\int M(y)\psi_v\left(\left\langle x,y \right\rangle\right) \dif y\;,
\end{equation*}
where the bilinear form $\langle\, ,\rangle$ is defined by
\begin{equation*}
\left\langle x,y \right\rangle\coloneqq \Tr(x \tensor[^\iota]{y}{})\;.
\end{equation*}
Notice that this is twice the polarization of the norm quadratic form, i.e.\ $\langle x,x\rangle=2 \Nr x$.
We normalize the measure on $B_v$ so that it is Fourier self-dual, i.e.\ $\Fr^2 M (x)=M(-x)$ for a Schwartz function $M$.

If $v$ is non-archimedean denote by $\Omega_v$ the space of Schwartz--Bruhat functions on $B_v$, i.e.\ locally constant functions of compact support. In the archimedean place, we need to consider a space that differs from the space of Schwartz functions because the Bergman kernel does not arise from a Schwartz function. To construct $\Omega_\infty$, we will start first with a larger space $L^2(B_\infty)$
and then restrict the Weil representation to a subspace $\Omega_\infty$ to be defined later.

The Weil representation of $\mathbf{SL}_2(F)$ on $\Omega_v$, $L^2(B_\infty)$ satisfies
\begin{align*}
\rho\left(\begin{pmatrix}
1 & \sigma \\ 0 & 1
\end{pmatrix}\right) M (x) &= \psi_v\left(\sigma \Nr(x)\right) M(x)\;,\\
\rho\left(\begin{pmatrix}
\lambda & 0 \\ 0 & \lambda^{-1}
\end{pmatrix}\right) M (x) &= |\lambda|_v^{2} M(\lambda x)\;,\\
\rho\left(\begin{pmatrix}
0 & 1 \\ -1 & 0
\end{pmatrix}\right) M (x) &= \gamma \Fr{M}(x)\;,
\end{align*}
where $\gamma=1$ if $B_v$ is split and $\gamma=-1$ otherwise. For a proof that this defines a representation see \cite[\S 1.1]{JacquetLanglands}.

Notice that the representation depends on the choice of an additive character $\psi_v$. We will usually suppress this dependence in the notation, but when we need to keep track of the character we shall write $\rho_{\psi_v}$.   Because $\mathbb{Q}_v$ is Fourier self-dual, all non-trivial additive characters are of the form $\tensor[^\varpi]{\psi}{_v}(a)=\psi_v(a\varpi)$ for some $\varpi\in\mathbb{Q}_v^\times$.  We see that
\begin{equation}\label{eq:weil-rep-change-of-character}
\rho_{\tensor[^\varpi]{\psi}{_v}}(g)=\rho_{\psi_v}(\diag(\varpi,1) g \diag(\varpi,1)^{-1})\;.
\end{equation}

\begin{lem}
Let $\operatorname{O}(B_v,\Nr)$ be the group of linear transformations preserving the norm form, this group acts on functions by $u.M(x)=M(u^{-1}x)$. The action of the orthogonal group $\operatorname{O}(B_v,\Nr)$ commutes with the action of $\mathbf{SL}_2(F)$ via $\rho$.
\end{lem}
\begin{proof}
It is sufficient to verify the claim for each of the formul{\ae} above. The action of the upper triangular matrices commutes with the action of any linear transformation that preserves the norm. The Fourier transform intertwines the action of $L\in\operatorname{GL}(B_v)$ with the action of $|\det L|_v^{-1}\tensor[^t]{L}{^{-1}}$. Hence, it commutes with orthogonal transformations.
\end{proof}

\begin{cor}
The $\rho$ action of $\mathbf{SL}_2(F)$ commutes with the right and left actions of $\mathbf{G}(F)$ by multiplication. Moreover, the $\rho$ action commutes with the $B^\times$-action by conjugation.
\end{cor}
\begin{proof}
The actions of $B^\times$ and $\mathbf{G}(F)$ preserve the norm form, hence they factor through the orthogonal group.
\end{proof}

\begin{lem}
The Weil representation is a continuous unitary representation of $\mathbf{SL}_2(F)$ on $\Omega_v$, $L^2(B_\infty)$.
\end{lem}
\begin{proof}
This is established by Weil \cite{Weil} for the space of Schwartz or Schwartz--Bruhat functions.
The same proof works for $L^2(B_\infty)$.
\end{proof}

\subsection{The Archimedean Weil Representation}\label{sec:archimedean-weil}
To construct the appropriate subspace $\Omega_\infty\subset L^2(B_\infty)$, we will use a method based on the work of Vign\'eras \cite{Vigneras}. We define the Laplacian $\Delta$ on $B_\infty$ as the Fourier multiplier operator with symbol $-4\pi^2 \Nr$. Write the archimedean additive character as $\psi_\infty(a)=\exp(2\pi i a \varpi)$ and consider the PDE
\begin{equation}\label{eq:harmonic-oscillator}
-\Delta M(x)+\omega^2 \Nr(x)M(x)= \frac{\omega m }{\varpi}M(x)\;,
\end{equation}
where $m\in\mathbb{Z}$ and $\omega>0$.
This is nothing but the PDE for energy eigenstates of four independent quantum harmonic oscillators with total energy $(\omega\varpi)  m$ and angular frequency $\omega \varpi$. We call $m$ the \emph{quantum number} of the equation and we denote by $V_{m,\omega}$ the $L^2$-closure of the space of Schwartz solutions to \eqref{eq:harmonic-oscillator}.
Notice that unlike the standard harmonic oscillator, the individual oscillators may have either positive or negative energy depending on the signature of the quadratic form $\Nr$. 

We fix henceforth $\psi_\infty(a)=\exp(2\pi ia)$, i.e. $\varpi=1$.
Consider the densely defined linear operator $L_\omega\colon L^2(B_\infty)\to L^2(B_\infty)$ given by $L_\omega[M]=-\Delta M +\omega^2 \Nr(x)\cdot M$ with the domain of Schwartz functions $D(L_\omega)=\mathcal{S}(B_\infty)$. Then, $L_\omega$ is real, i.e.\ $\langle L_\omega[M],M \rangle\in\mathbb{R}$ for all $M\in\mathcal{S}(B_\infty)$. Hence, $L_\omega$ is symmetric.
For explicitness, we state the following classical linear algebra lemma.
\begin{lem}\label{lem:V-orth}
The spaces $\left\{V_{m,\omega}\right\}_{m\in\mathbb{Z}}$ are mutually orthogonal.
\end{lem}
\begin{proof}
It is enough to show that if $M$,$M'$ are Schwartz solutions to \eqref{eq:harmonic-oscillator} with quantum numbers $m\neq m'$ then $\langle M, M' \rangle=0$. Because the operator $L_\omega$ is symmetric, we have $\langle {\omega m} M,M' \rangle = \langle L[M],M' \rangle=\langle M, L[M'] \rangle=\langle M, {\omega m'} M' \rangle$. We  deduce that $\langle M,M' \rangle=0$ in the usual fashion.
\end{proof}

\begin{lem}\label{lem:k-action-2pi}
Let $k_{\theta}\coloneqq \begin{pmatrix} \cos\theta & \sin\theta \\ -\sin\theta & \cos\theta \end{pmatrix}\in \mathbf{SO}_2(\mathbb{R})$ and set $\psi_\infty(a)=\exp(2\pi i a)$. Then, for every $M\in V_{m,2\pi}$ we have
\begin{equation*}
\left(\rho\left(k_{\theta}\right).M\right) (x)= e^{im\theta} M(x)\;.
\end{equation*}
Moreover, $L^2(B_\infty)=\bigoplus_{m\in\mathbb{Z}} V_{m,2\pi}$. Therefore, $V_{m,2\pi}$ is the $\left(\rho(\mathbf{SO}_2(\mathbb{R})),e^{im\theta}\right)$-isotypic subspace of $L^2(B_\infty)$. 
\end{lem}

The idea to use the one-dimensional Hermite functions in the proof of the lemma has been suggested to us by J. Wunsch.
\begin{proof}
We establish first the direct sum decomposition.
Recall that we need the Laplacian in \eqref{eq:harmonic-oscillator} to be defined consistently as having Fourier symbol $-4\pi^2 \Nr$. Choose a coordinate system $x_1,\ldots,x_4$ for $B_\infty$ such that $\Nr(x)=\sum_{i=1}^4 \epsilon_i x_i^2$ with $\epsilon_i\in\{\pm 1\}$. The Laplacian for our fixed character can be written in this coordinate system as
\begin{equation*}
\frac{1}{4}\sum_{i=1}^{4} \epsilon_i \frac{\partial^2}{\partial x_i^2}\;.
\end{equation*}

The $1/4$ factor appears because the Fourier transform is defined with respect to the bilinear form $\sum_{i=1}^4 2\epsilon_i x_i y_i$. The space of solutions to the one-dimensional quantum harmonic oscillator with angular frequency $4\pi$, $n\in \mathbb{Z}_{\geq 0}$,
\begin{equation*}
-\frac{1}{4}f''(x)+4\pi^2 x^2 f(x)=(2n+1)\pi f(x)
\end{equation*}
is one-dimensional and the $L^2$-normalized solution is
\begin{equation*}
f_n(x)\coloneqq\frac{1}{\sqrt{2^{n-1} n!}}\exp\left(-2\pi x^2\right) H_n(2\sqrt{\pi} x)\;,
\end{equation*}
where $H_n$ are the physicist's Hermite polynomials. Moreover,  these solutions form an orthonormal basis of the Hilbert space $L^2(\mathbb{R})$.
Define for every $\underline{n}=(n_1,n_2,n_3,n_4)\in \mathbb{Z}_{\geq 0}^4$ the function $M_{\underline{n}}\colon B_\infty\to \mathbb{C}$ by
\begin{equation*}
M_{\underline{n}}(x) = \prod_{j=1}^4 f_{n_j}(x_j) =  \prod_{j=1}^4 \frac{1}{\sqrt{2^{n_j-1}n_j!}} \exp\left(-2\pi x_j^2\right) H_{n_j}(2\sqrt{\pi} {x_j})\;,
\end{equation*}
where $x_1,\ldots,x_4$ are the normal form coordinates for the quadratic form $\Nr(x)$. Because $L^2(B_\infty)\simeq L^2(\mathbb{R})^{\bigotimes 4}$, we deduce that the functions $M_{\underline{n}}$ form an orthonormal basis of $L^2(B_\infty)$. These are Schwartz functions, and a separation of variables computation shows that $M_{\underline{n}}$ solves \eqref{eq:harmonic-oscillator} with\footnote{The sum of four odd numbers $\sum_{j=1}^4 \epsilon_j (2n_j+1)$ is always even.} $2m=\sum_{j=1}^4 \epsilon_j (2n_j+1)$. This and Lemma \ref{lem:V-orth} establish that $L^2(B_\infty)=\bigoplus_{m\in\mathbb{Z}} V_{m,2\pi}$ as claimed.

We need to prove that if $M\in V_{m,2\pi}$ then $\rho(k_\theta).M=e^{i m \theta}M$ for all $\theta\in[0,2\pi)$. By continuity of the Weil representation it is enough to establish this for Schwartz functions. Because Schwartz functions are smooth vectors for the Weil representation it is sufficient to show $\frac{\dif}{\dif \theta} \rho(k_\theta).M=i m \left(\rho(k_\theta).M\right)$.  Because the group $\mathbf{SO}_2(\mathbb{R})$ is abelian it is enough to verify this ODE at $\theta=0$. 
The formula $k_\theta=\exp( \theta w) $ for $w=\sm
0 & 1 \\ -1 & 0
\esm$ implies that the  ODE at $\theta=0$ is equivalent to
\begin{equation}\label{eq:Hermite-ode}
\dif \rho(w). M= im M\;,
\end{equation}
where $\dif \rho$ is the Lie algebra representation of $\mathfrak{sl}_2(\mathbb{R})$ on $\mathcal{S}(B_\infty)$ differentiated from the Weil representation of $\mathbf{SL}_2(\mathbb{R})$. Using the definition of the Weil action for upper diagonal unipotent matrices one easily computes that $\left(\dif\rho \sm 0 & 1 \\ 0 & 0 \esm.M\right)(x)=2\pi i \Nr (x) M(x)$. The formula $\sm
0 & 0 \\ -1 & 0
\esm=w^{-1} \sm 0 & 1 \\ 0 & 0 \esm w$, then implies that $\dif\rho \sm 0 & 0 \\ -1 & 0 \esm.M=\frac{1}{2\pi i}\Delta M$, and $(\dif \rho(w).M)(x)=\frac{1}{2\pi i} \Delta M(x)+2\pi i \Nr(x) M(x)$. Thus \eqref{eq:Hermite-ode} is equivalent to \eqref{eq:harmonic-oscillator}.
\end{proof}

\begin{cor}\label{cor:Weil-infinity-action}
Let $M\in V_{\omega,m}$ for arbitrary $\omega>0$ and fix $g=\sm
a & b\\ c & d
\esm\in\mathbf{SL}_2(\mathbb{R})$. Then,
\begin{equation*}
(\rho(g).M)(x)=
\frac{2 \pi}{\omega}
\frac{1}{D^2}
\psi_\infty\left(\frac{bd \frac{2\pi}{\omega} +ac \frac{\omega}{2\pi}}{D^2} \Nr x \right)\left(\sqrt{\frac{2\pi}{\omega}}\frac{d}{D}-i\sqrt{\frac{\omega}{2\pi}}\frac{c}{D}\right)^m
M\left(\sqrt{\frac{2 \pi }{ \omega}}\frac{x}{D}\right)\;,
\end{equation*}
where $D=\sqrt{c^2 \frac{\omega}{2 \pi} + d^2\frac{2 \pi}{\omega}}$.
\end{cor}
\begin{proof}
Denote $a\coloneqq\diag\left(\sqrt{\frac{2\pi}{\omega}},\sqrt{\frac{\omega}{2\pi}}\right)$ and write $\rho(g).M=\rho(g a^{-1}) \rho(a).M$. The Iwasawa decomposition of $g a^{-1}$ is
\begin{equation*}
g a^{-1}=\begin{pmatrix}
1 & \frac{bd \frac{2 \pi}{\omega}+ac \frac{\omega}{2 \pi}}{D^2}
\\ 0 & 1
\end{pmatrix}
\begin{pmatrix}
\frac{1}{D} & 0
\\ 0 & D
\end{pmatrix}
\begin{pmatrix}
\sqrt{\frac{2\pi}{\omega}}\frac{d}{D}
& -\sqrt{\frac{\omega}{2\pi}}\frac{c}{D}\\
\sqrt{\frac{\omega}{2\pi}}\frac{c}{D} &
\sqrt{\frac{2\pi}{\omega}}\frac{d}{D}
\end{pmatrix}\;.
\end{equation*}
The function
\begin{equation*}
(\rho(a).M)(x) =\frac{2 \pi}{\omega}M\left(\sqrt{\frac{2 \pi}{\omega}}x\right)
\end{equation*}
is a solution of $\eqref{eq:harmonic-oscillator}$ with angular frequency $2\pi$. Hence, we can apply Lemma \ref{lem:k-action-2pi} to $\rho(g a^{-1}) \rho(a).M$ and the Iwasawa decomposition of $g a^{-1}$.
\end{proof}

\begin{defi}\label{defi:Omega_infty}
Fix $m\in\mathbb{Z}$.
We are now ready to define the subspace $\Omega_\infty<L^2(B_\infty)$. This space will depend on a choice of $m$. Recall that $V_{m,\omega}$ is the $L^2$-completion of the space of solutions of the quantum harmonic oscillator equation \eqref{eq:harmonic-oscillator} for a fixed $m\in\mathbb{Z}$ and $\omega>0$. Define
\begin{equation*}
\Omega_\infty\coloneqq \Span_{\mathbb{C}}\left\{\psi_\infty(\sigma \Nr (x)) M(x) \mid \sigma\in\mathbb{R}\;, \exists \omega>0 \colon M\in V_{m,\omega} \textrm{ and } \exists \delta>0\colon |M(x)|\ll (1+\|x\|)^{-4-\delta} \right\}\;.
\end{equation*}
The span allows only for \emph{finite} linear combinations. In other words, $\Omega_\infty$ is the space generated by orbits of functions in $\bigcup_{\omega>0} V_{m,\omega}$ satisfying a decay condition at infinity under the Weil action of unipotent matrices.
The decay condition implies that any function in $\Omega_\infty$ is in $L^p(B_\infty)$ for all $p\geq 1$.
\end{defi}

\begin{prop}
The space $\Omega_\infty$ is invariant under the Weil representation and the action of $\operatorname{O}(B_\infty,\Nr)$.
\end{prop}
\begin{proof}
The space $V_{m,\omega}$ is invariant under $\operatorname{O}(B_\infty,\Nr)$ because equation \eqref{eq:harmonic-oscillator} commutes with orthogonal transformations. Also, the decay condition is invariant under orthogonal transformations.
Invariance under the Weil action follows from Corollary \ref{cor:Weil-infinity-action}.
\end{proof}

\begin{remark*} We note that we may assume the functions in $\Omega_{\infty}$ to be continuous, since we may replace them with the Fourier inverse of its Fourier transform as it converges absolutely uniformly on compacta due to the decay condition.
\end{remark*}

\subsection{The Non-archimedean Weil Representation}
We now describe the interaction between the Weil representation and Eichler orders in $B_v$ for $v<\infty$. In this section, we fix a prime $p$ and write $F=\mathbb{Q}_p$ and set $v$ to be the associated place. For clarity of notation, we will write $B_p\coloneqq B_v$.
We assume that $\psi_p=\psi_v$ is an unramified character.

\begin{defi}
Let $\mathcal{R}\subset B_p$ be an order. Then, the dual lattice $\widehat{\mathcal{R}}$ is defined as
\begin{equation*}
\mathcal{\widehat{\mathcal{R}}}=\left\{x \in B_p \mid \forall x\in \mathcal{R}\colon \Tr(x\tensor[^\iota]{y}{})\in \mathbb{Z}_p \right\}\;.
\end{equation*}
\end{defi}

We begin by discussing maximal orders.
\begin{defi}
Set $U_0(p^n)<\mathbf{SL}_2(\mathbb{Z}_p)$ to be the congruence subgroup defined by
\begin{equation*}
U_0(p^n)\coloneqq \begin{pmatrix}
\mathbb{Z}_p & \mathbb{Z}_p \\ p^n\mathbb{Z}_p & \mathbb{Z}_p
\end{pmatrix} \cap \mathbf{SL}_2(\mathbb{Z}_p)\;.
\end{equation*}
\end{defi}
\begin{lem}\label{lem:non-arch-Weil-maximal-order}
Let $\mathcal{R}\subset B_p$ be a maximal order. If $B_p$ is split then $\rho\left(\mathbf{SL}_2(\mathbb{Z}_p)\right).\mathds{1}_{\mathcal{R}}=\mathds{1}_{\mathcal{R}}$. If $B_p$ is ramified then $\rho\left(U_0(p)\right).\mathds{1}_{\mathcal{R}}=\mathds{1}_{\mathcal{R}}$.
\end{lem}
\begin{remark*}
This lemma is closely related to Lemmata 7 and 10 of \cite{Shimizu}.
\end{remark*}
\begin{proof}
All maximal orders in $B_p$ are conjugate to each other by an element of $B_p^\times$. Because the Weil action commutes with conjugation, it is enough to prove the claim for a specific maximal order. Moreover, the group $\mathbf{SL}_2(\mathbb{Z}_p)$ is generated by the subgroup $\mathcal{P}<\mathbf{SL}_2(\mathbb{Z}_p)$ of upper triangular integral matrices and the involution $w$. The fact that $\mathds{1}_\mathcal{R}$ is invariant under $\mathcal{P}$ follows because we have assumed $\psi_v$ is unramified. If $B_p$ is split, fix an isomorphism $B_p\simeq \operatorname{Mat}_{2\times 2}(\mathbb{Q}_p)$ and we need only show that $\mathds{1}_{\operatorname{Mat}_{2\times 2}(\mathbb{Z}_p)}$ is invariant under the Fourier transform. This follows from the fact that $\mathds{1}_{\mathbb{Z}_p}$ is invariant under the Fourier transform on $\mathbb{Q}_p$ with an unramified character.

If $B_p$ is a division algebra, we need only show invariance under the element $\sm
1 & 0 \\ p & 1
\esm =-w \sm
1 & -p\\ 0 &  1
\esm  w$. This element and the upper triangular integral matrices generate $U_0(p)$. Because of the duality of the Fourier transform, this is equivalent to showing that $\rho\left(\sm
1 & -p\\ 0 &  1
\esm\right).\Fr\mathds{1}_{\mathcal{R}}=\Fr\mathds{1}_{\mathcal{R}}$.

Let $E=F(\sqrt{a})/F$ be the unique unramified quadratic extension and write $B_p$ as the cyclic algebra $\bigl(\frac{a,p}{\mathbb{Q}_p}\bigr)$ with the standard generators $i,j,k$ and $i^2=a$, $j^2=p$ and $ij=-ji=k$. As usual, we identify $E$ with the sub-ring $\mathbb{Q}_p+i\mathbb{Q}_p< B_p$.
Denote by $\mathcal{O}_E$ the maximal order of $E$. Then, the unique maximal order of $B_p$ is $\mathcal{R}\simeq \mathcal{O}_E+j\mathcal{O}_E$. The Fourier transform on $B_p$ descends to the standard Fourier transform on $E$ with an unramified character. Identifying $B_p\simeq E\times E$ via $a+jb\mapsto (a,b)$, we can write
the Fourier self-dual measure on $B_p$ in these coordinates as $p^{-1}m_E\times m_E$. The $p^{-1}$ factor normalizes the measure to be self-dual.

The Fourier transform on $E$ satisfies $\Fr \mathds{1}_{\mathcal{O}_E}=\mathds{1}_{\mathcal{O}_E}$. An explicit computation with the Fourier self-dual measure implies
\begin{equation*}
\Fr \mathds{1}_{\mathcal{O}_E+j\mathcal{O}_E}=p^{-1}\mathds{1}_{\mathcal{O}_E+j^{-1}\mathcal{O}_E}\;.
\end{equation*}
Hence, $\Nr x\in p^{-1}\mathbb{Z}_p$ for all $x\in\supp \Fr \mathds{1}_{\mathcal{R}}$,
from which we deduce $\left(\rho\left(\sm
1 & -p \\ 0 & 1
\esm\right). \Fr \mathds{1}_{\mathcal{R}}\right)(x)=\psi_v(-p\Nr x) \left(\Fr \mathds{1}_{\mathcal{R}}\right)(x)=\left(\Fr \mathds{1}_{\mathcal{R}}\right)(x)$ and the claim follows.
\end{proof}

\begin{lem}\label{lem:jmap-local-ramified}
Assume $B_p$ is ramified and let $\mathcal{R}\subset B_p$ be the unique maximal order. Then, there is an isomorphism of finite abelian additive groups
\begin{equation*}
\jmath_v\colon \widehat{\mathcal{R}}/\mathcal{R}\to \mathbb{F}_{p^2}
\end{equation*}
such that $-p\Nr x \bmod p\mathbb{Z}_p \equiv \Nr \jmath_v(x)$ for all $x\in \widehat{\mathcal{R}}$.
The norm on the right hand side is the field norm $\mathbb{F}_{p^2}\to\mathbb{F}_p$.

Moreover, if $j$ is a uniformizer of $\mathcal{R}$ then we can choose $\jmath$ so that the composite map
\begin{equation*}
\mathcal{R}/j\mathcal{R}\xrightarrow{x\mapsto j^{-1}x}\widehat{\mathcal{R}}/\mathcal{R}\xrightarrow{\jmath_v} \mathbb{F}_{p^2}
\end{equation*}
is a field isomorphism.
\end{lem}
Note that there are exactly two field isomorphisms $\mathcal{R}/j\mathcal{R}\to \mathbb{F}_{p^2}$ and they differ by post-composition with the Frobenious, i.e.\ by the action of the Galois group. If $f\colon\mathcal{R}/j\mathcal{R}\to \mathbb{F}_{p^2}$ is such an isomorphism then $f(j x j^{-1})=f(x)^p$, i.e.\ conjugation by $j$ is intertwined with the Frobenious. 
Hence the composition of $\mathcal{R}/j\mathcal{R}\xrightarrow{x\mapsto xj^{-1}}\widehat{\mathcal{R}}/\mathcal{R}\xrightarrow{\jmath_v} \mathbb{F}_{p^2}$ is necessarily also a field isomorphism differing from the original one by post-composition with the Frobenious. 

\begin{proof}
Let $j$ be a uniformizer in $\mathcal{R}$. The field norm on 
$\mathcal{R}/j\mathcal{R}\simeq \mathbb{F}_{p^2}$ coincides with the reduction $\bmod p$ of the reduced norm in $\mathcal{R}$. This can be seen by taking a subfield $E\subset B_p$, such that $E$ is an unramified quadratic extension of $F$ that splits $B_p$, then the restriction of the reduced norm to $E$ coincides with the field norm on $E$ and $\mathcal{R}/j\mathcal{R}=\mathcal{O}_E/p\mathcal{O}_E\simeq \mathbb{F}_p^2$. 

Observe that $\widehat{\mathcal{R}}= j^{-1} \mathcal{R}$ and
$\widehat{\mathcal{R}}/\mathcal{R}\simeq \mathcal{R}/j\mathcal{R}\simeq \mathbb{F}_{p^2}$. The last isomorphism is a field isomorphism and thus it commutes with taking norms. The first isomorphism is via the map $x\mapsto j x$ and $\Nr (j x)=-p\Nr x$. This establishes the claimed formula for norms.
\end{proof}

\begin{prop}\label{prop:local-orbit-ramified}
Assume $B_p$ is ramified and let $\mathcal{R}\subset B_p$ be the unique maximal order. Then,
\begin{equation*}
\rho(\mathbf{SL}_2(\mathbb{Z}_p)).\mathds{1}_{\mathcal{R}}
=\left\{-p^{-1}\psi\left(\frac{\Nr x}{t}\right)\mathds{1}_{\widehat{\mathcal{R}}} \mid 0<t<p \right\}
\cup \left\{\mathds{1}_{\mathcal{R}},-p^{-1}\mathds{1}_{\widehat{\mathcal{R}}} \right\}\;.
\end{equation*}
Moreover, each of the functions above corresponds to a single coset in $\faktor{\mathbf{SL}_2(\mathbb{Z}_p)}{U_0(p)}$ .
\end{prop}
\begin{remark}
Because $\Nr x \in p^{-1}\mathbb{Z}_p$ and $\psi$ is unramified, we can rewrite the result above as
\begin{equation*}
\rho(\mathbf{SL}_2(\mathbb{Z}_p)).\mathds{1}_{\mathcal{R}}
=\left\{\mathds{1}_{\mathcal{R}} \right\} \cup \left\{-p^{-1}\psi\left(t \Nr x\right)\mathds{1}_{\widehat{\mathcal{R}}} \mid t\in\cyclic{p} \right\}
\;.
\end{equation*}
\end{remark}
\begin{proof}
As in the previous lemma, we put coordinates on $B_p$ corresponding to the cyclic algebra $\bigl(\frac{a,p}{\mathbb{Q}_p}\bigr)$ where $E=F(\sqrt{a})/F$ is the unique unramified quadratic extension. In these coordinates, we can write
\begin{align*}
\mathcal{R}&=\mathcal{O}_E+j\mathcal{O}_E\;, &
\widehat{\mathcal{R}}&=\mathcal{O}_E+j^{-1}\mathcal{O}_E\;,
\end{align*}
and make the map $\jmath_v$ explicit:
\begin{equation*}
\jmath_v(a+j^{-1}b)=b \mod p\mathcal{O}_E\;.
\end{equation*}
The map $\jmath_v$ is an isomorphism of abelian groups $\widehat{\mathcal{R}}/\mathcal{R}\to\mathbb{F}_{p^2}$ and $-p \Nr x \mod p\mathbb{Z}_p \equiv \Nr \jmath_v(x)$ for all
$x\in\widehat{\mathcal{R}}$, where the norm on the right hand side is the field norm $\mathbb{F}_{p^2}\to\mathbb{F}_p$.

For each $\alpha\in \mathbb{F}_{p^2}$, fix a representative $x_\alpha\in \jmath_v^{-1}(\alpha)$. Then,
\begin{equation}\label{eq:Rdual-ram-decomp}
\widehat{\mathcal{R}}=\bigsqcup_{\alpha\in \mathbb{F}_{p^2}} \left(x_\alpha+\mathcal{R}\right)\;.
\end{equation}
Explicitly, for each $\alpha\in\mathbb{F}_{p^2}$, we take $x_\alpha=j^{-1}\check{\alpha}$, where $\check{\alpha}\in\mathcal{O}_E$ satisfies $\check{\alpha} \bmod p\mathcal{O}_E=\alpha$.
The duality between $\mathcal{R}$ and $\widehat{\mathcal{R}}$ implies that $\Nr(x_\alpha+\mathcal{R})\subset \Nr x_\alpha + \mathbb{Z}_p$.

Recall that $\mathds{1}_{\mathcal{R}}$ is $\rho(U_0(p))$-invariant. Hence, we need only calculate the action of each representative of  $\mathbf{SL}_2(\mathbb{Z}_p)/U_0(p)$ on $\mathds{1}_{\mathcal{R}}$.
A set of representatives is given by $w$, $\sm 1 & 0 \\ t & 1\esm$, $0\leq t<p$. The action of $\rho(w)$ is the Fourier transform (multiplied by $\gamma=-1$) and we have already seen in the previous proof that
\begin{equation*}
\rho(w).\mathds{1}_{\mathcal{R}}=-p^{-1}\mathds{1}_{\widehat{\mathcal{R}}}\;.
\end{equation*}
Write $x=a+j^{-1}b$ for $x\in \widehat{\mathcal{R}}$.
Now, we compute the action of $\sm 1 & 0 \\ t & 1\esm=-w\sm 1 & -t \\ 0 & 1 \esm w$ using\footnote{Note that $\rho(-1).M(x)=M(-x)$ and that our function is symmetric, so $-1$ acts trivially.} \eqref{eq:Rdual-ram-decomp}
\begin{align*}
\left(\rho\left(\begin{pmatrix}1 & 0 \\ t & 1\end{pmatrix}\right).\mathds{1}_{\mathcal{R}}\right)(x)&=-p^{-1}\rho(w).\left(\psi(-t\Nr y)\mathds{1}_{\widehat{\mathcal{R}}}(y)\right)(x)=-p^{-1}\rho(w).\left(\sum_{\alpha\in\mathbb{F}_{p^2}}
\psi(-t\Nr x_\alpha) \mathds{1}_{x_\alpha+\mathcal{R}}\right)(x)\\
&=p^{-2}\mathds{1}_{\widehat{\mathcal{R}}}(x) \sum_{\alpha\in\mathbb{F}_{p^2}}
\psi(-t\Nr x_\alpha+\langle x, x_\alpha \rangle)\;.
\end{align*}
If $t\neq 0$ then the sum above becomes
\begin{align*}
\sum_{\alpha\in\mathbb{F}_{p^2}}
\psi\left(\frac{\Nr x- \Nr(x-tx_\alpha)}{t} \right)&=\psi\left(\frac{\Nr x}{t}\right)\sum_{\alpha\in\mathbb{F}_{p^2}}\psi\left(\frac{\Nr(b-t\alpha)}{tp}\right)
=\psi\left(\frac{\Nr x}{t}\right)\sum_{\alpha_0\in\mathbb{F}_{p^2}}\psi\left(\frac{\Nr \alpha_0}{tp}\right)\\
&=\psi\left(\frac{\Nr x}{t}\right)\left((p+1)\sum_{\beta\in\mathbb{F}_p^\times} \psi\left(\frac{\beta}{p}\right)+1\right)
=-p\psi\left(\frac{\Nr x}{t}\right)\;.
\end{align*}
We have used the fact that every element of $\mathbb{F}_p^\times$ is the norm of exactly $p+1$ elements in $\mathbb{F}_{p^2}^\times$. This establishes the claim for $0<t<p$. For $t=0$, the sum becomes
\begin{equation*}
\sum_{\alpha\in\mathbb{F}_{p^2}}
\psi\left(-\frac{\Tr(\alpha b)}{p}\right)=
\begin{cases}
p^2 & b\equiv 0 \bmod p\mathcal{O}_E\\
0 & \textrm{otherwise}
\end{cases}\;.
\end{equation*}
And $\mathds{1}_{\widehat{\mathcal{R}}}(x)\delta_{b\equiv 0 \bmod p\mathcal{O}_E}(x)=\mathds{1}_{\mathcal{R}}(x)$. Of course, the case of $t=0$ is actually trivial to compute because it corresponds to the identity representative.
\end{proof}

\begin{lem}\label{lem:schwartz-bruhat-stabilizer}
Let $M\colon B_p\to\mathbb{C}$ be a Schwartz--Bruhat function. Then, there is an open subgroup $U_M<\mathbf{SL}_2(\mathbb{Z}_p)$ such that $\rho(U_M).M=M$. In particular, $\rho\left(\mathbf{SL}_2(\mathbb{Z}_p)\right).M$ is a finite set.
\end{lem}
\begin{proof}
Fix a maximal order $\mathcal{R}\subset B_p$. We first claim that for every Schwartz--Bruhat function $M_0\colon B_p\to\mathbb{C}$ there is some diagonal matrix $a\in\mathbf{SL}_2(\mathbb{Q}_p)$ such that $\rho(a).M_0$ is a linear combination of translates of $\mathds{1}_{\mathcal{R}}$. Equivalently $\rho(a).M_0(x+\mathcal{R})=\rho(a).M_0(x)$. Because $M_0$ is Schwartz--Bruhat, there is a small neighborhood of the origin $\mathcal{V}\subset B_p$, such that $M_0(x+\mathcal{V})=M_0(x)$. There is $k\geq 1$ such that $p^k \mathcal{R}\subset \mathcal{V}$. The function $x\mapsto M_0(p^{k}x)$ is invariant under translations by $\mathcal{R}$.  Set $a=\diag(p^k,p^{-k})$, then $\rho(a).M_0$ is a linear combination of translates of $\mathds{1}_\mathcal{R}$ as claimed.

Fix $b\in B_p$ and consider the group $\mathcal{A}(b)=\left\{\diag(u, u^{-1})\colon u\in\mathbb{Z}_p^\times \textrm{ and } bu-b\in\mathcal{R} \right\}$. Then $\mathcal{A}(b)$ is an open subgroup of the diagonal group in $\mathbf{SL}_2(\mathbb{Z}_p)$ and $\rho(\mathcal{A}(b)).\mathds{1}_{b+\mathcal{R}}=\mathds{1}_{b+\mathcal{R}}$. Taking a finite intersection of such subgroups we  find an open subgroup  $\mathcal{A}_0$ of the diagonal group of $\mathbf{SL}_2(\mathbb{Z}_p)$, such that $\rho(\mathcal{A}_0 a).M=\rho(a).M$. Hence, $M$ is invariant under $\rho(\mathcal{A}_0)$.

In a similar fashion, let $k_b\geq 0$ such that $p^{k_b} b\in\mathcal{R}$ and define $\mathcal{N}(b)=\begin{pmatrix}
1 & p^{k_b} \mathbb{Z}_p \\ 0 & 1
\end{pmatrix}$. Then, $\mathcal{N}(b)$ is an open subgroup of the upper-triangular unipotent group of $\mathbf{SL}_2(\mathbb{Z}_p)$ and $\rho(\mathcal{N}_b).\mathds{1}_{b+\mathcal{R}}=\mathds{1}_{b+\mathcal{R}}$.  Taking a finite intersection of such subgroups we can find an open subgroup $\mathcal{N}_1'$ of the integral upper-triangular unipotent subgroup, such that $\rho(\mathcal{N}_1'a).M=\rho(a).M$. Set $\mathcal{N}_1=a^{-1} \mathcal{N}_1' a\cap \mathbf{SL}_2(\mathbb{Z}_p)$. Then, $\rho(\mathcal{N}_1).M=M$ and $\mathcal{N}_1$ is an open subgroup of the upper unipotent integral group. Replacing $M$ by $\rho(w).M$, we can find $\mathcal{N}_2$ such that $\rho(w^{-1} \mathcal{N_2} w).M=M$, and $w^{-1}\mathcal{N}_2 w$ is an open subgroup of the lower-triangular integral unipotent group.

Set now $U_M$ to be the group generated by $\mathcal{A}_0,\mathcal{N}_1,w^{-1}\mathcal{N}_2w$. Then, $U_M$ is an open subgroup of $\mathbf{SL}_2(\mathbb{Z}_p)$ and satisfies the requirements of the claim.
\end{proof}

Assume now $B_p\simeq \operatorname{Mat}_{2\times 2}(\mathbb{Q}_p)$ is split. Maximal orders in $\operatorname{Mat}_{2\times 2}(\mathbb{Q}_p)$ are endomorphism rings of lattices in $\mathbb{Q}_p^2$ and they are in one-to-one correspondence with the vertices of the Bruhat--Tits tree of $\mathbf{SL}_2(\mathbb{Q}_p)$.
\begin{defi}
An Eichler order in $B_p$ of level $p^n$ is an intersection of two maximal orders corresponding to two vertices in the Bruhat--Tits tree with distance $n$ between them.
\label{def:local-Eichler-level}
\end{defi}

\begin{lem}\label{lem:non-arch-Weil-Eichler}
Let $\mathcal{R}\subset B_p$ be an Eichler order of level $p^n$. Then, $\rho\left(U_0(p^n)\right).\mathds{1}_{\mathcal{R}}=\mathds{1}_{\mathcal{R}}$.
\end{lem}
\begin{proof}
The action of $B_p^\times$ on the vertices of the Bruhat--Tits tree is transitive on pairs of vertices of a fixed distance\footnote{This follows from the facts that the action of $\mathbf{SL}_2(\mathbb{Q}_p)$ is strongly transitive, i.e.\ it is transitive on pairs $(\mathcal{C},\mathcal{A})$ where $\mathcal{C}$ is a chamber in the apartment $\mathcal{A}$, and that $\mathbf{PGL}_2(\mathbb{Q}_p)$ has an element which inverts the orientation of a single chamber.}, thus it acts transitively by conjugation on the set of Eichler orders of a fixed level $p^n$. Because the conjugation action commutes with the Weil representation, it is enough to consider a single Eichler order of the form
\begin{equation*}
\mathcal{R}=\begin{pmatrix}
\mathbb{Z}_p & \mathbb{Z}_p \\ p^n \mathbb{Z}_p & \mathbb{Z}_p
\end{pmatrix}\;.
\end{equation*}
Indeed, invariance of $\mathds{1}_{\mathcal{R}}$ under upper-triangular integral matrices is immediate and it is enough to check invariance under the element $\sm
1 & 0 \\ p^{n} & 1
\esm =-w \sm
1 & -p^{n}\\ 0 &  1
\esm  w$. Equivalently, we need to show $\rho \left( \sm
1 & -p^{n}\\ 0 &  1
\esm \right).\Fr\mathds{1}_{\mathcal{R}}=\Fr\mathds{1}_{\mathcal{R}}$. We can compute the Fourier transform of $\mathds{1}_{\mathcal{R}}$ explicitly and arrive at
\begin{equation*}
\Fr \mathds{1}_{\begin{pmatrix}
\mathbb{Z}_p & \mathbb{Z}_p \\ p^n \mathbb{Z}_p & \mathbb{Z}_p
\end{pmatrix}}=p^{-n}
\mathds{1}_{\begin{pmatrix}
\mathbb{Z}_p & p^{-n} \mathbb{Z}_p \\ \mathbb{Z}_p & \mathbb{Z}_p
\end{pmatrix}}=p^{-n}\mathds{1}_{\widehat{\mathcal{R}}}\;.
\end{equation*}
In particular, for all $x\in \supp \Fr \mathds{1}_{\mathcal{R}}$, we have $\Nr x=\det x\in p^{-n}\mathbb{Z}_p$ and
$\left(\rho\left(\sm
1 & -p^{n} \\ 0 & 1
\esm\right). \Fr \mathds{1}_{\mathcal{R}}\right)(x)=\psi_v(-p^{n}\Nr x) \left(\Fr \mathds{1}_{\mathcal{R}}\right)(x)=\left(\Fr \mathds{1}_{\mathcal{R}}\right)(x)$ as necessary.
\end{proof}

\begin{lem}\label{lem:jmap-local}
Let $\mathcal{R}=\mathcal{R}_1\cap\mathcal{R}_2$ be an Eichler order of level $p^n$, where $\mathcal{R}_1$ and $\mathcal{R}_2$ are maximal orders.
There is an isomorphism of additive abelian groups $\jmath_v\colon\faktor{\widehat{\mathcal{R}}}{\mathcal{R}}\to \left(\cyclic{p^n}\right)^{2}$ such that
\begin{align*}
\jmath_v(\mathcal{R}_1)&=\cyclic{p^n}\times 0\;,\\
\jmath_v(\mathcal{R}_2)&=0\times\cyclic{p^n}\;,\\
\forall x\in\widehat{\mathcal{R}}&\colon -p^n\Nr (x) \equiv \jmath_v(x)_1\cdot \jmath_v(x)_2 \mod p^n\;.
\end{align*}

Moreover, the isomorphism $\jmath_v$ is unique up to post-composition with the map $(b,c)\mapsto (bu,cu^{-1})$ for $u\in\cyclic{p^n}^\times$.
\end{lem}
Notice that the isomorphism $\jmath_v$ depends not only on $\mathcal{R}$ but on an ordered choice of $\mathcal{R}_1$ and $\mathcal{R}_2$.

\begin{proof}
Because all local Eichler orders of fixed level are conjugate, it is enough to verify the claim for
\begin{align*}
\mathcal{R}_1&=\begin{pmatrix}
\mathbb{Z}_p & \mathbb{Z}_p \\  \mathbb{Z}_p & \mathbb{Z}_p
\end{pmatrix}\;, &
\mathcal{R}_2&=\begin{pmatrix}
\mathbb{Z}_p & p^{-n} \mathbb{Z}_p \\  p^n \mathbb{Z}_p & \mathbb{Z}_p
\end{pmatrix}\;.
\end{align*}
In this case, we have
\begin{align*}
\mathcal{R}&=\begin{pmatrix}
\mathbb{Z}_p & \mathbb{Z}_p \\ p^n \mathbb{Z}_p & \mathbb{Z}_p
\end{pmatrix}\;, &
\widehat{\mathcal{R}}&=\begin{pmatrix}
\mathbb{Z}_p & p^{-n} \mathbb{Z}_p \\  \mathbb{Z}_p & \mathbb{Z}_p
\end{pmatrix}\;.
\end{align*}
We define the abelian homomorphism $\jmath_v\colon\widehat{\mathcal{R}}\to \left(\cyclic{p^n}\right)^{2}$ explicitly as
\begin{equation*}
\begin{pmatrix}
a & b/p^n \\ c & d
\end{pmatrix}\mapsto (c,b) \bmod p^n\mathbb{Z}_p\;.
\end{equation*}
A direct computation shows that this homomorphism has kernel $\mathcal{R}$ and that it satisfies the claimed properties.

This isomorphism is unique up to post-composition with a linear automorphism of the first and second coordinate of $\left(\cyclic{p^n}\right)^{2}$, i.e.\ a transformation of the form $(b,c)\mapsto (b u_1, c u_2)$ for $u_1,u_2\in(\cyclic{p^n})^{\times}$. The requirement that the quadratic form $-p^n\Nr$ descends to the product form $(b,c)\mapsto b\cdot c$ forces $u_2=u_1^{-1}$.
\end{proof}

\begin{remark}
The previous lemma implies that given two maximal orders $\mathcal{R}_1$, $\mathcal{R}_2$ the map
$x\mapsto (\ord_p \jmath_v(x)_1,\ord_p \jmath_v(x)_2)$ from $\widehat{\mathcal{R}}$ to $\{0,1,\ldots,n\}^{2}$ is uniquely-defined.
\end{remark}

\begin{defi}
Let $\mathcal{R}\subset B_p$ be an Eichler order of level $p^n$. For every $p^k\mid p^n$ define the lattice
\begin{equation*}
\widehat{\mathcal{R}}^{(p^k)}=\left\{x\in\widehat{\mathcal{R}} \mid \jmath_v(x)\equiv (0,0) \mod p^k
\right\}\;.
\end{equation*}
The definition of $\widehat{\mathcal{R}}^{(p^k)}$ does not depend on the choice of $\jmath_v$.
Note that $\widehat{\mathcal{R}}^{(1)}=\widehat{\mathcal{R}}$ and $\widehat{\mathcal{R}}^{(p^n)}=\mathcal{R}$.

Moreover, for $x\in\widehat{\mathcal{R}}^{(p^k)}$ define\footnote{This definition does not depend on the choice of $\jmath_v$.} $\nu_{p^k}(x)\coloneqq -p^{-k}\jmath_v(x)_1 \cdot p^{-k}\jmath_v(x)_2\in \cyclic{p^{n-k}}$. Notice that $\nu_{p^0}(x)\equiv p^{n} \Nr x \mod p^{n}$.

\end{defi}
\begin{prop}\label{prop:local-orbit-Eichler}
Let $\mathcal{R}\subset B_p$ be an Eichler order of level $p^n$.Then,
\begin{align*}
\rho\left(\mathbf{SL}_2(\mathbb{Z}_p)\right).\mathds{1}_{\mathcal{R}}=
&\left\{\mathds{1}_{\widehat{\mathcal{R}}}(x)\cdot p^{-n}\psi_v(-p t_0 \Nr x) \mid 0 < t_0 \leq p^{n-1} \right\}\\
\cup &\left\{ \mathds{1}_{\widehat{\mathcal{R}}^{\left(p^{\ord_p t}\right)}}(x)\cdot p^{-(n-\ord_p t)} \psi_v\left(\frac{\nu_{p^{\ord_p t}}(x)}{p^{n-2\ord_p t}t}\right) \mid 0< t \leq p^n
\right\}\;.
\end{align*}
Moreover, each of the functions above corresponds to a single coset of $\faktor{\mathbf{SL}_2(\mathbb{Z}_p)}{U_0(p^n)}$.
For $t=p^n$ above, the phase is $1$, hence the representative for $t=p^n$ is simply $\mathds{1}_{\mathcal{R}}$.
\end{prop}
\begin{remark}
Because $\psi$ is unramified, we can rewrite the result above as
\begin{align*}
\rho\left(\mathbf{SL}_2(\mathbb{Z}_p)\right).\mathds{1}_{\mathcal{R}}= \bigcup_{0< k \leq n}
&\left\{ \mathds{1}_{\widehat{\mathcal{R}}^{(p^k)}}(x)\cdot p^{-(n-k)} \psi_v\left(\frac{u \cdot \nu_{p^k}(x)}{p^{n-k}}\right) \mid u\in\left(\cyclic{p^{n-k}}\right)^\times
\right\}\\
\cup
&\left\{ \mathds{1}_{\widehat{\mathcal{R}}}(x)\cdot p^{-n} \psi_v\left(t \Nr x\right) \mid t\in\cyclic{p^n}\right\}\;.
\end{align*}
The set on the second line is analogous to the $k=0$ case missing in the first line, but requires $t$ to traverse the whole congruence subgroup, not just the units.
\end{remark}
\begin{proof}
Again, as all Eichler orders are conjugate, the claim reduces to the case of
\begin{align*}
\mathcal{R}&=\begin{pmatrix}
\mathbb{Z}_p & \mathbb{Z}_p \\ p^n \mathbb{Z}_p & \mathbb{Z}_p
\end{pmatrix}\;, &
\widehat{\mathcal{R}}&=\begin{pmatrix}
\mathbb{Z}_p & p^{-n} \mathbb{Z}_p \\  \mathbb{Z}_p & \mathbb{Z}_p
\end{pmatrix}
\end{align*}
and
\begin{equation*}
\jmath_v\left(\begin{pmatrix}
a & b/p^n \\ c & d
\end{pmatrix}\right)
=(c,b) \mod p^n\mathbb{Z}_p\;.
\end{equation*}

Because of Lemma \ref{lem:non-arch-Weil-Eichler}, it is enough to compute $\rho(s).\mathds{1}_{\mathcal{R}}$ for each of the representatives of $\faktor{\mathbf{SL}_2(\mathbb{Z}_p)}{U_0(p^n)}$. To find these representatives, we first write representatives for $\faktor{\mathbf{SL}_2(\mathbb{Z}_p)}{U_0(p)}$
\begin{equation*}
\mathbf{SL}_2(\mathbb{Z}_p)=w U_0(p)\sqcup \bigsqcup_{0< t \leq p} \begin{pmatrix} 1 & 0 \\ t & 1 \end{pmatrix} U_0(p)\;.
\end{equation*}
This decomposition follows from the fact that $U_0(p)$ is the stabilizer of an edge in the apartment of the diagonal torus in the Bruhat--Tits tree of $\mathbf{SL}_2(\mathbb{Q}_p)$ and the representatives above permute the $p+1$ neighbors of the vertex stabilized by $\mathbf{SL}_2(\mathbb{Z}_p)$.

Next, we find representatives for $\faktor{U_0(p)}{U_0(p^n)}$ using their definition as congruence subgroups
\begin{equation*}
U_0(p)=\bigsqcup_{0 < t_0 \leq p^{n-1}} \begin{pmatrix} 1 & 0 \\ p t_0 & 1 \end{pmatrix} U_0(p^n)\;.
\end{equation*}

By combining, we arrive at
\begin{equation*}
\mathbf{SL}_2(\mathbb{Z}_p)=w \bigsqcup_{0 < t_0 \leq p^{n-1}} \begin{pmatrix} 1 & 0 \\ p t_0 & 1 \end{pmatrix} U_0(p^n) \sqcup
\bigsqcup_{0 < t \leq p^n} \begin{pmatrix} 1 & 0 \\ t & 1 \end{pmatrix} U_0(p^n)\;.
\end{equation*}

We now compute explicitly the action of all representatives.
We need to decompose $\widehat{\mathcal{R}}$ into $\mathcal{R}$-cosets
\begin{equation*}
\widehat{\mathcal{R}}=\bigsqcup_{0\leq \alpha,\beta<p^n} \begin{pmatrix}
0 & \alpha/p^n\\ \beta & 0
\end{pmatrix} +\mathcal{R}\;.
\end{equation*}
To simplify notations, we denote $x_{\alpha,\beta}\coloneqq \sm
0 & \alpha/p^n\\ \beta & 0
\esm$, with $\jmath_v(x_{\alpha,\beta})=(\beta,\alpha)$.
The duality between $\widehat{\mathcal{R}}$ and $\mathcal{R}$ implies $\Nr(x_{\alpha,\beta}+\mathcal{R})=\Nr x_{\alpha,\beta} + \mathbb{Z}_p=-\alpha\beta/p^n+\mathbb{Z}_p$.
Write $\sm 1 & 0 \\ t & 1 \esm=-w \sm 1 & -t \\ 0 & 1 \esm w$ and $x=\sm
a & b/p^n \\ c & d
\esm$. Then,
\begin{align*}
\rho\left(\begin{pmatrix} 1 & 0 \\ t & 1 \end{pmatrix}\right).\mathds{1}_{\mathcal{R}} (x)
&=p^{-n}\rho(w) \rho\left(\begin{pmatrix} 1 & -t \\ 0 & 1 \end{pmatrix}\right).  \mathds{1}_{\widehat{\mathcal{R}}} (x)
=p^{-n}\rho(w).\left( \psi_v(-t \Nr x)  \mathds{1}_{\widehat{\mathcal{R}}}(x)\right)\\
&= p^{-n}\rho(w).\left(\sum_{0\leq \alpha,\beta < p^n} \psi_v(t \alpha\beta/p^n) \mathds{1}_{x_{\alpha,\beta}+\mathcal{R}}(x) \right)\\
&= \mathds{1}_{\widehat{\mathcal{R}}}(x) \cdot p^{-2n} \sum_{0\leq \alpha,\beta < p^n} \psi_v(t \alpha\beta/p^n+\langle x_{\alpha,\beta},x \rangle)\\
&= \mathds{1}_{\widehat{\mathcal{R}}}(x) \cdot p^{-2n} \sum_{0\leq \alpha,\beta < p^n} \psi_v\left(\frac{-\alpha c- \beta b +t \alpha\beta}{p^n}\right)\;.
\end{align*}
Let $k=\ord_p t$. Then, summing first over $\alpha$ we deduce $p^k \mid c$ and summing first over $\beta$ we see that $p^k \mid b$. Using this input, we can sum first over $\alpha$ and then over $\beta$ to arrive at
\begin{align*}
\rho\left(\begin{pmatrix} 1 & 0 \\ t & 1 \end{pmatrix}\right).\mathds{1}_{\mathcal{R}} (x)
&=\mathds{1}_{\widehat{\mathcal{R}}}(x) p^{-(n-k)}
\begin{cases}
\psi_v\left(-\frac{(c/p^k)(b/p^k)}{(t/p^k)p^{n-k}}\right)=\psi_v\left(\frac{\nu_{p^k}(x)}{tp^{n-2k}}\right),
 & \jmath_v(x)\equiv (0,0) \bmod p^k, \\
 0, & \textrm{otherwise.}
\end{cases}
\end{align*}

We need only deal now with representatives of the form $w\sm 1 & 0 \\ p t_0 & 1\esm=\sm 1 & -p t_0 \\ 0 & 1 \esm w$ which are easier to compute
\begin{equation*}
\rho\left(w\begin{pmatrix} 1 & 0 \\ p t_0 & 1\end{pmatrix}\right).\mathds{1}_{\mathcal{R}} (x)=p^{-n}\psi_v(-p t_0 \Nr x) \mathds{1}_{\widehat{\mathcal{R}}}(x)\;.
\end{equation*}
\end{proof}

\subsection{Local Uniformity}
As a preparation for the global theory, we will need the following proposition that uniformly controls the decay of test functions for the Weil representation.
\begin{prop}\label{prop:Weil-local-uniformity}
Let $M\in \Omega_v$, $s\in\mathbf{SL}_2(F)$ and $L\in \operatorname{O}(B_v,\Nr)$. If $v=\infty$, then there is $\delta>0$ such that the inequality
\begin{equation*}
|(\rho(s)M)(Lx)| \ll (1+\|x\|)^{-4-\delta}
\end{equation*}
holds uniformly on compact sets in $\mathbf{SL}_2(F)\times \operatorname{O}(B_v,\Nr)$. If $v<\infty$, then for every compact subset of $\mathcal{K}\subset \mathbf{SL}_2(F)\times \operatorname{O}(B_v,\Nr)$ there is a compact subset $\mathcal{C}\subset B_v$ such that
\begin{equation*}
\forall (s,L)\in \mathcal{K} \colon |(\rho(s).M)(Lx)|\ll_{\mathcal{K}} \mathds{1}_{\mathcal{C}}\;.
\end{equation*}
\end{prop}
\begin{proof}
The claim for $v=\infty$ follows immediately from Corollary \ref{cor:Weil-infinity-action}. Fix now $v<\infty$.
Because $\mathcal{K}$ can be covered by finitely many product sets, we assume without loss of generality that $\mathcal{K}=\mathcal{K}_0 \times \mathcal{K}_1$ is a product set. Notice that it is enough to show that there is some $\mathcal{C}_0\subset B_v$ such that $|\rho(s).{M}|\ll \mathds{1}_{\mathcal{C}_0}$ for $s\in\mathcal{K}_0$. In particular, the compact set $\mathcal{C}=\bigcup_{L\in\mathcal{K}_1} L^{-1}\mathcal{C}_0$ will satisfy the claimed properties. Using the Iwasawa decomposition, we can find a compact subset $C_P$ of the group of upper triangular matrices such that $\mathcal{K}_0\subset C_P\mathbf{SL}_2(\mathbb{Z}_p)$. Recall from Lemma \ref{lem:schwartz-bruhat-stabilizer} that the $\rho\left(\mathbf{SL}_2(\mathbb{Z}_p)\right)$-orbit of $M$ is finite and a uniform bound on $\rho(\mathcal{K_0}).M$ will follow from a uniform bound on $\rho(C_P).M'$ for every $M'$ in  $\rho\left(\mathbf{SL}_2(\mathbb{Z}_p)\right).M$. A uniform bound on $|\rho(C_P).M'|$ follows immediately from the formul{\ae} for the Weil action of diagonal and upper unipotent matrices.
\end{proof}

\subsection{Global Weil Representation and Theta Series}
Fix an additive character $\psi\colon \mathbb{A} \slash \mathbb{Q}\to\mathbb{C}$ such that $\psi=\prod_{v} \psi_v$ and $\psi_v$ is unramified for all finite $v$. This is possible for the ad\`ele ring of $\mathbb{Q}$ and the standard character with $\psi_\infty(a)=\exp(-2\pi i a)$ will do. We consider henceforth always the local Weil representations on $\Omega_v$ to be with respect to $\psi_v$.

Set
\begin{equation*}
\Omega\coloneqq {\bigotimes}' \Omega_v=\Span_{\mathbb{C}} \left\{\prod_v M_v \mid M_v\in\Omega_v,\; \forall \textrm{ a.e. } v\colon M_v= \mathds{1}_{R_v} \right\}\;.
\end{equation*}
A pure tensor $M=\prod_v M_v \in \Omega$ is called a \emph{standard test function}.
The Weil representations for each $\Omega_v$ define in the standard way a representation of $\mathbf{SL}_2(\mathbb{A})$ on $\Omega$. To see that this action is well-defined we need to check that for a.e.\ $v$ we have $\rho\left(\mathbf{SL}_2(\mathbb{Z}_v)\right).\mathds{1}_{R_v}=\mathds{1}_{R_v}$, and this follows from Lemma \ref{lem:non-arch-Weil-maximal-order}. The complex vector space $\Omega$ also carries actions of $\mathbf{G}(\mathbb{A})$ by left and right multiplication because for every $v<\infty$ the function $\mathds{1}_{R_v}$ is invariant under left and right multiplication  by elements of $K_v$.

\begin{defi}\label{defi:theta-series}
For $M\in \Omega$ define the theta kernel $\Theta_M\colon\mathbf{G}(\mathbb{A})\times \mathbf{G}(\mathbb{A})\times \mathbf{SL}_2(\mathbb{A})\to\mathbb{C}$ by
\begin{equation*}
\Theta_M(l,r;s)\coloneqq \sum_{\xi\in B} \left(\rho(s)M\right)(l^{-1} \xi r)\;.
\end{equation*}
\end{defi}
The uniform decay from Proposition \ref{prop:Weil-local-uniformity} is sufficient for the theta series to converge absolutely and uniformly on compact sets in $\mathbf{G}(\mathbb{A})\times \mathbf{G}(\mathbb{A})\times \mathbf{SL}_2(\mathbb{A})$. In particular, the theta series is a well-defined continuous function on its domain.

The theta kernel is obviously $\mathbf{G}(\mathbb{Q})$-invariant on the left in the first two coordinates by virtue of its definition. Less obvious, yet well-known, is that it is also $\mathbf{SL}_2(\mathbb{Q})$ left invariant in the third coordinate. A simple way to verify this is by first showing invariance under upper-triangular matrices by an elementary calculation and then establishing invariance under the involution $\sm
0 & 1\\ -1 & 0
\esm$ using the Poisson summation formula. The decay conditions we have imposed on functions in $\Omega_\infty$ are sufficient for the Poisson summation formula to hold \cite[p. 252, Cor. 2.6]{Stein-Weiss-Fourier}.

We now recall the Fourier--Whittaker decomposition of a continuous function $\varphi\colon [\mathbf{SL}_2(\mathbb{A})]\to\mathbb{C}$. For any $\alpha \in \mathbb{Q}$, define the Whittaker function
\begin{equation*}
W_{\varphi}(g,\alpha)=\int_{[\mathbf{N}(\mathbb{A})]} \varphi\left(\begin{pmatrix}
1 & n \\ 0 & 1
\end{pmatrix}g\right) \psi(-\alpha n) \dif n\;.
\end{equation*}
We have the following standard properties of the Whittaker function
\begin{align*}
\forall n\in \mathbb{A}\colon W_{\varphi}\left(\begin{pmatrix}
1 & n \\ 0 & 1
\end{pmatrix}g,\alpha
\right)&=\psi(\alpha n) W_{\varphi}(g,\alpha)\;,\\
\forall \lambda\in\mathbb{Q}^\times\colon
W_{\varphi}\left(\begin{pmatrix}
\lambda & 0 \\ 0 & \lambda^{-1}
\end{pmatrix}g,\alpha
\right)&=W_{\varphi}(g,\lambda^2\alpha)\;.
\end{align*}
Because our function $\varphi$ is defined on $[\mathbf{SL}_2(\mathbb{A})]$ and not $[\mathbf{PGL}_2(\mathbb{A})]$, we can not reduce the dependence on $\alpha$ to the two cases $0$ and $1$. Pontryagin duality for the compact abelian group $[\mathbf{N}(\mathbb{A})]$ implies that the following equality
\begin{equation}\label{eq:Fourier-Whittaker-expansion}
\varphi(g)=\sum_{\alpha\in\mathbb{Q}} W_{\varphi}(g,\alpha)
\end{equation}
holds pointwise as long as the right hand side is absolutely convergent\footnote{For a fixed $g\in\mathbf{PGL}_2(\mathbb{A})$ this is the Fourier expansion of the function $n\mapsto \varphi(ng)$ evaluated at $n=e$. The function $n\mapsto\varphi(ng)$ is a continuous function on the compact Abelian group $[\mathbf{N}(\mathbb{A})]$. If the Fourier transform of a continuous functions converges absolutely, then it coincides with the function everywhere.}.

\begin{prop}\label{prop:Fourier-theta}
Fix $M\in\Omega$. Then, the Fourier--Whittaker coefficients of $\Theta_M(l,r;s)$ in the $s$-variable are
\begin{equation*}
W_{\Theta_M(l,r;\bullet)}(s,\alpha)=\sum_{\substack{\xi\in B \\ \Nr \xi=\alpha}} \left(\rho(s)M\right)(l^{-1} \xi r)\;.
\end{equation*}
\end{prop}
Because the theta series in Definition \ref{defi:theta-series} converges absolutely, an immediate corollary is that the Fourier--Whittaker expansion \eqref{eq:Fourier-Whittaker-expansion} holds pointwise for $\Theta_M(l,r;\bullet)$.

\begin{proof}
Because the theta series converges absolutely, we may exchange summation and integration in the definition of $W_{\Theta_M(l,r;\bullet)}$ and write
\begin{align*}
W_{\Theta_M(l,r;\bullet)}(s,\alpha)&=
\sum_{\xi\in B }\int_{[\mathbf{N}(\mathbb{A})]} \left(\rho\left(\begin{pmatrix}
1 & n \\ 0 & 1
\end{pmatrix}s\right)M \right) (l^{-1}\xi r) \psi(-\alpha n) \dif n\\
&=
\sum_{\xi\in B }\int_{[\mathbf{N}(\mathbb{A})]} \left(\rho\left(s\right)M \right) (l^{-1}\xi r) \psi(\Nr \xi\cdot n-\alpha n) \dif n\\
&=\sum_{\substack{\xi\in B \\ \Nr \xi=\alpha}} \left(\rho(s)M\right)(l^{-1} \xi r)\;.
\end{align*}
\end{proof}

\section{Theta Kernels for Eichler Orders}
\label{sec:actionEichler}
\subsection{Weil Action on Eichler Orders}
We first introduce the notation $D_B$ for the (reduced) discriminant of $B$, explicitly
\begin{equation*}
D_B=\prod_{p: B_p \textrm{ is ramified}} p\;.
\end{equation*}

\begin{defi}
An Eichler order $\mathcal{R}\subset B$ is an intersection of two maximal orders $\mathcal{R}_1$ and $\mathcal{R}_2$. The completion $\mathcal{R}_v$ of $\mathcal{R}$ at any finite place $v$ is a local Eichler order in $B_v$. We say that $\mathcal{R}$ is ramified at $v$ if $\mathcal{R}_v$ is non-maximal. If $B$ is ramified at $v$ then $\mathcal{R}_v$ is unramified at $v$ because $B_v$ has a unique maximal order.

For almost all places, the local orders $\mathcal{R}_{1,v}$ and $\mathcal{R}_{2,v}$ coincide\footnote{This happens for any two lattices in a rational vector space.} and $\mathcal{R}_v$ is a maximal order, i.e.\ the level of $\mathcal{R}_v$ at these places is $1$.
We define the level of $\mathcal{R}$ as
\begin{equation*}
\prod_{v<\infty} \textrm{level of }\mathcal{R}_v\;.
\end{equation*}
The reader may recall Definition \ref{def:local-Eichler-level}, where we defined the level of a local Eichler order.

The dual lattice to $\mathcal{R}$ is
\begin{equation*}
\widehat{\mathcal{R}}\coloneqq\left\{x\in B \mid \forall y\in \mathcal{R}\colon \Tr x\tensor[^\iota]{y}{}\in\mathbb{Z} \right\}\;.
\end{equation*}
Dualization commutes with localization, i.e. $(\widehat{\mathcal{R}})_v=\widehat{\mathcal{R}_v}$.
Denote the level of $\mathcal{R}$ by $q\in\mathbb{N}$.
Using the decomposition
\begin{equation*}
\widehat{\mathcal{R}}/\mathcal{R}=\prod_{v<\infty} \widehat{\mathcal{R}}_v/ \mathcal{R}_v
\end{equation*}
and Lemmata \ref{lem:jmap-local}, \ref{lem:jmap-local-ramified}, we see the existence of an isomorphism of abelian groups
\begin{equation*}
\jmath\colon \widehat{\mathcal{R}}/\mathcal{R}\to \left(\cyclic{q}\right)^{2} \times \prod_{p \mid D_B} \mathbb{F}_{p^2}
\;.
\end{equation*}
The map $\jmath$ fibers through the local maps $\jmath_v$ and satisfies the properties inherited from Lemma \ref{lem:jmap-local}:
\begin{align*}
\jmath(\mathcal{R}_1)&= \cyclic{q}\times 0\times \prod_{p \mid D_B} \mathbb{F}_{p^2}\;,\\
\jmath(\mathcal{R}_2)&= 0\times\cyclic{q} \times \prod_{p \mid D_B} \mathbb{F}_{p^2}\;,\\
\forall x\in \widehat{\mathcal{R}} &\colon -q\Nr x \equiv \jmath(x)_1\cdot \jmath(x)_2 \mod q\;,\\
\forall x\in \widehat{\mathcal{R}}\;, p\mid D_B &\colon -p\Nr x \equiv \Nr \jmath(x)_3 \mod p\;.
\end{align*}
\end{defi}

For $m\in\mathbb{N}$ and $x\in \mathbb{F}_{p^2}$ define $x\bmod m=\begin{cases} x & p\mid m \\ 0 & p\nmid m \end{cases}$. Similarly, set $x \bmod m= x \bmod \gcd(m,q)$ for all $x\in \cyclic{q}$. We extend this definition element-wise to a map 
\begin{equation*}
\left(\cyclic{q}\right)^{2} \times \prod_{p \mid D_B} \mathbb{F}_{p^2}\xrightarrow{\bmod m} \left(\cyclic{\gcd(m,q)}\right)^{2} \times \prod_{p\mid D_B} \mathbb{F}_{p^2}\;,
\end{equation*}
\begin{defi}
Let $\mathcal{R}=\mathcal{R}_1\cap\mathcal{R}_2 \subset B$ be an Eichler order of level $q$. For every $m\mid q D_B$ define
\begin{equation*}
\widehat{\mathcal{R}}^{(m)}=\left\{x\in\widehat{\mathcal{R}}\mid \jmath(x)\equiv 0 \mod m \right\}\;.
\end{equation*}
Notice that $\widehat{\mathcal{R}}^{(1)}=\widehat{\mathcal{R}}$ and $\widehat{\mathcal{R}}^{(qD_B)}=\mathcal{R}$ and
\begin{equation*}
\left[\widehat{\mathcal{R}} \colon \widehat{\mathcal{R}}^{(m)}\right]=m^2\;.
\end{equation*}
Moreover, the definition of $\widehat{\mathcal{R}}^{(m)}$ does not depend on the choices involved in the definition of $\jmath$.

We also define for $x\in\widehat{\mathcal{R}}^{(m)}$
\begin{equation*}
\nu_m(x)\coloneqq \prod_{p \mid qD_B} \begin{cases}
p^{1-\ord_p m}\Nr(x) \mod p^{\ord_p D_B - \ord_p m}& p \mid D_B \\
\nu_{p^{\ord_p m}} (x) & p \mid q
\end{cases}\in\left(\cyclic{(qD_B/m)}\right)\;.
\end{equation*}
\end{defi}
The lattices $\widehat{\mathcal{R}}^{(m)}$ will feature prominently in the description of the action of the Weil representation. In classical terms, they will appear in the Fourier expansion of a theta series at different cusps.
We will use the following notation for the completion of an integral lattice at all finite places.
\begin{defi}
If $L\subset B$ is a lattice, then define $\mathds{1}_{L_f}\colon B_f\to\mathbb{C}$ to be $\mathds{1}_{L_f}\coloneqq\prod_p \mathds{1}_{L_p}$, where $L_p\subset B_p$ is the $p$-adic closure of $L$.
\end{defi}

Our goal now is to describe the $\rho\left(\mathbf{SL}_2(\widehat{\mathbb{Z}})\right)$-action on $\mathds{1}_{\mathcal{R}_f}$. The first step is to identify the stabilizer of $\mathds{1}_{\mathcal{R}_f}$.
\begin{defi}
Define the compact-open subgroup $U_{\mathcal{R}}=\prod_{p<\infty}U_{p}<\mathbf{SL}_2(\mathbb{A}_f)$ by
\begin{equation*}
U_p=\begin{cases}
U_0(p) & B \textrm{ is ramified at } p\\
U_0(p^n) & \mathcal{R}_p \textrm{ has level } p^n\\
\mathbf{SL}_2(\mathbb{Z}_p) & \textrm{otherwise}
\end{cases}
\end{equation*}
\end{defi}
From Lemmata \ref{lem:non-arch-Weil-maximal-order}, \ref{lem:non-arch-Weil-Eichler}, we know that $\rho(U_{\mathcal{R}}).\mathds{1}_{\mathcal{R}_f}=\mathds{1}_{\mathcal{R}_f}$.
 
We define the arithmetic  function $\rho(a \mid q D_B)$ as
\begin{equation*}
\rho (a \mid q D_B)\coloneqq
\prod_{p\mid \gcd(q D_B/a, a)}
(1-p^{-1}) \;.
\end{equation*}
Note that $(\log \log (10 q D_B))^{-1} \ll \rho(a \mid q D_B) \leq 1$.

\begin{prop}\label{prop:SL2(hatZ)-action}
Let $\mathcal{R}\subset B$ be an Eichler order. Then,
\begin{equation*}
\rho\left(\mathbf{SL}_2(\widehat{\mathbb{Z}})\right).\mathds{1}_{\mathcal{R}_f}=
\Bigg\{
\frac{a (-1)^{\omega(D_B/\gcd(a,D_B))}}{q D_B}
\psi\left(\frac{ t \cdot \nu_a(x)}{qD_B/a}\right) \cdot \mathds{1}_{\widehat{\mathcal{R}}^{(a)}_f} \,\Big| a \mid q D_B,\,  t\in \cyclic{\frac{q D_B}{a}},\, \gcd(t,a)=1
\Bigg\}\;.
\end{equation*}
Moreover, each function on the right hand side corresponds to a single coset of $\faktor{\mathbf{SL}_2(\widehat{\mathbb{Z}})}{U_{\mathcal{R}}}$.
\end{prop}
\begin{remark}\label{rem:SL2(hatZ)-action-count}
For every $a\mid q D_B$, the characteristic function $\mathds{1}_{\widehat{\mathcal{R}}^{(a)}}$ appears above exactly $\frac{q D_B}{a} \rho(a\mid q D_B)$ times with different phase functions.
\end{remark}
\begin{proof}
This follows from combining the local contributions as calculated in Propositions \ref{prop:local-orbit-ramified}, \ref{prop:local-orbit-Eichler} and the remarks following these propositions.
\end{proof}

\subsection{Theta Series for Eichler Orders}
\label{sec:ThetaEichler}
We fix once and for all an Eichler order $\mathcal{R}=\mathcal{R}_1\cap\mathcal{R}_2\subset B$ of level $q$.
In this section, we unwind the adelic definition of a theta series for the case of Eichler orders into classical terms.

Denote  $K_{\mathcal{R}}=\prod_{v<\infty} \left(\mathbf{G}(\mathbb{Q}_v)\cap \mathcal{R}_v\right)$. Strong approximation for the simply connected group $\mathbf{G}$ implies that the double quotient
\begin{equation*}
\dfaktor{\mathbf{G}(\mathbb{Q})}{\mathbf{G}(\mathbb{A})}{K_{\mathcal{R}}}
\end{equation*}
is a single orbit of $\mathbf{G}(\mathbb{R})$. The stabilizer of the identity double coset in $\mathbf{G}(\mathbb{R})$ is
\begin{equation*}
\Gamma\coloneqq K_{\mathcal{R}}\cap \mathbf{G}(\mathbb{Q})=\left\{x\in\mathcal{R} \mid \Nr x = 1\right\}\;.
\end{equation*}
Hence, there is a canonical quotient map
\begin{equation*}
\pi_{\Gamma}\colon \lfaktor{\mathbf{G(\mathbb{Q})}}{\mathbf{G}(\mathbb{A})}\to \lfaktor{\Gamma}{\mathbf{G}(\mathbb{R})}\;.
\end{equation*}
Each fiber of this map is a torsor for $K_{\mathcal{R}}$. 
The quotient map $\pi_{\Gamma}$ induces a natural isomorphism
\begin{equation*}
\pi_{\Gamma}^*\colon\operatorname{Map}\left(\lfaktor{\Gamma}{\mathbf{G}(\mathbb{R})},\mathbb{C}\right)\to
\operatorname{Map}\left([\mathbf{G}(\mathbb{A})],\mathbb{C}\right)^{K_\mathcal{R}}\;,
\end{equation*}
where the right-hand-side is the set of all $K_{\mathcal{R}}$-invariant maps.

Set $\Lambda=U_{\mathcal{R}}\cap \mathbf{SL}_2(\mathbb{Q})<\mathbf{SL}_2(\mathbb{R})$. Explicitly, $\Lambda=\Gamma_0(q D_B)$ where $D_B$ is the product of the primes ramified in $B$ and $q$ is the level of $\mathcal{R}$. Again, the double quotient
\begin{equation*}
\dfaktor{\mathbf{SL}_2(\mathbb{Q})}{\mathbf{SL}_2(\mathbb{A})}{U_{\mathcal{R}}}
\end{equation*}
is a single orbit of $\mathbf{SL}_2(\mathbb{R})$ and the stabilizer of the identity is $\Gamma$. This induces a quotient map
\begin{equation}\label{eq:SL2-adelic2real}
\pi_{\Lambda}\colon \lfaktor{\mathbf{SL}_2(\mathbb{Q})}{\mathbf{SL}_2(\mathbb{A})}\to \lfaktor{\Lambda}{\mathbf{SL}_2(\mathbb{R})}
\end{equation}
and a natural isomorphism
\begin{equation*}
\pi_{\Lambda}^*\colon \operatorname{Map}\left(\lfaktor{\Lambda}{\mathbf{SL}_2(\mathbb{R})},\mathbb{C}\right)\to
\operatorname{Map}\left([\mathbf{SL}_2(\mathbb{A})],\mathbb{C}\right)^{U_\mathcal{R}}\;.
\end{equation*}

We can now write the adelic theta kernel in classical terms.
\begin{prop}\label{prop:theta-series-adelic-to-classical}
Fix $M=M_\infty \cdot \prod_{v<\infty}M_v \in \Omega$ such that $M_v=\mathds{1}_{\mathcal{R}_v}$ for all finite $v$.
Let $l_\infty, r_\infty\in\mathbf{G}(\mathbb{R})$ and $s_\infty\in\mathbf{SL}_2(\mathbb{R})$. Then,

\begin{equation*}
\Theta_M\left(l_\infty K_{\mathcal{R}},r_\infty K_{\mathcal{R}}; s_\infty U_{\mathcal{R}}\right)=
\sum_{\xi\in \mathcal{R}} (\rho(s_\infty)M_\infty)(l_\infty^{-1}\xi r_\infty)\;.
\end{equation*}
Hence, $\Theta_M$ defines a classical theta kernel on $\lfaktor{\Gamma}{\mathbf{G}(\mathbb{R})}\times \lfaktor{\Gamma}{\mathbf{G}(\mathbb{R})} \times \lfaktor{\Lambda}{\mathbf{SL}_2(\mathbb{R})}$.
\end{prop}
\begin{proof}
This follows from the discussion above, Lemma \ref{lem:non-arch-Weil-Eichler} and the local-to-global principle for lattices that implies
\begin{equation*}
\mathcal{R}=\bigcap_{v<\infty} \mathcal{R}_v\;,
\end{equation*}
where the intersection is taken in $B$.
\end{proof}

\begin{defi}\label{defi:theta-classical}
Fix $M_\infty\in\Omega_\infty$ and define $\vartheta_{M_\infty}\colon \lfaktor{\Gamma}{\mathbf{G}(\mathbb{R})}\times \lfaktor{\Gamma}{\mathbf{G}(\mathbb{R})} \times \lfaktor{\Lambda}{\mathbf{SL}_2(\mathbb{R})}$ by
\begin{equation*}
\vartheta_{M_\infty}(l_\infty,r_\infty; s_\infty)=\Theta_M\left(l_\infty K_{\mathcal{R}},r_\infty K_{\mathcal{R}}; s_\infty U_{\mathcal{R}}\right)\;.
\end{equation*}
\end{defi}

\subsection{The Weil \texorpdfstring{$L^2$}{L2}-norm of the Theta Kernel}
Our method relies heavily on bounding the $L^2$-norm of $\Theta_M(l,r;s)$ in the $s$-variable. We use the Fourier--Whittaker decomposition from Proposition \ref{prop:Fourier-theta} to bound the $L^2$-norm by a second moment count of rational matrices. Unfortunately, the classical representation above is not well adapted to this calculation because of the many cusps of $\lfaktor{\Lambda}{\mathbf{SL}_2(\mathbb{R})}$. Instead, we compute adelically the $L^2$-norm. This is easier to execute because the adelic quotient $[\mathbf{SL}_2(\mathbb{A})]=\lfaktor{\mathbf{SL}_2(\mathbb{Q})}{\mathbf{SL}_2(\mathbb{A})}$ has a single cusp.

\begin{prop}\label{prop:L2-s-Siegel-upper-bound}
Fix $M_\infty\in\Omega_\infty$. Then,
\begin{align*}
\frac{1}{\covol(\Lambda)}\int_{{\Lambda}\backslash{\mathbf{SL}_2(\mathbb{R})}}  &|\vartheta_{M_\infty}(l_\infty, r_\infty;s_\infty)|^2 \dif s_\infty
\leq    (q D_B)^{-1} \sum_{a\mid q D_B} \sum_{\alpha\in\mathbb{Q}}\\
&\int_{\sqrt{3}/2}^\infty \int_{\mathbf{SO}_2(\mathbb{R})}
 \sum_{\substack{\xi\in\widehat{\mathcal{R}}^{(a)}\\\Nr \xi=\alpha}}\left|
\left(\rho\left(\diag(y^{1/2},y^{-1/2}) k\right).M_\infty\right)(l_\infty^{-1} \xi r_\infty)\right|^2
 \dif k \frac{\dif y}{y^2}\;,
\end{align*}
where the measure on $\mathbf{SO}_2(\mathbb{R})$ is normalized to be a probability measure, and the left-hand side is independent of the measure normalization on $\mathbf{SL}_2(\mathbb{R})$.
\end{prop}
\begin{proof}
Fix $M=M_\infty \cdot \prod_{v<\infty}M_v \in \Omega$ such that $M_v=\mathds{1}_{\mathcal{R}_v}$ for all finite $v$. Then,
Proposition \ref{prop:theta-series-adelic-to-classical} and the isomorphism \eqref{eq:SL2-adelic2real} imply 
\begin{equation*}
\frac{1}{\covol(\Lambda)}\int_{{\Lambda}\backslash{\mathbf{SL}_2(\mathbb{R})}}  |\vartheta_{M_\infty}(l_\infty, r_\infty;s_\infty)|^2 \dif s_\infty=
\int_{[\mathbf{SL}_2(\mathbb{A})]} |\Theta_M(l_\infty K_{\mathcal{R}}, r_\infty K_{\mathcal{R}};s)|^2 \dif s\;,
\end{equation*}
We proceed to bound the adelic integral by expanding the domain of integration to a Siegel set.

Denote by $\mathbf{N}<\mathbf{SL}_2$ the algebraic subgroup of upper triangular matrices. We have $\mathbf{N}\simeq \mathbb{G}_a$ and a fundamental domain for the action of $\mathbf{N}(\mathbb{Q})$ on $\mathbf{N}(\mathbb{A})$ is
\begin{equation*}
\mathcal{N}=\begin{pmatrix}
1 & [0,1) \\ 0 & 1
\end{pmatrix} \cdot
\prod_{v<\infty} \begin{pmatrix}
1 & \mathbb{Z}_v \\ 0 & 1
\end{pmatrix}\;.
\end{equation*}

Set $A^>\coloneqq\left\{\diag(y^{1/2},y^{-1/2}) \colon y>\sqrt{3}/2\right\}\subset\mathbf{SL}_2(\mathbb{R})$.
A Siegel set for the action of $\mathbf{SL}_2(\mathbb{Q})$ on $\mathbf{SL}_2(\mathbb{A})$ is given by
\begin{equation*}
\mathcal{S}=\mathcal{N}\cdot A^> \cdot \mathbf{SO}_2(\mathbb{R}) \mathbf{SL}_2(\hat{\mathbb{Z}})\;.
\end{equation*}

Denote $l=(l_\infty,e,e,\dots)$ and similarly $r=(r_\infty,e,e\ldots)$.
Because the Siegel set contains a fundamental domain for the action of the lattice, we can write
\begin{align*}
\int_{[\mathbf{SL}_2(\mathbb{A})]} &|\Theta_M(l_\infty K_{\mathcal{R}}, r_\infty K_{\mathcal{R}};s)|^2 \dif s
\leq \int_{\mathcal{S}}  |\Theta_M(l_\infty K_{\mathcal{R}}, r_\infty K_{\mathcal{R}};s)|^2 \dif s\\
&=\varrho_{\mathbf{SL}_2}^{-1}\int_{[\mathbf{N}(\mathbb{A})]} \int _{\sqrt{3}/2}^\infty \int_{\mathbf{SL}_2(\hat{\mathbb{Z}})} \int_{\mathbf{SO}_2(\mathbb{R})} \left|\sum_{\xi\in B}
(\rho(n)\rho(\diag(y^{1/2},y^{-1/2}) k k_f).M)(l^{-1} \xi r)
\right|^2 \dif n  \frac{\dif y}{y^2}  \dif k_f\dif k\\
&=\varrho_{\mathbf{SL}_2}^{-1}\sum_{\alpha\in\mathbb{Q}} \int _{\sqrt{3}/2}^\infty \int_{\mathbf{SL}_2(\hat{\mathbb{Z}})} \int_{\mathbf{SO}_2(\mathbb{R})} \left|\sum_{\substack{\xi\in B \\ \Nr \xi=\alpha}}
(\rho(\diag(y^{1/2},y^{-1/2}) k)\rho(k_f).M)(l^{-1} \xi r)
\right|^2  \dif k_f \frac{\dif y}{y^2}  \dif k\\
&\leq\varrho_{\mathbf{SL}_2}^{-1}\sum_{\alpha\in\mathbb{Q}} \int _{\sqrt{3}/2}^\infty \int_{\mathbf{SL}_2(\hat{\mathbb{Z}})} \int_{\mathbf{SO}_2(\mathbb{R})} \left|\sum_{\substack{\xi\in B \\ \Nr \xi=\alpha}}
\left|(\rho(\diag(y^{1/2},y^{-1/2}) k)\rho(k_f).M)(l^{-1} \xi r)\right|
\right|^2  \dif k_f \frac{\dif y}{y^2}  \dif k\;.
\end{align*}
The last equality follows from the computation of the Fourier coefficients in the unipotent variable in Proposition \ref{prop:Fourier-theta} and the orthogonality of characters. We normalize the measure on $\mathbf{SL}_2(\mathbb{A}_f)$ so that $\mathbf{SL}_2(\widehat{\mathbb{Z}})$ has volume $1$. Then the global measure normalization constant $\varrho_{\mathbf{SL}_2}$ is equal to the volume of $\lfaktor{\mathbf{SL}_2(\mathbb{Z})}{\mathbb{H}}$ with respect to the standard hyperbolic measure $\frac{\dif x \dif y}{y^2}$, i.e.\ $\varrho_{\mathbf{SL}_2}=\covol(\mathbf{SL}_2(\mathbb{Z}))=\frac{\pi}{3}$.

In the last line we have inserted the absolute value into the sum using the triangle inequality, and we continue to evaluate the integral along $\mathbf{SL}_2(\hat{\mathbb{Z}})$. From Lemmata \ref{lem:non-arch-Weil-maximal-order}, \ref{lem:non-arch-Weil-Eichler}, we know that the integrand is invariant under the finite index subgroup $U_\mathcal{R}<\mathbf{SL}_2(\hat{\mathbb{Z}})$. We decompose the integral into $\left[\mathbf{SL}_2(\hat{\mathbb{Z}})\colon U_{\mathcal{R}}\right]$-integrals along the different cosets of $U_{\mathcal{R}}$ in $\mathbf{SL}_2(\hat{\mathbb{Z}})$.
We have an exact expression for the integrand on each coset due to Proposition \ref{prop:SL2(hatZ)-action}. The phases of the form $\psi\left(\frac{t \cdot \nu_a(x)}{qD_B/a}\right)$ that appear in each element in the $\rho\left(\mathbf{SL}_2(\widehat{\mathbb{Z}})\right)$-orbit are discarded because of the innermost absolute value.
Hence, each $U_{\mathcal{R}}$-coset reduces to a sum over elements in a lattice $\widehat{\mathcal{R}}^{(a)}$ for $a\mid q D_B$.
From Remark \ref{rem:SL2(hatZ)-action-count}, we deduce for any $a \mid q D_B$ that the weight of the sum over $\mathds{1}_{\widehat{R}^{(a)}}$ is $\rho(a \mid q D_B) \varrho_{\mathbf{SL}_2}^{-1} \left[\mathbf{SL}_2(\widehat{\mathbb{Z}})\colon U_{\mathcal{R}}\right]^{-1}$. Because $\Lambda=U_{\mathcal{R}}\cap\mathbf{SL}_2(\mathbb{A}_f)$ and $\mathbf{SL}_2(\mathbb{Z})=\mathbf{SL}_2(\widehat{\mathbb{Z}})\cap\mathbf{SL}_2(\mathbb{A}_f)$, the index satisfies $\varrho_{\mathbf{SL}_2}\left[\mathbf{SL}_2(\widehat{\mathbb{Z}})\colon U_{\mathcal{R}}\right]=\covol(\Lambda)$.
Because $\Lambda=\Gamma_0(qD_B)$, we see that the weight factor is equal to
\begin{equation*}
\covol(\Lambda)^{-1}\rho(a \mid q D_B)=\frac{\prod_{p\mid \gcd(q D_B/a,a)}(1-p^{-1})}{\frac{\pi}{3} q D_B\prod_{p\mid q D_B}(1+p^{-1})}\leq (q D_B)^{-1}
\end{equation*}
\end{proof}

\section{The Theta Lift}
In this section we discuss the pairing between a theta kernel and an automorphic form. This subject is well studied in the literature, we need to review and revisit several results because of the explicit form that we require. 
\label{sec:lift}
\subsection{Cuspidal Theta Series}
\begin{defi}
We say that a test function $M\colon B_{\mathbb{A}}\to\mathbb{C}$ is cuspidal if $\rho(s).M(l^{-1}\xi r)=0$ for all $l,r\in\mathbf{G}(\mathbb{A})$, $s\in\mathbf{SL}_2(\mathbb{A})$, and $\xi\in B$ with $\Nr \xi=0$.
\end{defi}
For example $M$ is cuspidal if $M=\prod_v M_v$ and there is a place $v$ such that $\rho(s_v).M_v(x_v)=0$ for every $s_v\in\mathbf{SL}_2(\mathbb{Q}_v)$, and $x_v\in B_v$ with $\Nr x_v=0$. The importance of cuspidal test functions is that their theta series, when well-defined, is a cuspidal function of $s$ on $[\mathbf{SL}_2(\mathbb{A})]$. This follows from Proposition \ref{prop:Fourier-theta}. 
Note that the cuspidality condition for $M$ is very restrictive if $\mathbf{G}$ is split. For example, if $\mathbf{G}=\mathbf{SL}_2$, then the test function $\exp(-D \Tr(x_\infty\tensor[^t]{x}{_\infty}))P(x_\infty) \prod_{v<\infty} \mathds{1}_{\mathcal{R}_v}$, for a polynomial $P\colon\operatorname{Mat}_{2\times 2}(\mathbb{R})\to\mathbb{C}$ and $D>0$, is used in \cite{Shimizu}. This test function is generally \emph{not} cuspidal.

\subsection{Unfolding}
\begin{lem}
	If $\mathcal{R}_p$ is an Eichler order, then $\Nr \mathcal{R}_p^\times=\mathbb{Z}_p^\times$.
\end{lem}
\begin{proof}
	This is simple to verify if $B$ is split at $p$ by conjugating $\mathcal{R}_p$ to $\sm
	\mathbb{Z}_p & \mathbb{Z}_p \\ p^n \mathbb{Z}_p & \mathbb{Z}_p
	\esm$. If $B$ is ramified at $p$, then $\mathcal{R}_p^\times=\mathcal{O}_{B_p}^\times$ is the unit group of the ring of algebraic integers in $B_p$. The algebra $B_p$ contains an unramified quadratic extension $E/\mathbb{Q}_p$, hence $\Nr \mathcal{R}_p^\times$ contains $\Nr \mathcal{O}_E^\times=\mathbb{Z}_p^\times$.
\end{proof}

\begin{lem}\label{lem:Af-integral}
Let $M_f=\prod_p \mathds{1}_{p^{-k_p}\mathcal{R}_p}\colon B\otimes \mathbb{A}_f\to \mathbb{C}$, where $k_p\in\mathbb{Z}_{\geq 0}$ for all $p$ and $k_p=0$ for a.e.\ primes $p$. Denote $N=\prod_p p^{2k_p}\in\mathbb{N}$ and fix $\xi\in B^\times$. Then,
\begin{equation*}
\int_{\mathbf{G}(\mathbb{A}_f)} M_f(l^{-1}\xi)\dif l
 \ll_\varepsilon \varrho_{\mathbf{G}}^{-1} N |\Nr \xi| q^{\varepsilon}\;,
\end{equation*}
where we recall that $q$ denotes the level of $\mathcal{R}$, and this integral vanishes unless $|\Nr\xi|\in N^{-1}\mathbb{Z}$.
\end{lem}
\begin{proof}
The integral decomposes into a product of local integrals $\varrho_{\mathbf{G}}^{-1}\prod_p \int_{\mathbf{G}(\mathbb{Q}_p)} \mathds{1}_{p^{-k_p}\mathcal{R}_p}(l_p^{-1}\xi) \dif l_p$. All elements of $p^{-k_p}\mathcal{R}_p$ have norms in $p^{-2 k_p} \mathbb{Z}_p$. Hence, the local integral vanishes if $\Nr(l_p^{-1}\xi)=\Nr(\xi)\ \not\in p^{-2 k_p} \mathbb{Z}_p$. Because $\Nr \xi\in \mathbb{Q}$, the non-vanishing conditions at all primes $p$ imply that $\int_{\mathbf{G}(\mathbb{A}_f)}\cdots$ vanishes if $|\Nr \xi |\not \in N^{-1}\mathbb{Z}$.

Fix now $p$ and assume $\Nr\xi\in p^{-2k_p} \mathbb{Z}_p$. Then, the local integral is equal to $\int_{\mathbf{G}(\mathbb{Q}_p)} \mathds{1}_{\mathcal{R_p}}(l_p^{-1}p^{k_p}\xi) \dif l_p$. The integrand is right invariant under $K_p$.  Denote by $\mathcal{R}_p(\alpha)$ the set of elements in $\mathcal{R}_p$ of norm $\alpha\in\mathbb{Q}_p^{\times}$. Of course, $\mathcal{R}_p(\alpha)=\emptyset$ if $\alpha\not\in\mathbb{Z}_p$. The set $\mathcal{R}_p(\alpha)$ is left-invariant under multiplication by $K_p$, and
\begin{equation*}
\int_{\mathbf{G}(\mathbb{Q}_p)} \mathds{1}_{\mathcal{R_p}}(l_p^{-1}p^{k_p}\xi) \dif l_p = m_{\mathbf{G}(\mathbb{Q}_p)}(K_p) \cdot\#\left(\lfaktor{K_p}{\mathcal{R}_p\left(p^{2k_p}\Nr\xi\right)} \right)\;.
\end{equation*}
We have $m_{\mathbf{G}(\mathbb{Q}_p)}(K_p)=1$ if $p\not\mid q$ and $m_{\mathbf{G}(\mathbb{Q}_p)}(K_p)=(p+1)^{-1}p^{-n+1}$ if $p^n\parallel q$ with $n>0$.
 
We now estimate $\#\left(\lfaktor{K_p}{\mathcal{R}_p(\alpha)} \right)$. Define $\mathcal{R}_p(\alpha)^\dagger=\left\{x\in\mathcal{R}_p\colon |\Nr x|_p=|\alpha|_p \right\}$, evidently  $\mathcal{R}_p(\alpha)\subset\mathcal{R}_p(\alpha)^\dagger$. The set $\mathcal{R}_p(\alpha)^\dagger$ is invariant under left multiplication by $\mathcal{R}_p^\times$. Because $\Nr \mathcal{R}_p^\times=\mathbb{Z}_p^\times$, each coset of $\lfaktor{\mathcal{R}_p^\times}{\mathcal{R}_p(\alpha)^\dagger}$ contains exactly one coset of $\lfaktor{K_p}{\mathcal{R}_p(\alpha)}$. Thus $\#\left(\lfaktor{K_p}{\mathcal{R}_p(\alpha)} \right)=\#\left(\lfaktor{\mathcal{R}_p^\times}{\mathcal{R}_p(\alpha)^\dagger} \right)$.
 
If $B_p$ ramifies, then the fact that $\ord_p(\Nr \bullet)$ is a valuation on $B_p$ implies that $\mathcal{R}(\alpha)^\dagger$ is a single coset of $\mathcal{R}_p^\times$ if $\alpha\in\mathbb{Z}_p$. In the split case, we can assume $B_p=\operatorname{Mat}_{2\times 2}(\mathbb{Q}_p)$ and $\mathcal{R}_p=\begin{pmatrix}
\mathbb{Z}_p & \mathbb{Z}_p \\ p^n \mathbb{Z}_p & \mathbb{Z}_p
\end{pmatrix}$, where $p^n\parallel q$. Let $\widetilde{\mathcal{R}}_p^\times$ be the image of $\mathcal{R}_p^\times$ in $\mathbf{PGL}_2(\mathbb{Q}_p)$. Then the map $\lfaktor{\mathcal{R}_p^\times}{\mathcal{R}_p(\alpha)^\dagger}\to \lfaktor{\widetilde{\mathcal{R}}_p^\times}{\mathbf{PGL}_2(\mathbb{Q}_p)}$ is injective because $Z_{\mathbf{GL}_2}(\mathbb{Q}_p)\cap\{g\in\mathbf{GL}_2(\mathbb{Q}_p)\colon |\det g|_p=1\}\subset \mathcal{R}_p^\times$. Hence it is enough to find an upper-bound for the number of $\widetilde{\mathcal{R}}_p^\times$ cosets in the image of $\mathcal{R}_p(\alpha)^\dagger$ in $\mathbf{PGL}_2(\mathbb{Q}_p)$.

The group $\mathbf{PGL}_2(\mathbb{Z}_p)$ is the stablizer of a vertex $v_0$ in the Bruhat-Tits tree of $\mathbf{PGL}_2(\mathbb{Q}_p)$ and  $\diag (p^{-n}, 1) \mathbf{PGL}_2(\mathbb{Z}_p) \diag(p^n,1)$ is a stabilizer of a vertex $v_n$, with $\operatorname{dist}(v_0,v_n)=n$. Hence, $\widetilde{\mathcal{R}}_p^\times$, which is the intersection of the two, is the stabilizer of the geodesic path of length $n$ connecting $v_0$ and $v_n$ in the tree. Because $\mathbf{PGL}_2(\mathbb{Q}_p)$ acts strongly transitively on its Bruhat-Tits tree, it acts transitively on the set of geodesic paths of length $n$. Hence, the map $g_p\mapsto (g_p^{-1}.v_0, g_p^{-1}.v_n)$ is a bijection between $\lfaktor{\widetilde{\mathcal{R}}_p^{\times}}{\mathbf{PGL}_2(\mathbb{Q}_p)}$ and the set of oriented geodesic paths of length $n$ in the tree. We need to find an upper bound on the number of paths that correspond to the image of $\mathcal{R}_p(\alpha)^\dagger$. 

Denote $\ord_p(\alpha)=m$. If $g_p\in\mathcal{R}_p(\alpha)^\dagger$, then the existence of the Smith normal form for $g_p\in M_2(\mathbb{Z}_p)$ implies that $g_p\in\mathbf{GL}_2(\mathbb{Z}_p)  \diag(p^{m_1}, p^{m_2})\mathbf{GL}_2(\mathbb{Z}_p)$ for some $m_2\geq m_1\geq0$, with $m_2+m_1=m$. Then, $\operatorname{dist}(g_p^{-1}.v_0,v_0)=m_2-m_1\leq m$. Hence, the number of possibilities for the first vertex of the $n$-path is at most the number of vertices in a ball of radius $m$, that is $1+\frac{p+1}{p-1}(p^m-1)\ll p^m$. Because the length of the path is $n$, the number of possibilities for the final vertex, after the first vertex has been fixed, is at most $\lfloor(p+1)p^{n-1}\rfloor=m_{\mathbf{G}(\mathbb{Q}_p)}(K_p)^{-1}$. We conclude that if $B_p$ is split, then $m_{\mathbf{G}(\mathbb{Q}_p)}(K_p) \cdot\#\left(\lfaktor{K_p}{\mathcal{R}_p(p^{2k_p}\Nr\xi)} \right)\ll  p^{2k_p+\ord_p(\Nr\xi)}\ll |N \Nr \xi|_p^{-1}$. Multiplying the contributions from all primes $p$ we arrive at the claimed bound.
\end{proof}

\begin{prop}\label{prop:unfolding}
Let $M\colon B_{\mathbb{A}}\to\mathbb{C}$ be a finite linear combination of standard test functions such that the component at infinity satisfies the decay condition of Proposition \ref{prop:Weil-local-uniformity}. Assume that $M$ is cuspidal. Fix $\varphi,\varphi' \in L^\infty([\mathbf{G}(\mathbb{A})])$ and let $\xi_\alpha\in B$ be an arbitrary element of norm $\alpha\in\mathbb{Q}^\times$. Then,
\begin{equation*}
\int_{[\mathbf{G}(\mathbb{A})]} \int_{[\mathbf{G}(\mathbb{A})]} \sum_{\xi\in B} M(l^{-1}\xi r) \varphi(l) \varphi'(r) \dif l \dif r
=\sum_{\alpha \in \mathbb{Q}^\times} \int_{[\mathbf{G}(\mathbb{A})]} \varphi'(r) \int_{\mathbf{G}(\mathbb{A})} M(l^{-1}\xi_\alpha r) \varphi(l) \dif l \dif r
\end{equation*}
\end{prop}
\begin{remark}
The assumption that $M$ is cuspidal is crucial here. Otherwise, there will be an additional contribution from the norm-zero elements of $B$. This contribution may in general diverge.
\end{remark}
\begin{proof}
The theta series $\Theta_M$ can be rewritten as a sum over $B^\times$ due to the vanishing condition for norm-zero elements. A priori, we do not even know that the left-hand-side is integrable. Thus, we proceed with the following computation as formal operations which hold for positive valued functions. We will then use the positive valued case to show absolute convergence which will justify these operations in general.

Unfold first the integral along the $l$ variable to rewrite the left-hand-side above as
\begin{equation*}
\int_{[\mathbf{G}(\mathbb{A})]}\varphi'(r) \sum_{\xi\in \lfaktor{\mathbf{G}(\mathbb{Q})}{B^\times}} \int_{\mathbf{G}(\mathbb{A})} M(l^{-1}\xi r) \varphi(l)  \dif l \dif r\;.
\end{equation*}
Two elements in $B^\times$ are in the same left $\mathbf{G}(\mathbb{Q})$-orbit exactly when they have the same norm. So, the equality in question holds if we can establish absolute integrability. To show absolute integrability, we first bound
\begin{equation*}
\int_{[\mathbf{G}(\mathbb{A})]} \int_{[\mathbf{G}(\mathbb{A})]} \sum_{\xi\in B}\left| M(l^{-1}\xi r) \varphi(l) \varphi'(r)\right| \dif l \dif r \leq \|\varphi\|_\infty \|\varphi'\|_\infty \sum_{\alpha\in \mathbb{Q}^\times} \int_{[\mathbf{G}(\mathbb{A})]} \int_{\mathbf{G}(\mathbb{A})} \left| M(l^{-1}\xi_\alpha r) \right| \dif l \dif r
\;.
\end{equation*}
By expanding the function $M$ into finite summands of standard test functions, we reduce to the case that $M=\prod_v M_v$. Furthermore, because we are only interested in upper bounds, we can further reduce to the case that in any finite place $v=p$ the function $M_v$ is a multiple of the characteristic function of $p^{-k_p} \mathcal{R}_p$, where $k_p=0$ for almost all $p$. Taking into account this reduction, the function $M$ is right and left invariant under $K_{\mathcal{R}}$ and we can apply Lemma \ref{lem:Af-integral} above. First we deduce that the integral over $l$ vanishes unless $\alpha\in N^{-1}\mathbb{Z}$ for some fixed integer $N$ depending only on $M$. And using the bound from Lemma \ref{lem:Af-integral} we can write
\begin{align}
\int_{[\mathbf{G}(\mathbb{A})]} \int_{[\mathbf{G}(\mathbb{A})]} &\sum_{\xi\in B}\left| M(l^{-1}\xi r) \varphi(l) \varphi'(r)\right| \dif l \dif r \ll_{\mathbf{G},\varphi,\varphi',M,\varepsilon} Nq^{\varepsilon} \sum_{0\neq \alpha\in N^{-1}\mathbb{Z}} |\alpha| \int_{\lfaktor{\Lambda}{\mathbf{G}(\mathbb{R})}} \int_{\mathbf{G}(\mathbb{R})} \left|M_\infty(l^{-1}\xi_\alpha r)\right| \dif l \dif r
\nonumber\\
= Nq^{\varepsilon}&\sum_{0\neq \alpha\in N^{-1}\mathbb{Z}} |\alpha| \int_{\lfaktor{\Lambda}{\mathbf{G}(\mathbb{R})}} \int_{\mathbf{G}(\mathbb{R})} \left|M_\infty(l^{-1}\xi_\alpha)\right| \dif l \dif r
=Nq^{\varepsilon}\sum_{0\neq \alpha\in N^{-1}\mathbb{Z}} |\alpha| \int_{\mathbf{G}(\mathbb{R})} \left|M_\infty(l^{-1}\xi_\alpha)\right| \dif l
\nonumber\\
\ll Nq^{\varepsilon}&\sum_{0\neq \alpha\in N^{-1}\mathbb{Z}} |\alpha| \int_{\mathbf{G}(\mathbb{R})} \left(1+\left\|l^{-1}\xi_\alpha \right\|\right)^{-4-\delta} \dif l\;.
\label{eq:unfolding-bound}
\end{align}
In the second line, we have made a change of variable $l\mapsto \xi_\alpha r \xi_\alpha^{-1}l$. Note that we can take here $\xi_\alpha$ to be any \emph{real} matrix of determinant $\alpha$, choose
$\xi_\alpha=\diag \left(\sqrt{|\alpha|}, \sign(\alpha)\sqrt{|\alpha|}\right)$. The integral in \eqref{eq:unfolding-bound} can be computed using the formula for the Haar measure in Cartan coordinates
\begin{align*}
\int_{\mathbf{G}(\mathbb{R})} &\left(1+\sqrt{|\alpha|}\left\|l^{-1}\diag(1,\sign(\alpha)) \right\|\right)^{-4-\delta} \dif l \ll \int_0^{\infty}\left(1+\sqrt{2|\alpha|\cosh(t))}\right)^{-4-\delta}\sinh(t) \dif t \\
\ll &|\alpha|^{-2-\delta/2}\int_0^\infty \frac{\sinh(t)}{\cosh(t)^{2+\delta/2}} \dif t=|\alpha|^{-2-\delta/2}\left[\frac{-1}{\cosh(t)^{1+\delta/2}(1+\delta/2)}\right]_0^\infty\ll |\alpha|^{-2-\delta/2}\;.
\end{align*}
At last, we see that the expression in \eqref{eq:unfolding-bound} is bounded from above by
\begin{equation*}
N^{2+\delta/2}q^{\varepsilon}\sum_{0\neq n \in \mathbb{Z}} \frac{1}{|n|^{1+\delta/2}}<\infty\;.
\end{equation*}
\end{proof}

\begin{prop}\label{prop:L2-s}
Let $M=\prod_vM_v\in \Omega$ be cuspidal and assume $M_\infty\in V_{m,2\pi}$ for $m\in\mathbb{Z}$. Fix $\varphi,\varphi'\in L^\infty([\mathbf{G}(\mathbb{A})])$. Denote
\begin{equation*}
F(s)=\int_{[\mathbf{G}(\mathbb{A})]} \int_{[\mathbf{G}(\mathbb{A})]} \Theta_M(l,r;s) \varphi(l) \varphi'(r) \dif l \dif r\;.
\end{equation*}
Then, $F(s)\in L^2([\mathbf{SL}_2(\mathbb{A})])$.
\end{prop}
\begin{proof}
By Proposition \ref{prop:unfolding}, we know that $F(s)$ is well-defined and can be rewritten as
\begin{equation*}
F(s)=\sum_{\alpha \in \mathbb{Q}^\times} \int_{[\mathbf{G}(\mathbb{A})]} \varphi'(r) \int_{\mathbf{G}(\mathbb{A})} \rho(s) M(l^{-1}\xi_\alpha r) \varphi(l) \dif l \dif r\;.
\end{equation*}
To calculate $\int_{[\mathbf{SL}_2(\mathbb{A})]} |F(s)|^2\dif s$, we will bound the integral over $[\mathbf{SL}_2(\mathbb{A})]$ by an integral over a Siegel set $\mathcal{S}=\mathcal{N}\cdot A^> \cdot \mathbf{SO}_2(\mathbb{R}) \mathbf{SL}_2(\hat{\mathbb{Z}})$ as in the proof of Proposition \ref{prop:L2-s-Siegel-upper-bound}. Because $M$ belongs to $\Omega$ and $M_\infty\in V_{m,2\pi}$, the function $M$ has a finite $\rho\left( \mathbf{SL}_2(\hat{\mathbb{Z}})\right)$-orbit and $\rho\left(\mathbf{SO}_2(\mathbb{R})\right)$-isotypic. Hence, it is enough to bound $\int_{\mathcal{N}\cdot A^>} F_1(z) \overline{F_2(z)} \dif z$ where $F_1$, $F_2$ are defined in the same manner as $F$ but with $M$ replaced by test functions $M_1$, $M_2$ in the $\rho\left(\mathbf{SO}_2(\mathbb{R}) \mathbf{SL}_2(\hat{\mathbb{Z}})\right)$-orbit of $M$. Denote $a(y)=\diag(y^{1/2},y^{-1/2})\in\mathbf{SL}_2(\mathbb{R})$. Using the orthogonality relation of additive characters and the sup-norm bound on $\varphi$, $\varphi'$, we arrive at
\begin{align}
\left|\int_{\mathcal{N}\cdot A^>} F_1(z) \overline{F_2(z)} \dif z\right| \ll_{\varphi,\varphi'}
\sum_{\alpha\in\mathbb{Q}^\times} &\int_{[\mathbf{G}(\mathbb{A})]} \int_{\mathbf{G}(\mathbb{A})}  \int_{[\mathbf{G}(\mathbb{A})]} \int_{\mathbf{G}(\mathbb{A})}
\nonumber\\
&\int_{\sqrt{3}/2}^{\infty} \left|\rho_\infty(a(y)).M_1(l_1^{-1} \xi_\alpha r_1) \overline{\rho_\infty(a(y)).M_2(l_2^{-1} \xi_\alpha r_2)}\right| \frac{\dif y}{y^2} \dif l_2 \dif r_2 \dif l_1 \dif r_1
\nonumber\\
=\sum_{\alpha\in\mathbb{Q}^\times} &\int_{\mathbf{G}(\mathbb{A})} \int_{\mathbf{G}(\mathbb{A})} \int_{\sqrt{3}/2}^{\infty} \left|\rho_\infty(a(y)).M_1(l_1^{-1} \xi_\alpha) \overline{\rho_\infty(a(y)).M_2(l_2^{-1} \xi_\alpha )}\right| \frac{\dif y}{y^2} \dif l_2 \dif l_1\;.
\label{eq:F1F2}
\end{align}
As in the proof of Proposition \ref{prop:unfolding}, we apply Lemma \ref{lem:Af-integral} to the integrals in the $l_1$ and $l_2$ variables. The integral vanishes unless $\alpha\in N^{-1}\mathbb{Z}$ for some integer $N>0$ depending only on $M$. For every $\varepsilon>0$, we can bound \eqref{eq:F1F2} from above by
\begin{align*}
\varrho_{\mathbf{G}}^{-2}&\sum_{0\neq \alpha\in N^{-1}\mathbb{Z}} |N\alpha|^{2}q^{\varepsilon} \int_{\mathbf{G}(\mathbb{R})} \int_{\mathbf{G}(\mathbb{R})} \int_{\sqrt{3}/2}^{\infty} \left|\rho_\infty(a(y)).M_1(l_1^{-1} \xi_\alpha) \overline{\rho_\infty(a(y)).M_2(l_2^{-1} \xi_\alpha )}\right| \frac{\dif y}{y^2} \dif l_2 \dif l_1\\
=\varrho_{\mathbf{G}}^{-2}& \sum_{0\neq \alpha\in N^{-1}\mathbb{Z}}  |N\alpha|^{2}q^{\varepsilon} \int_{\mathbf{G}(\mathbb{R})} \int_{\mathbf{G}(\mathbb{R})} \int_{\sqrt{3}/2}^{\infty} \left|M_1(\sqrt{y} l_1^{-1} \xi_\alpha) \overline{M_2(\sqrt{y} l_2^{-1} \xi_\alpha )}\right| {\dif y} \dif l_2 \dif l_1\\
\ll_{\mathbf{G},M}  &\sum_{0\neq \alpha\in N^{-1}\mathbb{Z}}  |N\alpha|^{2}q^{\varepsilon} \int_{\mathbf{G}(\mathbb{R})} \int_{\mathbf{G}(\mathbb{R})} \int_{\sqrt{3}/2}^{\infty} \|\sqrt{y} l_1^{-1} \xi_\alpha\|^{-4-\delta}\|\sqrt{y} l_2^{-1} \xi_\alpha\|^{-4-\delta} {\dif y} \dif l_2 \dif l_1\\
\ll & \sum_{0\neq \alpha\in N^{-1}\mathbb{Z}}  |N\alpha|^{2}q^{\varepsilon} \int_{\mathbf{G}(\mathbb{R})} \int_{\mathbf{G}(\mathbb{R})}
\|l_1^{-1} \xi_\alpha\|^{-4-\delta}\|l_2^{-1} \xi_\alpha\|^{-4-\delta} \dif l_2 \dif l_1
\end{align*}
Take $\xi_\alpha=\diag(\sqrt{|\alpha|},\sign(\alpha)\sqrt{|\alpha|})$ and bound the last integral from above in the same manner as in the proof of Proposition \ref{prop:unfolding} by a multiple of $|\alpha|^{-4-\delta}$. It follows that
\begin{equation*}
\int_{[\mathbf{SL}_2(\mathbb{A})]} |F(s)|^2 \dif s \ll_{\varphi,\varphi',G, M, \varepsilon} N^{4+\delta}q^{\varepsilon}\sum_{0\neq n \in \mathbb{Z}} \frac{1}{|n|^{2+\delta}}<\infty\;.
\end{equation*}
\end{proof}

\subsection{The Theta Lift}
\begin{defi}\label{defi:theta-lift}
Let $\varphi\in L^2([\mathbf{G}(\mathbb{A})])\cap L^\infty([\mathbf{G}(\mathbb{A})])$ and fix $M\in\Omega$ cuspidal. Define $\varphi_M\colon [\mathbf{SL}_2(\mathbb{A})]\to\mathbb{C}$ by
\begin{equation*}
\varphi_M(s)\coloneqq \int_{[\mathbf{G}(\mathbb{A})]} \int_{[\mathbf{G}(\mathbb{A})]} \Theta_M(l,r;s) \varphi(l) \overline{\varphi(r)} \dif r \dif l\;.
\end{equation*}
We call $\varphi_M$ the \emph{theta lift} of $\varphi$.

For any $\alpha\in \mathbb{Q}^\times$, we also define
\begin{equation*}
T_\alpha^M\varphi(r)\coloneqq
\begin{cases}
\int_{\mathbf{G}(\mathbb{A})} M(l^{-1}\xi_\alpha r) \varphi(l) \dif l, & \alpha\in\Nr B^\times, \\
0, & \textrm{otherwise}.
\end{cases}
\end{equation*}
Assume $M=M_\infty M_f$ with $M_f=\prod_p M_p$.
It would be useful to separate the finite and the archimedean parts in the integral above.
This motivates the definition
\begin{equation*}
T_{\alpha}^{M_f}\varphi(r)\coloneqq
\int_{\mathbf{G}(\mathbb{A}_f)} M_f(l_f^{-1}\xi_\alpha r_f) \varphi\left(\left(\frac{\xi_\alpha}{\sqrt{|\alpha|}}\right)_\infty r_\infty \epsilon_\infty^{(1-\sign \alpha)/2}\cdot l_f\right) \dif l_f\;,
\end{equation*}
where $\epsilon_\infty\in B\otimes \mathbb{R}$ normalizes $K_\infty$ and satisfies\footnote{Such an element does not exist if $B$ is ramified at infinity, but then there are also no elements $\xi_\alpha\in B$ of negative norm.} $\Nr\epsilon_\infty=-1$, $\epsilon_\infty^2=1$. 
Using the change of variable $\left(\frac{\xi_\alpha}{\sqrt{|\alpha|}}\right)^{-1}l_\infty \epsilon_\infty^{(1-\sign(\alpha))/2}\mapsto l_\infty$, 
we arrive at
\begin{multline}
T_{\alpha}^M \varphi(r)= \int_{\mathbf{G}(\mathbb{R})} M_\infty\left(\sqrt{|\alpha|} \epsilon_\infty^{(1-\sign \alpha)/2} l_\infty^{-1}  r_\infty\right) T_\alpha^{M_f}\varphi(l_\infty r_f) \dif l_\infty \\
  =\left(T_\alpha^{M_f}\varphi \star_{\mathbf{G}(\mathbb{R})} M_\infty\left(\sqrt{|\alpha|} \epsilon_\infty^{(1-\sign \alpha)/2} \cdot \bullet\right)\right)(r)  \;.
\label{eq:T-M-alpha-finite-infinite-conv}
\end{multline}
\end{defi}
Note that by Propositions \ref{prop:unfolding} and \ref{prop:L2-s} the theta lift $\varphi_M$ is well-defined and belongs to $L^2([\mathbf{SL}_2(\mathbb{A})])$. The proof of Proposition \ref{prop:L2-s} implies that $T_\alpha^M\varphi$ is a square-integrable function on $\left[\mathbf{G}(\mathbb{A})\right]$ and that
\begin{equation}\label{eq:theta-lift-unfolded}
\varphi_M(s)=\sum_{\alpha\in\mathbb{Q}^\times} \langle T^{\rho(s)M}_\alpha \varphi,\varphi\rangle\;.
\end{equation}
\begin{prop}\label{prop:theta-lift-whittaker}
Let $\varphi$ and $M$ be as in Definition \ref{defi:theta-lift}. Then, for all $\alpha\in\mathbb{Q}^\times$,
\begin{equation*}
W_{\varphi_M}(s;\alpha)=\langle T^{\rho(s)M}_\alpha \varphi,\varphi\rangle\;.
\end{equation*}

More generally, fix $\varphi,\varphi'\in L^2([\mathbf{G}(\mathbb{A})])\cap L^\infty([\mathbf{G}(\mathbb{A})])$ and set
\begin{equation*}
F(s)= \int_{[\mathbf{G}(\mathbb{A})]} \int_{[\mathbf{G}(\mathbb{A})]} \Theta_M(l,r;s) \varphi(l) \overline{\varphi'(r)} \dif r \dif l\;.
\end{equation*}
Then $W_{F}(s;\alpha)=\langle T^{\rho(s)M}_\alpha \varphi,\varphi'\rangle$, and
\begin{equation} \label{eq:theta-product-unfolded}
F(s)=\sum_{\alpha\in\mathbb{Q}^\times} \langle T^{\rho(s)M}_\alpha \varphi,\varphi'\rangle\;.
\end{equation}
\end{prop}
\begin{proof}
We only establish the second claim as it immediately implies the first. Proposition \ref{prop:L2-s} implies $F(s)\in L^2([\mathbf{SL}_2(\mathbb{A})])$. We then apply Propositions \ref{prop:unfolding} to deduce \eqref{eq:theta-product-unfolded}.

Denote $u_n\coloneqq \sm
1 & n \\ 0 & 1
\esm$.
Fubini's theorem and the orthogonality of characters imply for all $\alpha,\beta\in\mathbb{Q}^\times$, $s \in \mathbf{SL}_2(\mathbb{A})$, and $x\in\mathbf{G}(\mathbb{A})$
\begin{equation*}
\int_{[\mathbf{N}(\mathbb{A})]} T^{\rho(u_n s)M}_\alpha \varphi(x)  \psi(-\beta n)\dif n=\begin{cases}
T^{\rho(s)M}_\alpha\varphi (x),  & \alpha=\beta,\\
0, & \alpha \neq \beta.
\end{cases}\
\end{equation*}
The claim follows from substituting this expression in the definition of the Whittaker function applied to  \eqref{eq:theta-product-unfolded}.
\end{proof}

\subsubsection{Hecke Operators}
We would like to describe the relation between the Fourier--Whittaker expansion of $\varphi_M$ and the Hecke translates of $\varphi$. A minor difficulty is that the Hecke algebra of $\mathbf{G}(\mathbb{A})$ is not rich enough and we would prefer to work with the Hecke algebra of the adjoint group $\mathbf{G}^\mathrm{adj}(\mathbb{A})$. To that end, we lift a $K_{\mathcal{R}}$-invariant function on $[\mathbf{G}(\mathbb{A})]$ to $[\mathbf{G}^\mathrm{adj}(\mathbb{A})]$. An alternative more conceptual approach is to work with a $\mathbf{PGL}_2$ Weil representation, c.f.\ \cite[\S I.3]{Waldspurger} and \cite[\S 2.2.5]{NelsonQV2}.

Let us recall that the adjoint group is the affine algebraic group over $\mathbb{Q}$ representing the functor
\begin{equation*}
\mathbf{G}^\mathrm{adj}(L)\coloneqq  \lfaktor{L^\times}{(B\otimes L)^\times} 
\end{equation*}
for any $\mathbb{Q}$-algebra $L$, where $L^\times$ is embedded centrally in $(B\otimes L)^\times$. We will also use the algebraic  group $\mathbf{B}^\times(L)=(B\otimes L)^\times$, i.e.\ $\mathbf{G}^\mathrm{adj}=\lfaktor{Z_{\mathbf{B}^\times}}{\mathbf{B}^\times}$.

\begin{defi}
For each finite place $v$ denote by $\widetilde{K}_v$ the image of $\mathcal{R}_v^\times$ in $\mathbf{G}^\mathrm{adj}(\mathbb{Q}_v)=\lfaktor{\mathbb{Q}_v^\times}{B_v^\times}$.
\end{defi}

\begin{cor}\label{cor:G-Gadj}
The natural map
\begin{equation*}
\dfaktor{\mathbf{G}(\mathbb{Q})}{\mathbf{G}(\mathbb{A})}{K_{\mathcal{R}}}\rightarrow
\dfaktor{\mathbf{G}^\mathrm{adj}(\mathbb{Q})}{\mathbf{G}^\mathrm{adj}(\mathbb{A})}{\widetilde{K}_{\mathcal{R}}}
\end{equation*}
is a measure preserving bijection. In particular, we have a Hilbert space isomorphism between  $L^2([\mathbf{G}(\mathbb{A})])^{K_{\mathcal{R}}}$ and  $L^2([\mathbf{G}^\mathrm{adj}(\mathbb{A})])^{K_{\mathcal{R}}}$.
\end{cor}
\begin{proof}
Denote by $h \colon \mathbf{G}\to\mathbf{G}^{\mathrm{adj}}$ the standard isogeny. The image is a normal subgroup and the quotient is abelian.
The kernel of the map $h$ is the center $\mathbf{Z}<\mathbf{G}$. The center is isomorphic to the group of second order roots of unity $\mu_2$. The reduced norm map then completes the short exact sequence $1\to\mu_2\to\mathbf{G}\to\mathbf{G}^{\mathrm{adj}}\xrightarrow{\Nr}\lfaktor{\mathbb{G}_m^{2}}{\mathbb{G}_m}\to 1$. 

For a local field or a number field $F$ the image of $\Nr((B\otimes F)^\times)$ in $F^\times$, is determined by the Hasse--Schilling--Maass theorem. In particular, $\Nr((B\otimes F)^\times)=F^\times$ if $F=\mathbb{Q}_p$, or $F=\mathbb{R}$ and $B$ is indefinite. If $B$ is definite, then $\Nr((B\otimes \mathbb{R})^\times)=\mathbb{R}_{>0}$. Finally, $\Nr((B\otimes \mathbb{Q})^\times)$ is $\mathbb{Q}^\times$ if $B$ is indefinite and $\mathbb{Q}_{>0}$ otherwise.
It follows that $\lfaktor{h(\mathbf{G}(\mathbb{A}))}{\mathbf{G}^\mathrm{adj}(\mathbb{A})}\xrightarrow[\Nr]{\sim}\lfaktor{\mathbb{A}^{\times 2}}{\mathbb{A}^\times}$ if $B$ is indefinite and $\lfaktor{h(\mathbf{G}(\mathbb{A}))}{\mathbf{G}^\mathrm{adj}(\mathbb{A})}\xrightarrow[\Nr]{\sim}\lfaktor{\mathbb{A}^{\times 2}}{\left(\mathbb{R}_{>0}\times\mathbb{A}_f^\times\right)}\simeq \lfaktor{\mathbb{A}_f^{\times 2}}{\mathbb{A}_f^\times}$  if $B$ is definite.
\item\paragraph*{Injectivity}
Assume $h(g')=\tilde{\gamma} h(g) \tilde{k}$ for some $g,g'\in\mathbf{G}(\mathbb{A})$, $\tilde{\gamma}\in\mathbf{G}^{\mathrm{adj}}(\mathbb{Q})$ and $\tilde{k}\in \widetilde{K}_{\mathcal{R}}$. We need to show $[g]=[g']$ in $\dfaktor{\mathbf{G}(\mathbb{Q})}{\mathbf{G}(\mathbb{A})}{K_{\mathcal{R}}}$. To show that $\tilde{\gamma}\in h(\mathbf{G}(\mathbb{Q}))$ we establish that $\Nr(\tilde{\gamma})$ is a square in $\mathbb{Q}^\times$, this can be checked locally at all places.
Examining the archimedean component of the equality, we arrive at $\tilde{\gamma}=h(g'_\infty g_\infty^{-1})$. Hence, $\Nr \tilde{\gamma}$ is positive. Similarly in all finite places $p<\infty$ we can write $\tilde{\gamma}=h(g'_pg_p^{-1})\tilde{k}_p^{-1}$, and $|\Nr\tilde{\gamma}|_p=|h(g_p'g_p^{-1})|_p\in p^{2\mathbb{Z}}$. Thus, $\tilde{\gamma}=h(\gamma)$ for some $\gamma\in\mathbf{G}(\mathbb{Q})$. We can now write $h(g')=h(\gamma g)\tilde{k}$. Hence, $\Nr \tilde{k}=1$ as well, and $\tilde{k}=h(k)$ for some $k\in K_{\mathcal{R}}$. We deduce $h(g')=h(\gamma g k)$, and $[g']\in \left[\mathbf{Z}(\mathbb{A}) g \right]$ in
$\dfaktor{\mathbf{G}(\mathbb{Q})}{\mathbf{G}(\mathbb{A})}{K_{\mathcal{R}}}$.

To conclude $[g']=[g]$ it is enough to show that $\dfaktor{\mathbf{Z}(\mathbb{Q})}{\mathbf{Z}(\mathbb{A})}{\prod_p\left(K_p\cap \mathbf{Z}(\mathbb{Q}_p)\right)}$ is a trivial group. Because $K_p$ contains $\pm \Id$ for all $p$, this group is $\dfaktor{\pm1}{\mathbb{A}^\times[2]}{\prod_p \mathbb{Z}_p^\times[2]}\simeq 1$ as required.
\paragraph*{Surjectivity}
Using the norm map, it is enough to demonstrate that $\dfaktor{\mathbb{A}^{\times 2}\mathbb{Q}^\times}{\mathbb{A}^\times}{\prod_p \Nr \widetilde{K}_p}$ is trivial.
The previous lemma implies that $\Nr \widetilde{K}_p=\mathbb{Z}_p^\times$ for all $p$. Because $\mathbb{Q}$ has class number $1$, the double quotient is isomorphic to $\dfaktor{\mathbb{R}^{\times 2}}{\mathbb{R}^\times}{\mathbb{Z}^\times}\simeq 1$.
\paragraph*{Measure preservation}
Strong approximation implies that the group $\mathbf{G}(\mathbb{R})$ acts transitively on the left-hand-side in the claimed equality. Hence, it acts transitively on the right-hand-side as well because the map is equivariant.
The Haar measure on both spaces is a $\mathbf{G}(\mathbb{R})$-invariant probability measure on a locally compact homogeneous $\mathbf{G}(\mathbb{R})$-space. Uniqueness of Haar measure implies that the map is measure preserving.
\end{proof}

\begin{defi}\label{defi:adj-lift}
Let $\varphi\colon[\mathbf{G}(\mathbb{A})]\to\mathbb{C}$ be $K_{\mathcal{R}}$-invariant. Denote by $\widetilde{\varphi}\colon[\mathbf{G}^{\mathrm{adj}}(\mathbb{A})]\to\mathbb{C}$ its unique lift to a $\widetilde{K}_\mathcal{R}$-invariant function on $[\mathbf{G}^{\mathrm{adj}}(\mathbb{A})]$.
\end{defi}

Denote $\widetilde{\Gamma}=\mathbf{G}^\mathrm{adj}(\mathbb{Q})\cap\widetilde{K}_\mathcal{R}$, which is a a lattice in $\mathbf{G}^\mathrm{adj}(\mathbb{R})$. Equivalently, $\widetilde{\Gamma}$ is the image of $\mathcal{R}^\times$ in $\mathbf{G}^\mathrm{adj}(\mathbb{R})$.  Corollary \ref{cor:G-Gadj} implies that $\mathbf{G}^{\mathrm{adj}}(\mathbb{R})$ acts transitively on $\dfaktor{\mathbf{G}^\mathrm{adj}(\mathbb{Q})}{\mathbf{G}^\mathrm{adj}(\mathbb{A})}{\widetilde{K}_{\mathcal{R}}}$ and $\dfaktor{\mathbf{G}^\mathrm{adj}(\mathbb{Q})}{\mathbf{G}^\mathrm{adj}(\mathbb{A})}{\widetilde{K}_{\mathcal{R}}}\simeq \lfaktor{\widetilde{\Gamma}}{\mathbf{G}^{\mathrm{adj}}}(\mathbb{R})$.
We introduce Hecke operators adapted to the order $\mathcal{R}$.
\begin{defi} Let $\alpha\in \mathbb{Q}$ and $f\colon \lfaktor{\widetilde{\Gamma}}{\mathbf{G}^{\mathrm{adj}}(\mathbb{R})}\to\mathbb{C}$ continuous. Set $\mathcal{R}(\alpha)=\left\{x\in\mathcal{R}\mid \Nr x =\alpha \right\}$, $\mathcal{R}(\alpha)^\dagger=\left\{x\in\mathcal{R}\mid |\Nr x|_\infty =|\alpha|_\infty \right\}$ and define
\begin{equation*}
T_\alpha f(g)=
\sum_{[\delta^+]\in \lfaktor{\mathcal{R}(1)}{\mathcal{R}(\alpha)}} f\left(\delta^+ g \epsilon_{\infty}^{(1-\sign\alpha)/2}\right)
=\sum_{[\delta]\in \lfaktor{\mathcal{R}^\times}{\mathcal{R}(\alpha)^\dagger}} f\left(\delta g \epsilon_{\infty}^{(1-\sign\alpha)/2}\right) \;.
\end{equation*}
The two expressions are equal because $\mathcal{R}^\times$ contains an element of determinant $-1$ if $B$ is indefinite. These operators coincide with the classical Hecke operators for $\alpha>0$ co-prime to $q D_B$. Note that if $\alpha\not\in \Nr\mathcal{R}$ then $T_\alpha=0$.
\end{defi}

\begin{lem}\label{lem:Hecke-convolution}
Let $\alpha\in \mathbb{Q}^\times$ and $f\colon [\mathbf{G}^{\mathrm{adj}}(\mathbb{A})]\to\mathbb{C}$ continuous and $\widetilde{K}_\mathcal{R}$-invariant. Set $\mathcal{R}_f(\alpha)^\dagger=\prod_p \left\{ x_p\in\mathcal{R}_p \mid |\Nr x_p|_p\in |\Nr \alpha|_p\ \right\}$. Then for every $g\in \mathbf{G}^{\mathrm{adj}}(\mathbb{R})$
\begin{equation*}
\left(T_\alpha f(\bullet \widetilde{K}_\mathcal{R})\right)(g)=\left( f\star \mathds{1}_{\mathcal{R}_f(\alpha)^\dagger}\right)(g  \epsilon_{\infty}^{(1-\sign\alpha)/2})\;,
\end{equation*}
where the convolution takes place in $\mathbf{B}^\times(\mathbb{A}_f)$ with the measure normalization $m_{\mathbf{B}^\times(\mathbb{A}_f)}(\mathcal{R}_f^\times)=1$.
\end{lem}
\begin{proof} 
The right $\widetilde{K}_{\mathcal{R}}$-invariance of $f$ and the left $\mathcal{R}_f^\times$-invariance of $\mathds{1}_{\mathcal{R}_f(\alpha)^\dagger}$ imply
\begin{equation*}
\left( f\star \mathds{1}_{\mathcal{R}_f(\alpha)^\dagger}\right)(g)=\int_{\mathcal{R}_f(\alpha)^\dagger} f(g h_f^{-1}) \dif h_f 
=\sum_{[a_f]\in \lfaktor{\mathcal{R}_f^\times}{\mathcal{R}_f(\alpha)^\dagger}} f(g a_f^{-1})\;.
\end{equation*}
There is a natural map $\lfaktor{\mathcal{R}^\times}{\mathcal{R}(\alpha)^\dagger}\to \lfaktor{\mathcal{R}_f^\times}{\mathcal{R}_f(\alpha)}^\dagger$. Strong approximation implies that this map is surjective. To show this map is also injective we observe that if $\delta\equiv \delta' \mod \mathcal{R}_f^\times$ for $\delta,\delta'\in \mathcal{R}(\alpha)$, then $\delta \delta'^{-1},\delta'\delta^{-1}\in B \cap \mathcal{R}_f^\times\subset \mathcal{R}$, and $\delta \delta'^{-1}\in \mathcal{R}^\times$. By choosing a rational representative for each coset in $ \lfaktor{\mathcal{R}_f^\times}{\mathcal{R}_f(\alpha)}$ and using the left $\mathbf{G}^\mathrm{adj}(\mathbb{Q})$-invariance of $f$, we arrive at the claim.
\end{proof}

\begin{prop}\label{prop:theta-hecke}
Let $0\neq\alpha\in \mathbb{Q}^\times$ and assume $M_f=\prod_p \mathds{1}_{\mathcal{R}_p}$. Then,
\begin{equation*}
\widetilde{T^{M_f}_\alpha \varphi}(r)=T_\alpha \widetilde{\varphi}\left(r\right)\;.
\end{equation*}
\end{prop}

\begin{proof}
Assume $\alpha\in \Nr\mathcal{R}$, otherwise the claim is trivial.
Although we claim the equality for all $r\in[\mathbf{G}^\mathrm{adj}(\mathbb{A})]$, because of the uniqueness of the lift in Definition \ref{defi:adj-lift}, it is enough to verify the claim for $r\in[\lfaktor{\mathbf{Z}}{\mathbf{G}}(\mathbb{A})]$. We apply Lemma \ref{lem:Hecke-convolution}
and evaluate the convolution by decomposing the Haar measure on $\mathbf{B}^\times(\mathbb{A}_f)$ into fibers over $\lfaktor{\mathbf{G}(\mathbb{A}_f)}{\mathbf{B}^\times(\mathbb{A}_f)}$, this is possible because $\mathbf{B}^\times(\mathbb{A}_f)$ and $\mathbf{G}(\mathbb{A}_f)$ are unimodular. For consistent measure normalization we set $m_{\mathbf{B}^\times(\mathbb{A}_f)}(\mathcal{R}_f^\times)=m_{\mathbf{G}(\mathbb{A}_f)}(K_f)=m_{\Nr B_{\mathbb{A}_f}^\times}(\widehat{\mathbb{Z}}^\times)=1$.
\begin{align*}
\widetilde{\varphi} \star \mathds{1}_{\mathcal{R}_f(\alpha)^\dagger} (r)&= \int_{\mathbf{B}^\times(\mathbb{A}_f)} \mathds{1}_{\mathcal{R}_f(\alpha)^\dagger}(l_f^{-1} r_f) \widetilde{\varphi}(r_\infty l_f) \dif l_f=
\int_{\lfaktor{\mathbf{G}(\mathbb{A}_f)}{\mathbf{B}^\times(\mathbb{A}_f)}}  \int_{\mathbf{G}(\mathbb{A}_f)}
\mathds{1}_{\mathcal{R}_f(\alpha)^\dagger}(\lambda^{-1} l_f^{-1} r_f) \widetilde{\varphi}(r_\infty l_f \lambda) \dif l_f \dif\lambda\\
&=
\int_{\dfaktor{\mathbf{G}(\mathbb{A}_f)}{\mathbf{B}^\times(\mathbb{A}_f)}{\mathcal{R}_f^\times}}  \int_{\mathbf{G}(\mathbb{A}_f)}
\mathds{1}_{\mathcal{R}_f(\alpha)^\dagger}(\lambda^{-1} l_f^{-1} r_f) \widetilde{\varphi}(r_\infty l_f \lambda) \dif l_f \dif\lambda\;.
\end{align*}
In the last line, we have used the fact that $\mathds{1}_{\mathcal{R}_f(\alpha)^\dagger}(l_f^{-1} r_f)$ is left $\mathcal{R}_f^\times$-invariant and $\widetilde{\varphi}$ is right $\widetilde{K}_f$-invariant. Fix $\xi_\alpha\in B^\times$ with $\Nr \xi_\alpha=\alpha$.
Because $\Nr l_f^{-1} r_f=1$ and $\Nr \mathcal{R}_f(\alpha)^\dagger=\hat{\mathbb{Z}}^\times \alpha$, the external integral vanishes unless $\lambda \equiv \xi_\alpha^{-1} \mod \dfaktor{\mathbf{G}(\mathbb{A}_f)}{\mathbf{B}^\times(\mathbb{A}_f)}{\mathcal{R}_f^\times}$. We conclude that
\begin{align*}
\widetilde{\varphi} \star \mathds{1}_{\mathcal{R}_f(\alpha)^\dagger} (r)
&=\int_{\mathbf{G}(\mathbb{A}_f)}
\mathds{1}_{\mathcal{R}_f(\alpha)^\dagger}(\xi_\alpha l_f^{-1} r_f) \widetilde{\varphi}(r_\infty l_f (\xi_\alpha)_f^{-1}) \dif l_f
=\int_{\mathbf{G}(\mathbb{A}_f)}
\mathds{1}_{\mathcal{R}_f(\alpha)^\dagger}(l_f^{-1} \xi_\alpha r_f) \widetilde{\varphi}(r_\infty(\xi_\alpha)_f^{-1} l_f) \dif l_f\\
&=\int_{\mathbf{G}(\mathbb{A}_f)}
\mathds{1}_{\mathcal{R}_f(\alpha)^\dagger}(l_f^{-1} \xi_\alpha r_f) \widetilde{\varphi}((\xi_\alpha)_\infty r_\infty l_f) \dif l_f
=\int_{\mathbf{G}(\mathbb{A}_f)}
\mathds{1}_{\mathcal{R}_f(\alpha)^\dagger}(l_f^{-1} \xi_\alpha r_f) \widetilde{\varphi}\left(\left(\frac{\xi_\alpha}{\sqrt{|\alpha|}}\right)_\infty r_\infty l_f\right) \dif l_f\\
&=
T_\alpha^{M_f} \varphi (r\epsilon_\infty^{(1-\sign \alpha)/2})\;,
\end{align*}
where in the first line we have used the change of variables $\xi_\alpha l_f \xi_\alpha^{-1}\mapsto l_f$ and in the second line we have applied the left  $\mathbf{G}^{\mathrm{adj}}(\mathbb{Q})$-invariance of $\widetilde{\varphi}$.
\end{proof}

\begin{cor}\label{cor:Whittaker-theta-product-Hecke}
Let $\varphi,\varphi'\in L^2([\mathbf{G}(\mathbb{A})])^{K_\mathcal{R}}\cap L^{\infty}([\mathbf{G}(\mathbb{A})])$ and $\alpha\in\mathbb{Q}^\times$. Assume $T_\alpha\varphi=\lambda(\alpha)\varphi$. Then, the function
\begin{equation*}
F(s)= \int_{[\mathbf{G}(\mathbb{A})]} \int_{[\mathbf{G}(\mathbb{A})]} \Theta_M(l,r;s) \varphi(l) \overline{\varphi'(r)} \dif r \dif l\;,
\end{equation*}
satisfies
\begin{equation*}
W_{F}(s_\infty U_{\mathcal{R}};\alpha)=\lambda(\alpha)\left\langle \varphi \star \left(\rho(s_\infty).M_\infty\left(\sqrt{|\alpha|} \epsilon_\infty^{(1-\sign \alpha)/2} \cdot \bullet\right)\right), \varphi' \right\rangle\;
\end{equation*}
for all $s_\infty\in\mathbf{SL}_2(\mathbb{R})$, where the convolution takes place in $\mathbf{G}(\mathbb{R})$.
\end{cor}
\begin{proof}
Propositions \ref{prop:theta-lift-whittaker}, \ref{prop:theta-hecke}, and Equation \eqref{eq:T-M-alpha-finite-infinite-conv} imply that
\begin{align*}
W_{F}(s_\infty U_{\mathcal{R}};\alpha)&=\left\langle T_{\alpha}^{M_f}\varphi \star \left(\rho(s_\infty).M_\infty\left(\sqrt{|\alpha|} \epsilon_\infty^{(1-\sign \alpha)/2} \cdot \bullet\right)\right), \varphi' \right\rangle\\
&=\lambda(\alpha)\left\langle \varphi \star \left(\rho(s_\infty).M_\infty\left(\sqrt{|\alpha|} \epsilon_\infty^{(1-\sign \alpha)/2} \cdot \bullet\right)\right), \varphi' \right\rangle\;.
\end{align*}
\end{proof}
\section{The Bergman Kernel}
\label{sec:Bergman}
\subsection{The Bergman Archimedean Test Function}
From now on, we shall assume that $B$ is split over $\mathbb{R}$. Recall that  we have fixed an isomorphism $B_\infty\simeq \operatorname{Mat}_{2\times 2}(\mathbb{R})$ and have used it to identify the two spaces.
We construct a theta series whose Fourier--Whittaker coefficients coincide with the Bergman kernel. For this endeavour, we will use the following archimedean test function. We fix the global character $\psi$ so that $\psi_\infty(r)=\exp(-2\pi i r)$, note that this is different from the convention we have used in \S\ref{sec:archimedean-weil}.

\begin{defi}\label{defi:M_infty-bergman}
Fix a weight $m\geq 2$ and define
\begin{equation*}
M_{\infty}^{(m)}(x)=\exp(-2\pi \Nr x)  \begin{cases}
\displaystyle
\frac{\Nr(x)^{m-1}}{\left(\frac{(b-c)+i(a+d)}{2i}\right)^m} & \Nr x >0, \\
 0 & \Nr x \leq 0,
\end{cases}
\end{equation*}
for $x=\sm
a & b \\ c & d
\esm$. Notice that $M_\infty^{(m)}(\tensor[^\iota]{x}{})=\overline{M_\infty^{(m)}}(x)$.

Set $\mu\colon \mathbf{PGL}_2(\mathbb{R})\to\mathbb{C}$
\begin{equation*}
\mu(x)=
\begin{cases}
\displaystyle
\frac{2i\sqrt{\Nr x}}{(b-c)+i(a+d)}  & \Nr x >0,\\
0 & \Nr x \le 0.
\end{cases}
\end{equation*}
Then, we can write $M_\infty^{(m)}(x)=\exp(-2 \pi \Nr x) \Nr(x)^{m/2-1}\mu(x)^m$.
\end{defi}

\begin{lem}
Let $k_{\theta}\coloneqq \sm \cos\theta & \sin\theta \\ -\sin\theta & \cos\theta \esm \in \mathbf{SO}_2(\mathbb{R})$. Then, for every $g\in\mathbf{PGL}_2(\mathbb{R})$,
\begin{equation}\label{eq:mu(g)-inv}
\mu(g k_{\theta})=\mu(g)e^{i\theta}.
\end{equation}
\end{lem}
\begin{proof}
We assume $\Nr g>0$ as the claim is trivial for non-positive determinants. Write
\begin{equation}
\mu(g)^{-1}=\frac{i\Tr(g)-\Tr(gw)}{2i\sqrt{\Nr g}}\;.
\label{eq:mu(g)-expr}
\end{equation}
Then, $\mu(g k_{\theta})^{-1}=\frac{i\Tr(gk_{\theta})-\Tr(gk_{\theta}w)}{2i\sqrt{\Nr g}}$. Note that $\dif k_\theta/\dif \theta= k_\theta k_{\pi/2} =k_\theta w$. Using the linearity of the trace and matrix multiplication operations, we deduce
\begin{equation*}
\frac{\dif}{\dif \theta} \left(\mu(g k_{\theta})^{-1}\right)
=\mu\left(g\frac{\dif}{\dif \theta}k_{\theta}\right)^{-1} =\mu(g k_{\theta}w)^{-1}
=\frac{i\Tr(g k_{\theta}w)+\Tr(g k_{\theta})}{2i\sqrt{\Nr g}}
=-i\mu(g k_{\theta})^{-1}\;,
\end{equation*}
where we have used formula \eqref{eq:mu(g)-expr}. Solving this simple first-order ODE we have $\mu(g k_{\theta})^{-1}=A(g)e^{-i\theta}$. Moreover, in the special case $\theta=0$ we have $A(g)=\mu(g k_0)^{-1}=\mu(g)^{-1}$, and the claim follows. 
\end{proof}
\begin{cor}
For every weight $m\geq 2$ and $k_{\theta_1},k_{\theta_2}\in \mathbf{SO}_2(\mathbb{R})$,
\begin{equation*}
M_\infty^{(m)} (k_{\theta_1}x k_{\theta_2})=e^{i m (\theta_2+\theta_1)} M_\infty^{(m)}(x)\;.
\end{equation*}
\end{cor}\label{cor:k-double-invariance}
\begin{proof}
Apply the previous lemma to $M_\infty^{(m)}(x)=\exp(-2\pi \Nr x) \Nr(x)^{m/2-1} \mu(x)^m$ and use the identity $M_\infty^{(m)}(x^{\iota})=\overline{M_\infty^{(m)}(x)}$.
\end{proof}

\begin{lem}\label{lem:Bergman-decay}
If $m\geq 2$, then
\begin{equation*}
|M_\infty^{(m)}(x)| \ll_m (1+\|x\|)^{-m}\;.
\end{equation*}
\end{lem}
\begin{proof}
This is trivial if $\Nr x \le 0$, hence we assume the determinant is positive. Denote $r=\|x\|=\sqrt{\Tr(x\tensor[^t]{x}{})}$. A simple calculation shows that
\begin{align*}
|\mu(x)|^{-2}&=\frac{r^2+2\Nr(x)}{4\Nr(x)}\;, \\
|M_\infty^{(m)}(x)|&=\exp(-2\pi \Nr x)\frac{2^m\Nr(x)^{m-1}}{(r^2+2\Nr(x))^{m/2}}\;.
\end{align*}
If $r\leq 1$, then we write
\begin{equation*}
|M_\infty^{(m)}(x)|\ll_m \exp(-2\pi \Nr x) \Nr(x)^{m/2-1} \ll 1\;.
\end{equation*}
The last equality holds for all $\Nr(x)>0$ because we have assumed $m\geq2$.
Otherwise, if $r>1$ then
\begin{equation*}
|M_\infty^{(m)}(x)|\leq \exp(-2\pi \Nr x)2^m \Nr(x)^{m-1} r^{-m}\ll_m r^{-m}\;.
\end{equation*}
Where, we have used the fact that the real function $\exp(-2\pi t) t^{m-1}$ is bounded for $t \in [0,\infty)$ and $m\geq 1$.
\end{proof}

Up until this point,  we have established that $M_\infty^{(m)}$ satisfies the decay condition in Definition \ref{defi:Omega_infty} if $m>4$. We now turn to show that it also belongs to the space $V_{m,2\pi}$ by checking that it solves the harmonic oscillator equation \eqref{eq:harmonic-oscillator}.

\begin{lem}\label{lem:Bergman-isotypic}
If $m \ge 6$ then the function $M_\infty^{(m)}$ from Definition \ref{defi:M_infty-bergman} belongs to $V_{m,2\pi}$.
\end{lem}
\begin{proof}
Lemma \ref{lem:Bergman-decay} above implies that $M_\infty^{(m)}\in L^2(B_\infty)$ if $m>4$. To prove $M_\infty^{(m)}\in V_{m,2\pi}$ we will show that $M_\infty^{(m)} \perp V_{m',2\pi}$ for all $m'\neq m$. It is enough to establish $\langle M_\infty^{(m)}, M' \rangle=0$ for any Schwartz solution $M'$ of \eqref{eq:harmonic-oscillator} with quantum number $m'\neq m$ and $\omega=2\pi$.

Define $F(x)= \exp(-2\pi\Nr(x))N(x)$ and $N(x)=\Nr(x)^{m-1}(2i)^m\left((b-c)+i(a+d)\right)^{-m}$. Then, $F(x)$ is a well-defined continuous function on the open subset $\mathcal{U}=B_\infty \setminus\left\{\begin{pmatrix}  a & b \\ b & -a
\end{pmatrix}\colon a,b\in\mathbb{R} \right\}$. Moreover, $M_\infty^{(m)}=\mathds{1}_{\Nr(x)>0}\cdot F$. 
Define $V=\{x\in B_\infty \mid \Nr x\geq 0 \}$ and $V_R=V\cap B(0,R) \setminus B(0,R^{-1})$ for $R>1$, where $B(0,r)$ is a closed ball of radius $r$ centered at the origin. Note that $V_R\subset \mathcal{U}$.

We claim that $F(x)$ solves the PDE \eqref{eq:harmonic-oscillator} on $\mathcal{U}$ with $\omega=2\pi$ and quantum number $m$. The PDE \eqref{eq:harmonic-oscillator} with $\omega=2\pi$, $\varpi=1$ for $F$ is equivalent to the following PDE for $N$
\begin{equation}\label{eq:euler-operator-harmonic}
-\Delta N+2\pi \left\langle x, \nabla \right \rangle N=2\pi (m-2) N\;,
\end{equation}
where $\nabla$ stands for the gradient operator and the bilinear form $\langle x_1, x_2\rangle$ is the twisted trace form $\Tr(x_1 x_{2}^{\iota})$ as before. Using the definition of the Laplace operator as the Fourier multiplier with symbol $-4\pi^2 \Nr$ and the definition of the gradient, we arrive at
\begin{align*}
\Delta &= \frac{\partial^2}{\partial a \partial d}-\frac{\partial^2}{\partial b \partial c}\;,\\
\left\langle x, \nabla \right \rangle & = a\frac{\partial}{\partial a} + d\frac{\partial}{\partial d}
+ b\frac{\partial}{\partial b} + c\frac{\partial}{\partial c}\;.
\end{align*}
Substituting the definition $N(x)=(2i)^m\frac{(ad-bc)^{m-1}}{((b-c)+i(a+d))^m}$ into the formul{\ae} above we see that
\begin{align*}
\Delta N &= 0\;,\\
\left\langle x, \nabla \right \rangle N &=(m-2) N\;.
\end{align*}
These and \eqref{eq:euler-operator-harmonic} show that $F$ is a solution with quantum number $m$. Because the PDE \eqref{eq:harmonic-oscillator} is local, this establishes that\footnote{The function $M_\infty^{(m)}$ is not necessarily in $C^2(B_\infty)$ and we extend the definition of $L_{2\pi}$ to this function in the naive way, in particular this equality does not need to be well defined on the cone $\Nr x =0$.} $L_{2\pi}[M_\infty^{(m)}]=2\pi m M_\infty^{(m)}$. 

Fix a Schwartz function $M'\colon B_\infty\to\mathbb{C}$. We want to show now that $\langle L_{2\pi}[M_\infty^{(m)}],M'\rangle =\langle M_\infty^{(m)}, L_{2\pi}[M']\rangle$. The equality $\langle \Nr(x) M_\infty^{(m)},M' \rangle =\langle M_\infty^{(m)}, \Nr(x) M' \rangle$ is obvious. We need only show
\begin{equation}\label{eq:Delta-symmetry}
\langle\Delta M_\infty^{(m)},M' \rangle =\langle M_\infty^{(m)}, \Delta M' \rangle\;.
\end{equation}
Note that the integrals defining these individual inner-products are absolutely convergent because $M'$ and $\Delta M'$ are Schwartz, and $M_\infty^{(m)}$, $\Nr(x) M_\infty^{(m)}$ have at most polynomial growth. To establish \eqref{eq:Delta-symmetry}  we use the equality
\begin{equation*}
\langle \Delta M_\infty^{(m)},M' \rangle= \lim_{R\to\infty} \int_{V_R} \Delta F(x) \overline{M'(x)}\dif x\;,
\end{equation*}
and the analogous formula for $\langle  M_\infty^{(m)},\Delta M' \rangle$. These follow from the dominated convergence theorem. Denote $W=\begin{pmatrix} 0 & 1/2 & 0 & 0 \\ 1/2 & 0 & 0 & 0 \\ 0 & 0 & 0 & -1/2 \\ 0 & 0 & -1/2 & 0\end{pmatrix}$ and 
write $\Delta =\nabla \cdot (W \nabla)$ with respect to the coordinates $(a,d,b,c)$ . Using the symmetry of the matrix $W$ and the divergence theorem we arrive at 
\begin{align*}
 \int_{V_R} \Delta F(x) \overline{M'(x)}\dif x&=  \int_{V_R} \nabla \cdot \left(\overline{M'(x)}W \nabla F\right)\dif x-\int_{V_R} ( W\nabla F) \cdot (\nabla M') \dif x\\
 &=\int_{V_R} \nabla \cdot \left(\overline{M'(x)}W \nabla F- F W \nabla \overline{M'(x)} \right)\dif x
+\int_{V_R}  F(x)\Delta \overline{M'(x)}\dif x\\
&= \int_{\partial V_R}  \left(\overline{M'(x)}W \nabla F- F W \nabla \overline{M'(x)} \right) \cdot \hat{n} \dif A(x)
+\int_{V_R}  F(x)\Delta \overline{M'(x)}\dif x\;.
\end{align*}
The divegence theorem's conditions are satisfied because $F$ and $M$ are smooth in an open neighborhood of the compact set $V_R$, and the boundary $\partial V_R$ is piecewise smooth. A direct computation, as in the proof of Lemma \ref{lem:Bergman-decay}, shows that $F$ and $\nabla F$ vanish on the boundary of the cone $V$, except perhaps the origin where they remain bounded. 
It remains to consider the contributions from the surfaces $S_R=\partial V_R \cap B(0,R)$ and $s_R=\partial V_R\cap B(0,1/R)$. The area of $S_R$ is bounded from above by the area of a $3$-sphere of radius $R$, thus $\operatorname{Area}(S_R)\ll R^3$. On the other hand, because $M'$ is Schwartz and $F, \nabla F$ are bounded on $V$ we have that
\begin{equation*}
\left| \int_{S_R}  \left(\overline{M'(x)}W \nabla F- F W \nabla \overline{M'(x)} \right) \cdot \hat{n} \dif A(x) \right| \ll_N R^{-N} R^3 \to_{R\to\infty} 0\;.
\end{equation*}
Similarly, $\operatorname{Area}(s_R)\ll R^{-3}$ and  $\left|\overline{M'(x)}W \nabla F- F W \nabla \overline{M'(x)} \right|$ is uniformly bounded on a $B(0,1)\cap V$, hence the surface integral over $s_R$ converges to $0$ as $R\to\infty$.

Let now $M'\in V_{2\pi,m'}$ be a Schwartz function and assume $m'\neq m$, then we have $2\pi m\langle M_\infty^{(m)}, M' \rangle=\langle L_{2\pi}[M_\infty^{(m)}], M' \rangle = \langle M_\infty^{(m)}, L_{2\pi}[M']\rangle=2\pi m' \langle M_\infty^{(m)}, M' \rangle$ and we deduce $\langle M_\infty^{(m)}, M' \rangle=0$ as necessary.
\end{proof}

\begin{cor}\label{cor:Bergman-nice}
The Bergman test function of weight $m \ge 6$ belongs to $\Omega_\infty$.
\end{cor}
\begin{proof}
Lemma \ref{lem:Bergman-decay} implies that this test function satisfies the decay condition in the definition of $\Omega_\infty$ and Lemma \ref{lem:Bergman-isotypic} above implies that the Bergman test function transforms under $\rho(\mathbf{SO}_2(\mathbb{R}))$ by a character.
\end{proof}

\section{The Spectral Expansion}
\label{sec:specexp}
Fix a global Eichler order $\mathcal{R}=\mathcal{R}_1\cap\mathcal{R}_2\subset B$ and a weight $m>2$. We focus henceforth on the test function $M=M_\infty^{(m)}\cdot\prod_p \mathds{1}_{\mathcal{R}_p}\in\Omega$. This test function is cuspidal as $M_\infty^{(m)}(x)=0$ if $\Nr x=0$ and we denote the classical theta series attached to the test function $M$ by Proposition \ref{prop:theta-series-adelic-to-classical} by
\begin{equation*}
\vartheta^{(m)}(l,r;s)=\sum_{\xi\in \mathcal{R}} \left(\rho(s).M_\infty^{(m)}\right)(l^{-1}\xi r)
\end{equation*}
for $s\in \lfaktor{\Lambda}{\mathbf{SL}_2(\mathbb{R})}$, $l,r\in \lfaktor{\Gamma}{\mathbf{G}(\mathbb{R})}$.
In this section, we prove the main theorem about the spectral expansion of $\vartheta^{(m)}$.

\begin{defi}
Denote by $S_m(\Gamma)$ the space of $\Gamma$-modular weight $m$ modular forms on $\mathbb{H}$. Write $S_m(\Gamma)=S_m^\mathrm{old}(\Gamma)\oplus S_m^\mathrm{new}(\Gamma)$ for the direct sum decomposition into new and old forms. The decomposition is orthogonal with respect to the Petersson inner-product.

If $f\in S_m(\Gamma)$ we denote by $f^\sharp\colon \lfaktor{\Gamma}{\mathbf{G}(\mathbb{R})}\to\mathbb{C}$ the automorphic lift of $f$. Specifically, if $g=\sm 1 & x \\ 0 & 1 \esm \sm y^{1/2} & 0 \\ 0 & y^{-1/2} \esm k_\theta$, then $f^\sharp(g)=y^{m/2} e^{im\theta} f(x+iy)$. Following the discussion in \S\ref{sec:ThetaEichler} we shall also consider $f^\sharp$ as $K_\mathcal{R}$-invariant function on $[\mathbf{G}(\mathbb{A})]$.
Similarly, we decompose $S_m(\Lambda)=S_m^\mathrm{old}(\Lambda)\oplus S_m^\mathrm{new}(\Lambda)$, and denote by ${f^*}^\sharp\colon\lfaktor{\Lambda}{ \mathbf{SL}_2(\mathbb{R})}\to\mathbb{C}$ the automorphic lift of $f^*\in S_m(\Lambda)$. Moreover, we shall also consider ${f^*}^\sharp$ as a $U_\mathcal{R}$-invariant function on $[\mathbf{SL}_2(\mathbb{A})]$.
\end{defi}

\begin{thm}\label{thm:spectral-expansion}
Fix an orthonormal basis $\mathcal{B}_m$ of Hecke eigenforms for $S_m(\Gamma)$. 
Denote by $f_M^\sharp(s)$ the theta lift of $f^\sharp(g)$.
Then,
\begin{equation*}
\vartheta^{(m)}(l,r;s)= \frac{1}{\covol(\Gamma)} \frac{8\pi}{m-1} \sum_{f \in \mathcal{B}_m} f_M^\sharp(s)\overline{f^\sharp(l)}f^\sharp(r)
\end{equation*}
for all $s\in \lfaktor{\Lambda}{\mathbf{SL}_2(\mathbb{R})}$, $l,r\in \lfaktor{\Gamma}{\mathbf{G}(\mathbb{R})}$.

Let $\lambda_f(\alpha)$ be the $T_\alpha$-eigenvalue of $f^\sharp$, then $f_M^\sharp$ is the automorphic lift of a \emph{cusp} form $f_M\in S_m(\Lambda)$ with Fourier expansion
\begin{equation*}
f_M(\zeta)= \sum_{n>0} n^{m/2-1} \lambda_f(n) \exp(2\pi i m \zeta)\;.
\end{equation*}
\end{thm}
\begin{remark}
The operator $T_1$ acts as the identity on $K_\mathcal{R}$-invariant functions on $[\mathbf{G}(\mathbb{A})]$, i.e.\ functions on $\lfaktor{\Gamma}{\mathbf{G}(\mathbb{R})}$. Hence $\lambda_f(1)=1$ and the theta lift $f_M$ is an arithmetically normalized cusp form.
\end{remark}

The case where $\lfaktor{\Gamma}{\mathbf{G}(\mathbb{R})}=\lfaktor{\mathbf{SL}_2(\mathbb{Z})}{\mathbf{SL}_2(\mathbb{R})}$ is already contained in \cite[Section 2, Prop. 1]{ZagierDoiLecture}, see also Equation \eqref{eq:simplethetasl2Z}. For the general case, the proof will bootstrap from the fact that the convolution operator $\star_{\mathbf{G}(\mathbb{R})} M^{(m)}_{\infty}$ acting on $\lfaktor{\Gamma}{\mathbf{G}(\mathbb{R})}$ coincides with the Bergman kernel on $\lfaktor{\Gamma}{\mathbb{H}}$. Geometric expressions for the Bergman kernel in terms of Poincar\'e series were already known to Petersson \cite{Pet1,Pet2}. The particular expression for the Bergman kernel suitable for our needs may be found in either \cite{EichSelbZagier},\cite[Section 2, Prop. 1]{ZagierDoiLecture},\cite[Theorem 3]{Steiner16}, or \cite[Section 2.3]{DasSengupta}. The first three references each contain the split case and the latter the non-split case. There does, however, appear to be an error in the constant in \cite{DasSengupta}. Compare to the computation in \cite{EichSelbZagier,Steiner16}, whose proofs also apply to co-compact lattices. The statement is as follows.
\begin{prop}\label{prop:begman-spherical}
Set
\begin{equation*}
k^{(m)}(l,r)\coloneqq\sum_{\gamma\in \Gamma} \mu^m(l^{-1}\gamma r)\;.
\end{equation*}
The function $k^{(m)}$ is the kernel of the convolution operator $\star M_\infty^{(m)}$ acting on $L^2(\lfaktor{\Gamma}{\mathbf{G}(\mathbb{R})})$, where the convolution takes place in $\mathbf{G}(\mathbb{R})$.
Fix an orthonormal basis $\mathcal{B}_m$ for $S_m(\Gamma)$. Then, for all $l,r\in\mathbf{G}(\mathbb{R})$,
\begin{equation*}
k^{(m)}(l,r)=\frac{1}{\covol(\Gamma)} \frac{8\pi}{m-1}
\sum_{f \in \mathcal{B}_m} \overline{f^\sharp(l)}f^\sharp(r)\;.
\end{equation*}
In particular, the operator $\star M_\infty^{(m)}$ annihilates all the continuous, residual and cuspidal spectrum, whose archimedean component is not discrete series.
\end{prop}

\begin{proof}[Proof of Theorem \ref{thm:spectral-expansion}]
Let $\zeta=\sigma+i\tau \in \mathbb{H}$ and fix $s=\sm
1 & \sigma \\ 0 & 1
\esm \sm
\tau^{1/2} & 0 \\ 0 & \tau^{-1/2}
\esm k_\theta\in{\mathbf{SL}_2(\mathbb{R})}$. The definition of the Weil action, Definition \ref{defi:M_infty-bergman}, and Lemmata \ref{lem:k-action-2pi}, \ref{lem:Bergman-isotypic} imply for $n>0$, $g\in\mathbf{G}(\mathbb{R})$
\begin{equation}\label{eq:Bergmann-Weil-formula}
\rho(s).M_\infty^{(m)}(\sqrt{n}g)= \tau^{m/2} n^{m/2-1} \exp(2\pi i n \zeta+i m \theta)\mu^{m}(g)\;.
\end{equation}
We will establish that $\vartheta^{(m)}(l,r;s)$ coincides with the spectral expansion in the claim by showing equality in $L^2(\lfaktor{\Gamma}{\mathbf{G}(\mathbb{R})}\times \lfaktor{\Gamma}{\mathbf{G}(\mathbb{R})})$. Pointwise identity then follows because both sides are continuous.  

The Bergmann test function $M_\infty^{(m)}$ vanishes on the null-cone $\{x\in B_\infty \colon \Nr x=0\}$, thus it follows from Corollary \ref{cor:Weil-infinity-action} that $M$ is cuspidal. For any $\varphi,\varphi'\in L^2(\lfaktor{\Gamma}{\mathbf{G}(\mathbb{R})}) \cap L^\infty(\lfaktor{\Gamma}{\mathbf{G}(\mathbb{R})})$ we can use Proposition \ref{prop:theta-lift-whittaker} and \eqref{eq:Bergmann-Weil-formula} to write the Fourier expansion
\begin{align}
\label{eq:Begrmann-general-lift-expansion0}
\int_{[\mathbf{G}(\mathbb{A})]} \int_{[\mathbf{G}(\mathbb{A})]} \Theta_M(l,r;s) \varphi(l) \overline{\varphi'(r)} \dif l \dif r& = \sum_{n>0}  \left\langle (T_n^{M_f}\varphi) \star \left(\rho(s).M_\infty(\sqrt{n}\cdot\bullet)\right), \varphi' \right\rangle \\
&=\sum_{n>0} \tau^{m/2} n^{m/2-1} \exp(2\pi i n \zeta+i m \theta)\left\langle (T_n^{M_f}\varphi) \star \mu^m, \varphi' \right\rangle
\;.\nonumber
\end{align}
Because $T_n^{M_f}$ is a convolution operator, the maps $\varphi\mapsto T_n^{M_f}\varphi$, $\varphi\mapsto (T_n^{M_f}\varphi) \star \mu^m$ preserve the continuous and the discrete spectra. Proposition \ref{prop:begman-spherical} then implies that \eqref{eq:Begrmann-general-lift-expansion0} vanishes whenever $\varphi$ or $\varphi'$ is a bounded function in the continuous spectrum. Using pseudo-Eisenstein series we can construct a dense set of bounded vectors in the continuous spectrum of $L^2([\mathbf{G}(\mathbb{A})])$, hence $\Theta_M(l,r;s)\in L^2_\mathrm{discrete}([\mathbf{G}(\mathbb{A})]\times [\mathbf{G}(\mathbb{A})])$. Moreover, $\Theta_M(l,r;s)$ is $K_{\mathcal{R}}\times K_{\mathcal{R}}$-invariant. There is an orthonormal basis of $L^2_\mathrm{discrete}([\mathbf{G}(\mathbb{A})])^{K_{\mathcal{R}}}$ consisting of bounded Hecke eigenforms. If $\varphi$ is a bounded Hecke eigenform  with eigenvalues $\lambda(\bullet)$, we can use Corollary \ref{cor:Whittaker-theta-product-Hecke} to rewrite \eqref{eq:Begrmann-general-lift-expansion0} as 
\begin{equation}
\label{eq:Begrmann-general-lift-expansion}
\int_{[\mathbf{G}(\mathbb{A})]} \int_{[\mathbf{G}(\mathbb{A})]} \Theta_M(l,r;s) \varphi(l) \overline{\varphi'(r)} \dif l \dif r=\sum_{n>0} \lambda(n) \tau^{m/2} n^{m/2-1} \exp(2\pi i n \zeta+i m \theta)\left\langle \varphi \star \mu^m, \varphi' \right\rangle
\;.
\end{equation}
Proposition \ref{prop:begman-spherical} immediately implies that the expression above vanishes unless both $\varphi$ and $\varphi'$ are lifts of weight $m$ modular forms. The claimed spectral expansion follows from Proposition \ref{prop:begman-spherical}, because the automorphic lifts of $\mathcal{B}_m(\Gamma)\times \mathcal{B}_m(\Gamma)$ can be completed to an orthogonal basis of $L^2_\mathrm{discrete}([\mathbf{G}(\mathbb{A})]\times [\mathbf{G}(\mathbb{A})])$ consisting of bounded Hecke eigenforms.
The formula for the Fourier--Whittaker expansion of $f_M$ follows from \eqref{eq:Begrmann-general-lift-expansion} with $\varphi=\varphi'=f^\sharp$.
\end{proof}

A careful local analysis, following Shimizu \cite{Shimizu} shows that if $f$ is a newform then $f_M(s)$ is the unique arithmetically normalized new modular form in the Jacquet--Langlands transfer of the automorphic representation generated by $\widetilde{f^\sharp}$ (to be defined momentarily). We will need only a weaker result. To discuss the Jacquet--Langlands transfer we need to lift functions from $[\mathbf{SL}_2(\mathbb{A})]$ to $[\mathbf{PGL}_2(\mathbb{A})]$. Define $\widetilde{U}_{\mathcal{R}}$ to be the image of  $\left\{g \in \sm \widehat{\mathbb{Z}} & \widehat{\mathbb{Z}} \\ q D_B\widehat{\mathbb{Z}} & \widehat{\mathbb{Z}} \esm \colon \det g\in \widehat{\mathbb{Z}}^\times  \right\}$ in $\mathbf{PGL}_2(\mathbb{A}_f)$. Then $\widetilde{U}_{\mathcal{R}}$ is a compact open subgroup, and an argument identical to Corollary \ref{cor:G-Gadj} proves that
\begin{equation*}
\dfaktor{\mathbf{SL}_2(\mathbb{Q})}{\mathbf{SL}_2(\mathbb{A})}{U_{\mathcal{R}}}\to \dfaktor{\mathbf{PGL}_2(\mathbb{Q})}{\mathbf{PGL}_2(\mathbb{A})}{\widetilde{U}_{\mathcal{R}}}
\end{equation*}
is a measure preserving bijection. Hence, we have a unique lift $\varphi\mapsto\widetilde{\varphi}$ from $L^2([\mathbf{SL}_2(\mathbb{A})])^{U_{\mathcal{R}}}$ to $L^2([\mathbf{PGL}_2(\mathbb{A})])^{\widetilde{U}_{\mathcal{R}}}$.
\begin{defi}
Let $f\in S_m(\Gamma)$ be a Hecke eigenform. If $f$ is a newform denote by $f_\mathrm{JL}\in S_m^\mathrm{new}(\Lambda)$ the unique arithmetically normalized holomorphic newform such that $\widetilde{f_\mathrm{JL}^\sharp}$ belongs to the Jacquet--Langlands transfer of the automorphic representation generated by $\widetilde{f^\sharp}$. That such a vector exists and is unique follows from \cite{JacquetLanglands,Shimizu}.  If $f$ is an oldform then it factors through a newform with respect to a lattice arising from an Eichler order $\mathcal{R}' \supsetneq \mathcal{R}$ with level $q'\mid q$. In this case, we denote by $f_\mathrm{JL}$ the lift of the Jacquet--Langlands transfer, defined as above, from $S_m^\mathrm{new}(\Gamma_0(q'D_B))$ to $S_m^\mathrm{old}(\Lambda=\Gamma_0(q D_B))$.

In both cases, the modular form $f_{\mathrm{JL}}$ is an eigenform of all the classical Hecke operators corresponding to $n$ co-prime to $q D_B$, and its $n$-Fourier coefficient coincides with the $T_n$ Hecke eigenvalue of $f^\sharp$ if $\gcd(n, q D_B)=1$. 
\end{defi}

\begin{lem}\label{lem:f_M-f_JL}
Let $f\in S_m(\Gamma)$ be a Hecke eigenform. If $f$ is a newform then the orthogonal projection of $f_M$ 
onto $S_m^\mathrm{new}(\Lambda)$ is equal to $f_\mathrm{JL}$. If $f$ is an oldform then $f_M$ is an oldform as well.
\end{lem}
\begin{proof}
Theorem \ref{thm:spectral-expansion} implies that the Fourier coefficients of $f_M$ and $f_{\mathrm{JL}}$ coincide for all $n$ co-prime to $q D_B$, which is the level of $\lfaktor{\Lambda}{\mathbb{H}}$. Theorem 1 of \cite{AtkinLehner} then implies that $f_M-f_{\mathrm{JL}}$ is an oldform. Because $f_\mathrm{JL}$ is a newform if $f$ is, the claim holds for newforms. If $f$ is an oldform then so is $f_\mathrm{JL}$. Hence in this case $f_M$ is a sum of oldforms and is an oldform by itself.
\end{proof}

\begin{cor}\label{cor:spectral-lower-bound}
Let $l,r\in\mathbf{G}(\mathbb{R})$ and set $z_1=l.i,\, z_2=r.i\in\mathbb{H}$.
Fix an orthonormal basis $\mathcal{B}_m^\mathrm{new}$ of Hecke eigenforms for $S_m^\mathrm{new}(\Gamma)$. Then, 
\begin{multline*}
\frac{1}{\covol(\Lambda)}\|\vartheta^{(m)}(l,r;\bullet)\|_{L^2(\lfaktor{\Lambda}{\mathbf{SL}_2(\mathbb{R})})}^2\geq
\left(\frac{1}{\covol(\Gamma)}\frac{8\pi}{m-1}\right)^2
(\Im (z_1) \Im (z_2))^{m}\sum_{f \in \mathcal{B}_m^\mathrm{new}} \|f_M\|_2^2 |f(z_1)|^2 |f(z_2)|^2\\
\gg_\varepsilon \frac{1}{\covol(\Gamma)^2} (mqD_B)^{-\varepsilon       }\frac{\Gamma(m)}{(4\pi)^m m^2}
(\Im (z_1) \Im (z_2))^{m}\sum_{f \in \mathcal{B}_m^\mathrm{new}}  |f(z_1)|^2 |f(z_2)|^2\;.
\end{multline*}
\end{cor}
\begin{proof}
Define $\vartheta_\sharp(\zeta)=\tau^{-m/2}\vartheta^{(m)}(l,r;s)$, where $\zeta=\sigma+i\tau \in \mathbb{H}$ and $s=\sm 1 & \sigma \\ 0 & 1 \esm \sm \tau^{1/2} & 0 \\ 0 & \tau^{-1/2} \esm$. Because $\vartheta^{(m)}(l,r;\bullet)$ is $K_\infty$-isotypic we have ${\vartheta_\sharp}^\sharp(s)=\vartheta^{(m)}(l,r;s)$.
Theorem \ref{thm:spectral-expansion} and Lemma \ref{lem:f_M-f_JL} above imply that the orthogonal project of $\vartheta_\sharp$ onto $S_m^\mathrm{new}(\Lambda)$ is equal to
\begin{equation}\label{eq:theta-new-projection}
\vartheta_\sharp^\mathrm{new}(\zeta)=\frac{1}{\covol(\Gamma)}\frac{8\pi}{m-1} \sum_{f \in \mathcal{B}_m^\mathrm{new}} f_\mathrm{JL}(\zeta)\overline{f^\sharp(l)}f^\sharp(r)\;.
\end{equation}
 
Because oldforms are orthogonal to newforms $\|\vartheta^{(m)}(l,r;\bullet)\|_{L^2(\lfaktor{\Lambda}{\mathbf{SL}_2(\mathbb{R})})}^2\geq \|\vartheta_\sharp^\mathrm{new}\|_2^2$. The first claimed inequality follows from \eqref{eq:theta-new-projection} and the orthogonality relations of Hecke eigenforms. The second inequality follows from the Hoffstein--Lockhart \cite{HoffsteinLockhart} bound on the $L^2$-norm of an arithmetically normalized holomorphic Hecke newform $f$ of level $N$
\begin{equation*}
\|f\|_2^2\gg_\varepsilon \frac{\Gamma(m)}{(4\pi)^m} (mN)^{-\varepsilon}\;.
\end{equation*}
This bound holds when the Petersson inner-product is normalized with respect to the probability measure on $\lfaktor{\Lambda}{\mathbf{SL}_2(\mathbb{R})}$.
\end{proof}

At this point, we shall note that we have also proven Theorem \ref{thm:individualtheta}. Indeed, if we lift adelically $f \in S_m(\Gamma)$ to $f^{\sharp}$, then we find
\begin{equation*}
\Im(z)^{\frac{m}{2}} \Im(\zeta)^{\frac{m}{2}} \mathcal{F}_{f}(z;\zeta) = c\sum_{n>0} T_{n}^{M_f}f^{\sharp} \star \left(\rho(s_\infty).M_\infty(\sqrt{n} \cdot \bullet)\right)(r_{\infty}) = c\int_{[\mathbf{G}(\mathbb{A})]} \Theta_M(l,r;s) f^{\sharp}(l) \dif l
\label{eq:indivtheta}
\end{equation*}
by Propositions \ref{prop:theta-lift-whittaker}, \ref{prop:theta-hecke}, and \ref{prop:begman-spherical}, where $c=\covol(\Gamma)(m-1)/(8\pi)$ and $r_\infty=\sm 1 & x \\ & 1 \esm \sm y^{1/2}  & \\ & y^{-1/2}\esm$ and $r_p$ being the identity for all finite places, and similarly for $s$ (with respect to $\zeta$). Thus, $\mathcal{F}_{f}(z;\zeta)$ is the classical holomorphic modular form associated to $\int \Theta_M(l,r;s)f^{\sharp}(l) \dif l$, from which the theorem follows.

\section{The Geometric Expansion}
\label{sec:L2reduction}
We have now established in Corollary \ref{cor:spectral-lower-bound} a lower bound on $\|\vartheta^{(m)}(l,r;\bullet)\|_{L^2(\lfaktor{\Lambda}{\mathbf{SL}_2(\mathbb{R})})}^2$ in terms of a fourth moment of Hecke eigenforms of weight $m$. In this section, our goal is to establish an upper bound in terms of a count of quaternions by norm. In the next section, we will establish a sharp upper-bound for this count.

\begin{defi}
For $g\in \mathbf{PGL}_2(\mathbb{R})$, define
\begin{equation*}
u(g)=\frac{\Tr (g\tensor[^t]{g}{})-2|\det g|}{4|\det g|}\;.
\end{equation*}
Specifically, for $g=\sm
a & b \\ c & d
\esm$
\begin{equation*}
4 u(g)=\frac{a^2+b^2+c^2+d^2}{|ad-bc|}-2\;.
\end{equation*}
Using the fixed isomorphism $B\otimes\mathbb{R}\simeq \operatorname{Mat}_{2\times 2}(\mathbb{R})$, we extend the function $u$ to the group $(B\otimes\mathbb{R})^\times$.
\end{defi}

\begin{lem}\label{lem:mu-u}
For all $g\in \mathbf{PGL}_2(\mathbb{R})$ with $\det g>0$,
\begin{equation*}
|\mu(g)|^2=(1+u(g))^{-1}\;.
\end{equation*}
\end{lem}
\begin{proof}
Fix  $g=\sm
a & b \\ c & d
\esm$ with $\det g >0$. We deduce from Definition \ref{defi:M_infty-bergman} that
\begin{equation*}
|\mu(g)|^{-2}=\frac{(b-c)^2+(a+d)^2}{4\det g}
=\frac{a^2+b^2+c^2+d^2+2\det g}{4\det g}=1+u(g)
\;.
\end{equation*}
\end{proof}

\begin{prop}\label{prop:geometric-upper-bound}
\begin{align*}
\frac{1}{\covol(\Lambda)}\|\vartheta^{(m)}(l,r;\bullet)\|_{L^2(\lfaktor{\Lambda}{\mathbf{SL}_2(\mathbb{R})})}^2 \ll_\varepsilon (qD_B)^{1+\varepsilon} \frac{\Gamma(m-1)}{(4\pi)^m}
\sum_{n>0}& \frac{1}{n}\left( \sum_{\substack{\xi\in R \\ \Nr \xi =n}} \left(1+u(l^{-1}\xi r)\right)^{-m/2}\right)^2
\\
&\cdot\begin{cases}
1, & n<(qD_B)^2 m,\\
\exp(-n/(qD_B)^2), & n>(q D_B)^2 m.
\end{cases}
\end{align*}
\end{prop}

\begin{proof}
We first apply Proposition \ref{prop:L2-s-Siegel-upper-bound} to $\vartheta^{(m)}$ and use the fact that our choice of $M_\infty=M_\infty^{(m)}$ is $K$-isotypical and transforms simply under the diagonal group.
\begin{align*}
\frac{1}{\covol(\Lambda)}\int_{{\Lambda}\backslash{\mathbf{SL}_2(\mathbb{R})}}  |\vartheta^{(m)}(l, r;s)|^2 \dif s
\leq   &(q D_B)^{-1}\sum_{a\mid q D_B} \sum_{\substack{\alpha\in\mathbb{Q}\\\alpha>0}}\\
&\int_{\sqrt{3}/2}^\infty
\tau^{m}\alpha^{m-2}\exp(-4\pi \alpha \tau) \left( \sum_{\substack{\xi\in\widehat{\mathcal{R}}^{(a)}\\\Nr \xi=\alpha}}\left|
\mu^m(l^{-1} \xi r)
\right| \right)^2  \frac{\dif \tau}{\tau^2}\;.
\end{align*}
We bound the integral over $\tau$, which is equivalent to the definition of the partial gamma function, in two ways. Write first
\begin{equation*}
\int_{\sqrt{3}/2}^\infty \tau^{m-2} \alpha^{m-2} \exp(-4\pi\alpha \tau) \dif \tau=
\frac{1}{\alpha} \int_{\sqrt{3}/2 \cdot \alpha}^\infty x^{m-2} \exp(-4\pi x) \dif x\;.
\end{equation*}
For $\alpha\leq m$, we complete the integral to deduce
\begin{equation*}
\int_{\sqrt{3}/2 \cdot \alpha}^{\infty} x^{m-2} \exp(-4\pi x) \dif x
\leq (4\pi)^{-(m-1)}\Gamma(m-1)\;.
\end{equation*}
For $\alpha>m$, we argue
\begin{align*}
\int_{\sqrt{3}/2 \cdot \alpha}^{\infty} &x^{m-2} \exp(-4\pi x) \dif x
\leq \exp(-2\pi \sqrt{3}/2 \cdot \alpha) 2^{m-2} \int_{\sqrt{3}/2 \cdot \alpha}^{\infty} (x/2)^{m-2} \exp(-4\pi (x/2)) \dif x\\
&\leq \exp(-2\pi \sqrt{3}/2 \cdot \alpha) 2^{m-1}  (4\pi)^{-(m-1)} \Gamma(m-1)\\
&\ll \exp\left(-\left(2\pi \sqrt{3}/2-\log 2\right) \cdot \alpha\right)  (4\pi)^{-(m-1)} \Gamma(m-1)
\leq \exp(-\alpha) (4\pi)^{-(m-1)} \Gamma(m-1)\;.
\end{align*}
We thus arrive at
\begin{align*}
\int_{\sqrt{3}/2}^\infty
\tau^{m}\alpha^{m-2}\exp(-4\pi \alpha \tau) \left( \sum_{\substack{\xi\in\widehat{\mathcal{R}}^{(a)}\\\Nr \xi=\alpha}}\left|
\mu^m(l^{-1} \xi r)
\right|  \right)^2  \frac{\dif \tau}{\tau^2}
\ll (4\pi)^{-m}\Gamma(m-1) \frac{1}{\alpha} &\left(\sum_{\substack{\xi\in\widehat{\mathcal{R}}^{(a)}\\\Nr \xi=\alpha}}|\mu^{m}(l^{-1}\xi r)|\right)^2\\
&\cdot
\begin{cases}
1, & \alpha\leq m,\\
\exp(-\alpha), & \alpha>m.
\end{cases}
\end{align*}
Note that $\widehat{\mathcal{R}}^{(a)}\subset (qD_B)^{-1}\mathcal{R}$ and
\begin{equation*}
\left(\sum_{\substack{\xi\in\widehat{\mathcal{R}}^{(a)}\\\Nr \xi=\alpha}}|\mu^{m}(l^{-1}\xi r)|\right)^2
\leq \left(\sum_{\substack{\xi\in \mathcal{R}\\ \Nr \xi =(qD_B)^2\alpha}}|\mu^{m}(l^{-1}\xi r)|\right)^2
\;.
\end{equation*}
The claimed bound follows from combining these inequalities with the divisor bound and Lemma \ref{lem:mu-u} above.
\end{proof}

\begin{defi}
For any $g\in \mathbf{G}(\mathbb{R})$, $n\in\mathbb{N}$, and $\delta>0$, set
\begin{equation*}
M(g,n;\delta)\coloneqq \left\{ \xi\in\mathcal{R} \mid \Nr \xi=n\;, u(g^{-1}\xi g) < \delta
\right\}\;.
\end{equation*}
\end{defi}

\begin{cor}\label{cor:geometric-upper-bound}
If $m>2$ then
\begin{align*}
\frac{1}{\covol(\Lambda)}\|\vartheta^{(m)}(g,g;\bullet)\|_{L^2(\lfaktor{\Lambda}{\mathbf{SL}_2(\mathbb{R})})}^2 \ll_\varepsilon &(qD_B)^{1+\varepsilon} \frac{\Gamma(m-1)}{(4\pi)^m}
\Bigg\{\frac{m}{2}\int_0^\infty \left[\sum_{n=1}^{(qD_B)^2m} \frac{1}{n}M(g,n;\delta)^2\right]^{1/2}\\
+&\left[\sum_{n>(qD_B)^2m}^\infty \frac{\exp(-n/(qD_B)^2)}{n} M(g,n;\delta)^2 \right]^{1/2}
\frac{\dif\delta}{(1+\delta)^{m/2+1}}\Bigg\}^2\;.
\end{align*}
\end{cor}
\begin{proof}
Integration by parts for Riemann--Stieltjes integrals implies
\begin{align}
\sum_{\substack{\xi\in R \\ \Nr \xi =n}} \left(1+u(g^{-1}\xi g)\right)^{-m/2}
&=\int_{0}^{\infty} (1+\delta)^{-m/2}\dif M(g,n;\delta) \nonumber\\
&=\lim_{\delta\to\infty}M(g,n;\delta)(1+\delta)^{-m/2}
+\frac{m}{2}\int_0^\infty (1+\delta)^{-m/2-1} M(g,n;\delta) \dif \delta\;.
\label{eq:integration-by-parts}
\end{align}
The left-hand side is exactly the expression we need to bound in Proposition \ref{prop:geometric-upper-bound}.
Iwaniec and Sarnak in \cite[Lemma 1.3 and Appendix 1]{IwaniecSarnak} establish the bound
\begin{equation}\label{eq:IS-M-bound}
M(g,n;\delta)\ll_{\varepsilon,q,D_B,g} n^{\varepsilon}+(\delta+\delta^{1/4})n^{1+\varepsilon}\;.
\end{equation}
Thus, the first term in \eqref{eq:integration-by-parts} vanishes for $m>2$. Set $w_n=1/n$ if $n\leq (qD_B)^2m$ and $w_n=\exp(-n/(qD_B)^2)/n$ otherwise.
We apply Minkowski's integral inequality to deduce
\begin{equation*}
\sum_{n=1}^{\infty} w_n \left|\sum_{\substack{\xi\in\mathcal{R} \\ \Nr \xi=n}} (1+u(g^{-1}\xi g))^{-m/2}\right|^2
\leq \left(\frac{m}{2}\int_0^\infty \sqrt{\sum_{n=1}^\infty w_n M(g,n;\delta)^2}
\frac{\dif\delta}{(1+\delta)^{m/2+1}}\right)^2\;.
\end{equation*}
The claim follows by splitting the sum into two ranges: $1\leq n \leq (qD_B)^2m$ and $n>(qD_B)^2m$ and using the $l^2$--$l^1$ inequality.
\end{proof}

\section{Second Moment Count of Quaternions by Norm}
\label{sec:matrix}
In this section, we prove our main results about the second moment count of quaternions by norm  in a small ball. This bound in combination with the results of the previous sections will lead to the proof of Theorem \ref{thm:main}. To bound $\sum_{n=1}^N M(g,n;\delta)^2$ we can assume henceforth without loss of generality that $\mathcal{R}$ is a maximal order, otherwise we can replace the Eichler order $\mathcal{R}=\mathcal{R}_1\cap\mathcal{R}_2$ by $\mathcal{R}_1$ and the second moment sum will only increase.

We shall deal separately with the split case $\mathbf{G}=\mathbf{SL}_2$ and the case of anisotropic $\mathbf{G}$. The proof in both cases is very similar except that we need to track the dependence on $g$ differently. While in the split case we shall work with the Iwasawa decomposition of $g$, in the anisotropic case we will use an adapted Cartan decomposition of $g$.

\subsection{Second Moment Bound for the Split Matrix Algebra}\label{sec:second-moment-iwasawa}
In this section, we fix $\mathbf{G}=\mathbf{SL}_2$, i.e.\ $B=\operatorname{Mat}_{2\times 2}(\mathbb{Q})$ and $\mathcal{R} =  \operatorname{Mat}_{2\times 2}(\mathbb{Z})$. If we write in coordinates $\xi =\begin{pmatrix}
a & b \\c & d
\end{pmatrix}$, then the inequalities $u(\xi)<\delta$, $0<\det \xi <N$ imply
\begin{align}
\label{eq:u-det-sp-1}
a^2+b^2+c^2+d^2&<N(4+2\delta)\;,\\
\label{eq:u-det-sp-2}
(a-d)^2+(b+c)^2&< 4N \delta\;.
\end{align}

For $g\in\mathbf{G}(\mathbb{R})$ write $g=n a k$ with $k\in\mathbf{SO}_2(\mathbb{R})$ and
\begin{align*}
n&=\begin{pmatrix}
1 & x \\ 0 & 1
\end{pmatrix}\\
a&=\begin{pmatrix}
y^{1/2} & 0 \\ 0 & y^{-1/2}
\end{pmatrix}\;.
\end{align*}
This is the standard Iwasawa decomposition of $g$. 

\begin{prop}\label{prop:second-moment-sp}
Let $g\in\mathbf{SL}_2(\mathbb{R})$ and write $g.i=x+iy$. Assume $|x| \le C$ and $y \ge A>0$. Then,
\begin{equation*}
\sum_{n=1}^N M(g,n;\delta)^2\ll_{A,C,\varepsilon} N^{3+\varepsilon}\delta^2+N +N^{1/2+\varepsilon}\min(N^{1/2},(N\delta)^{1/2}+1)(y^2N\delta+1) \;.
\end{equation*}
\end{prop}

\begin{remark} In the end, we may restrict ourselves to $g$ in a fundamental domain for $\mathbf{SL}_2(\mathbb{Z})$ and hence the restrictions on $x,y$ will be satisfied.
\end{remark}

Using the inequality of geometric and arithmetic means we can split the second-moment count into two cases. The first one is when both matrices are upper triangular and the second one is when neither one is. We now prepare some preliminary results needed in the proof of Proposition \ref{prop:second-moment-sp}.

\begin{lem}\label{lem:secon-moment-u}
Denote by $\mathfrak{U}\subset B_\infty$ the subset of upper triangular matrices. Then,
\begin{align*}
\#\Big\{(\xi_1,\xi_2) \in \left(\operatorname{Mat}_{2\times 2}(\mathbb{Z})\cap \mathfrak{U}\right)^{2}&\colon u(g^{-1}\xi_1g),u(g^{-1}\xi_2g)<\delta\;, 0<\det \xi_1=\det\xi_2 < N  \Big\}\\
&\ll_\varepsilon N^{1/2+\varepsilon}\min(N^{1/2},(N\delta)^{1/2}+1)(y^2N\delta+1)\;.
\end{align*}
\end{lem}
\begin{proof}
Write $\xi_i= \sm a_i & b_i \\ 0 & d_i \esm$.
For upper triangular matrices we rewrite \eqref{eq:u-det-sp-2} for $g^{-1}\xi_ig$ as
\begin{equation*}
(a_i-d_i)^2+\frac{\left(b_i+2x(a_i-d_i)\right)^2}{y^2}<4 N \delta\;.
\end{equation*}
Hence, we have at most $\ll (N\delta)^{1/2}+1$ choices for $|a_1-d_1|$.
In addition, the condition $0<4\det \xi_1=(a_1+d_1)^2-(a_1-d_1)^2<4N$ implies $0<|a_1+d_1|-|a_1-d_1|\ll N^{1/2}$. We deduce that there are at most $\ll ((N\delta)^{1/2}+1)N^{1/2}$ possibilities for $(|a_1+d_1|,|a_1-d_1|)$ and a similar statement holds for $(a_1,d_1)$.

On the other hand $0<a_1 d_1<N$ and the divisor bound implies that the number of possible pairs $(a_1,d_1)$ is also bounded by $\ll_\varepsilon N^{1+\varepsilon}$. The number of possibilities for $b_1$ is now bounded above by $y(N\delta)^{1/2}+1$. Thus, there are at most $\ll_\varepsilon N^{1/2+\varepsilon}\min\left(N^{1/2},(N\delta)^{1/2}+1\right)(y(N\delta)^{1/2}+1)$ possibilities for $\xi_1$.

Once $\xi_1$ is fixed the condition $\det\xi_1=\det\xi_2>0$ fixes $a_2 d_2$ and the divisor bound restricts the number of possible pairs $(a_2,d_2)$ to $\ll_\varepsilon N^{\varepsilon}$. At last, the number of possible $b_2$'s after fixing $(a_2,d_2)$ is at most $\ll y(N\delta)^{1/2}+1$. 
\end{proof}

We continue to analyze the case when neither matrix is upper triangular. We will use 
the direct sum decomposition $B_\infty = \mathbb{R}\Id + B_\infty^0$. This decomposition is preserved by the conjugation action. We denote by $\xi^0=\xi-\frac{1}{2}\Tr \xi$ the traceless part of $\xi\in B_\infty$. In coordinates we write
\begin{equation*}
\xi^0=\begin{pmatrix}
e & b \\ c & -e
\end{pmatrix}
\end{equation*}
where $e=\frac{a-d}{2}$.
If $\xi$ satisfies \eqref{eq:u-det-sp-1} and \eqref{eq:u-det-sp-2} then $\xi^0$ satisfies $2 e^2+b^2+c^2< N(4+2\delta)$. This leads us to define
\begin{align*}
B_\infty^0\supset X&\coloneqq\left\{\xi^0=\begin{pmatrix}
e & b \\ c & -e
\end{pmatrix} \colon 4e^2+(b+c)^2<4 N \delta,\; 2e^2+b^2+c^2<N(4+2\delta) \right\}\;,\\
B_\infty^0\supset X^g&\coloneq g X g^{-1}\;.
\end{align*}

The set $X$ is invariant under conjugation by $K$ and using the Iwasawa decomposition we can write the for $\sm e & b \\ c & -e \esm \in X^g$ explicitly as
\begin{align}
\label{eq:X^g-sp-1}
2(e-xc)^2+\frac{(b+2xe-x^2c)^2}{y^2}+y^2c^2&<N(4+2\delta)\\
\label{eq:X^g-sp-2}
4(e-xc)^2+\frac{\left(b+2x(e-xc)+(x^2+y^2)c\right)^2}{y^2}&<4N\delta
\end{align}

\begin{lem}\label{lem:traceless-count-sp}
Assume $y\geq A>0$ and $0<\delta \le 1$, then 
\begin{equation*}
\#\left(\tfrac{1}{2}\operatorname{Mat}_{2\times 2}(\mathbb{Z})^0 \cap X^g\setminus\mathfrak{U}\right)\ll  
N^{3/2}\delta+N^{1/2}\;. 
\end{equation*}
\end{lem}
\begin{proof}
From equation \eqref{eq:X^g-sp-1} we learn that there are $\ll y^{-1}N^{1/2}(1+\delta)$ options for $c$ ($c\neq 0$ because the matrices are not upper triangular). For any fixed $c$, equation \eqref{eq:X^g-sp-2} describes an ellipse in the $e,b$ plane with radii $\ll \sqrt{N\delta}, y\sqrt{N\delta}$. Hence, the number of possibilities for $(e,b)\in \frac{1}{2}\mathbb{Z}\times \frac{1}{2}\mathbb{Z}$ is bounded from above by
\begin{equation*}
\ll y N\delta +(N\delta)^{1/2}\max(y,1)+1\ll_A y \left(N\delta + (N\delta)^{1/2}+1\right)\;.
\end{equation*}
Multiplying this by the bound for possible values of $c$ and the inequality $N\delta^{1/2}<\max(N^{3/2}\delta,N^{1/2})$ establish the claim.
\end{proof}

\begin{lem} \label{lem:eq-det-traceless}
Assume $y \ge A>0$, $|x|<C$, and $0<\delta \le 1$. Then,
\begin{equation*}
\#\left\{(\xi_1^0,\xi_2^0)\in \left(\tfrac{1}{2}\operatorname{Mat}_{2\times 2}(\mathbb{Z})^0 \cap X^g\setminus\mathfrak{U}\right)^{2} \colon \det \xi_1^0=\det \xi_2^0
\right\}\ll_{A,C} N^{5/2}\delta^2+N^{1/2}\;. 
\end{equation*}
\end{lem}
\begin{proof}
Note that the number of possible elements $\xi_1^0$ is bounded by Lemma \ref{lem:traceless-count-sp} above. 
We fix henceforth $\xi_1^0$ as in the claim and count the number of possible $\xi_2^0\not\in \mathfrak{U}$ with $\det\xi_2^0=\det\xi_1^0$. 
Denote
\begin{equation*}
g^{-1}\xi_i^0 g=\begin{pmatrix}
\tilde{e}_i & \tilde{b}_i \\ \tilde{c}_i & -\tilde{e}_i
\end{pmatrix}\;.
\end{equation*}
We now rewrite equation \eqref{eq:u-det-sp-2} for $\xi_i^0$ as
\begin{equation}\label{eq:X^g-sp-2-shifted}
4(e_i-c_ix)^2+\left((\tilde{b}_i-\tilde{c}_i)+2yc_i\right)^2<4N\delta\;.
\end{equation}
Then, $(\tilde{b}_1-\tilde{c}_1)$ is restricted to an interval of length $\ll \sqrt{N\delta}$.
Equation \eqref{eq:u-det-sp-2} implies 
\begin{equation*}
(\tilde{b}_i+\tilde{c}_i)^2 +4 \tilde{e_i}^2= (\tilde{b}_i-\tilde{c}_i)^2-4\det(\xi_i^0)=(\tilde{b}_i-\tilde{c}_i)^2-4\det(g^{-1}\xi_i^0g)\in[0,4N\delta]
\end{equation*}
and $\left|(\tilde{b_1}-\tilde{c}_1)^2-(\tilde{b_2}-\tilde{c}_2)^2\right|\ll N\delta$. We deduce that $\left||\tilde{b_1}-\tilde{c}_1|-|\tilde{b_2}-\tilde{c}_2|\right|\ll \sqrt{N\delta}$. In particular, $(\tilde{b}_2-\tilde{c}_2)$ is restricted to two intervals of length $\ll \sqrt{N\delta}$. 

Consider inequality \eqref{eq:X^g-sp-2-shifted} for $\xi_2^0$ with $(\tilde{b}_2-\tilde{c}_2)$ as a varying quantity in the aforementioned intervals, thus ignoring their dependencies on $e_2,c_2$. It describes  an ellipse in the variables $e_2,c_2$ with center $-\frac{\tilde{b}_2-\tilde{c}_2}{2y}\cdot(x,1)$. Because $\xi_1^0$ is fixed the center of the ellipse is restricted to one of two intervals of length $\ll_{C} y^{-1} \sqrt{N\delta}$. The radii of the ellipse satisfy $\ll_{C} \sqrt{N\delta}, y^{-1}\sqrt{N\delta}$. We deduce that the number of possibilities for $(e_2,c_2)$ is $\ll_{A,C} y^{-1}N\delta+(N\delta)^{1/2}+1\ll_{A,C} N\delta+1$. Once $\xi_1^0$, $e_2$ and $c_2\neq0$ are fixed the value of $b_2$ is fixed by the equality $\det \xi_2^0=\det\xi_1^0$. Hence, the total number of pairs $(\xi_1^0,\xi_2^0)$ is bounded from above in this case by
\begin{equation*}
\ll_{A,C} (N^{3/2}\delta+N^{1/2})(N\delta+1)\ll_{A,C}  N^{5/2}\delta^2+N^{3/2}\delta+N^{1/2} \ll_{A,C}  N^{5/2}\delta^2+N^{1/2} \;.
\end{equation*} 
\end{proof}

\begin{proof}[Proof of Proposition \ref{prop:second-moment-sp}]
Define
\begin{align*}
M_\star(g,n;\delta)&\coloneqq \left\{ \xi\in\operatorname{Mat}_{2\times 2}(\mathbb{Z})\setminus\mathfrak{U} \mid \det \xi=n\;, u(g^{-1}\xi g) < \delta
\right\}\;,\\
M_u(g,n;\delta)&\coloneqq \left\{ \xi\in\operatorname{Mat}_{2\times 2}(\mathbb{Z})\cap\mathfrak{U} \mid \det \xi=n\;, u(g^{-1}\xi g) < \delta
\right\}\;.\\
\end{align*}
Then, the inequality of means imply
\begin{equation*}
\sum_{n=1}^N M(g,n;\delta)^2\leq 2\sum_{n=1}^N M_\star(g,n;\delta)^2+2\sum_{n=1}^N M_u(g,n;\delta)^2
\end{equation*}
and we turn to bounding each term individually. The second term is controlled by Lemma \ref{lem:secon-moment-u} and is consistent with the claim.

To bound the first term we need to bound the number of pairs $(\xi_1,\xi_2)\in \left(\operatorname{Mat}_{2\times 2}(\mathbb{Z})\setminus\mathfrak{U}\right)^{2}$ such that $0<\det\xi_1=\det\xi_2<N$ and $u(g^{-1}\xi_i g)<\delta$ for $i=1,2$.
Assume first $\delta>1$, we then argue as in \cite{IwaniecSarnak} to show the stronger bound $M_\star(g,n;\delta)\ll_{A,C,\varepsilon} n^{1+\varepsilon}\delta$. Let $\det\xi=n$ and write $\xi=\begin{pmatrix}
a & b \\ c & d
\end{pmatrix}$ as usual.  When $\delta>1$, we can replace the right hand side in inequalities \eqref{eq:u-det-sp-1} and \eqref{eq:X^g-sp-1} by $6N\delta$. If either $a=0$ or $d=0$, then the equation $bc=n$ and the divisor bound imply that we have at most $\ll_\varepsilon n^{\varepsilon}$ possibilities for $(b,c)$. Moreover, \eqref{eq:u-det-sp-1} implies that there are at most $\ll (n\delta)^{1/2}$ options for $\Tr \xi$. Hence, the number of possible values of $\xi$ in these cases is $\ll_\varepsilon n^{1/2+\varepsilon}\delta^{1/2}\ll n^{1+\varepsilon}\delta$. Assume next $a\neq 0$ and $d\neq 0$.
Equation \eqref{eq:X^g-sp-1} implies that we have at most $\ll y^{-1}(n\delta)^{1/2}$ options for $c$. Likewise, we have $\ll 1+(n\delta)^{\frac{1}{2}}(x^2+y^2)y^{-1}  \ll_{A,C} y(n \delta)^{\frac{1}{2}}$ choices for $b$. This may be seen by either repeating the computation for \eqref{eq:X^g-sp-1} using the Iwasawa decomposition with respect to the lower triangular unipotents or noting that $\sm & 1 \\ -1 & \esm g.i= (-x+iy)/(x^2+y^2)$ and $\sm & 1 \\ -1 & \esm \xi \sm & 1 \\ -1 & \esm^{-1} = \sm d & -c \\ -b & a \esm$. 
Now that $(b,c)$ is fixed, we use the equality $ad=bc+n$ and the divisor bound to see that there are at most $\ll_\varepsilon n^{\varepsilon}$ possibilities for $(a,d)$. This establishes the inequality for $\delta>1$.

Assume henceforth $\delta<1$. We will be using the simple identity
\begin{equation}\label{eq:det-tr}
\det \xi=\frac{\left(\Tr \xi\right)^2}{4}+\det \xi^0
\end{equation}
and argue in two different ways depending on whether the traces of $\xi_1$, $\xi_2$ are equal or not.

\paragraph*{Case I: $|\Tr\xi_1|\neq |\Tr\xi_2|$}
Lemma \ref{lem:traceless-count-sp} implies that there are at most $\ll_A N^3\delta^2+N$ options for $(\xi_1^0,\xi_2^0)$. After fixing the traceless parts, equation \eqref{eq:det-tr} fixed $(\Tr \xi_1)^2-(\Tr \xi_2)^2$. Because the traces are not equal in absolute value the divisor bound and the trivial bound $|\Tr\xi| \ll N$ imply there are at most $\ll_\varepsilon N^\varepsilon$ choices for the traces. This establishes the claim in this case.

\paragraph*{Case II: $|\Tr\xi_1|=|\Tr\xi_2|$}
In this case, we use the trivial bound $|\Tr\xi_1|=|\Tr\xi_2|\ll N^{1/2}$ from equation \eqref{eq:u-det-sp-1} to fix the traces and Lemma \ref{lem:eq-det-traceless} to fix the traceless part. The final bound is consistent with the claim.
\end{proof}

\subsection{Second Moment Bound for Division Algebras}\label{sec:second-moment-cartan}
In the section, we assume $\mathbf{G}$ is anisotropic, i.e.\ $B$ is a ramified quaternion algebra over $\mathbb{Q}$. Fix and imaginary quadratic field $E/\mathbb{Q}$ of discriminant $D_E$, such that every prime dividing $D_B$ is inert in $E$. By a theorem of Chinburg and Friedman \cite{ChinburgFriedman}, there is an optimal embedding $E\hookrightarrow\mathcal{R}$. We identify henceforth $E$ with its image in $\mathcal{R}$. Denote by $K_E<\mathbf{G}(\mathbb{R})$ the group of norm $1$ elements in $(E\otimes\mathbb{R})^\times$. 
Recall that we have a fixed isomorphism $B_\infty\simeq\operatorname{Mat}_{2\times 2}(\mathbb{R})$ that induces a group isomorphism $\mathbf{G}(\mathbb{R})\simeq \mathbf{SL}_2(\mathbb{R})$, which we use to identify the two groups. Moreover, $K_\infty$ was defined as the image of $\mathbf{SO}_2(\mathbb{R})$ in $\mathbf{G}(\mathbb{R})$, and we define similarly $A$ to be the image of the diagonal subgroup in $\mathbf{G}(\mathbb{R})$.
The group $K_E$ is conjugate to $K_\infty$, and we can write $K_\infty=h K_E h^{-1}$. 

\begin{prop}\label{prop:second-moment-embedded}
Assume $E\hookrightarrow\mathcal{R}$ is an optimal embedding of an imaginary quadratic field $E$ in the maximal order $\mathcal{R}$. Let $h\in\mathbf{G}(\mathbb{R})$ be an element conjugating $K_E$ to $K_{\infty}<\mathbf{SL}_2(\mathbb{R})$.
Then, for any $g\in\mathbf{G}(\mathbb{R})$, $1>\delta>0$,
\begin{equation*}
\sum_{n=1}^N M(g,n;\delta)^2 \ll |D_E|^{2+\varepsilon}N^{\varepsilon}\left[N^3\delta^2+(\lambda+\lambda^{-1})^{2+\varepsilon}(N^{5/2}\delta^{3/2}+N)\right]
\;,
\end{equation*}
where we write $\sqrt{\lambda}\geq(\sqrt{\lambda})^{-1}>0$ for the eigenvalue of the diagonal part in the $K_\infty A K_\infty$ Cartan decomposition of $hg$. Moreover, if $\delta\geq 1$ the bound 
\begin{equation*}
M(g,n;\delta)\ll_{\varepsilon}\left((\lambda+\lambda^{-1})n \delta |D_E|^{1/2}\right)^{1+\varepsilon}
\end{equation*}    
holds for all $g\in\mathbf{G}(\mathbb{R})$ and $n\in\mathbb{N}$.
\end{prop}

We now fix $\mathcal{R},E,h,g$ as in the proposition above and prepare some notation and lemmata that we will use in the course of the proof. The proof is very similar to the split case, except that we track the dependence on $g$ differently, not using its Iwasawa decomposition but rather its Cartan decompsition relative to the stabilizer of $E\hookrightarrow B$.

Because of our choice of $E$ as optimally embedded in $\mathcal{R}$, we can find an isomorphism $B\otimes E\simeq M_2(E)$ where $\mathcal{R}$ is mapped to 
\begin{equation*}
\left\{
\begin{pmatrix}
a & D_B b \\ \tensor[^\sigma]{b}{} & \tensor[^\sigma]{a}{}
\end{pmatrix} \colon a,b\in\widehat{\mathcal{O}_E}\;, a+b\in \mathcal{O}_E
\right\}
\end{equation*}
and after fixing a field embedding $E\hookrightarrow\mathbb{C}$ the algebra $B_\infty$ coincides with $\left\{\sm
a & D_B b \\ \tensor[^\sigma]{b}{} & \tensor[^\sigma]{a}{}
\esm\colon a,b\in\mathbb{C}\right\}$. We denote by $B_\infty^0$ the subspace of traceless elements, equivalently pure quaternions. There is a direct sum decomposition $B_\infty = \mathbb{R}\Id + B_\infty^0$. This decomposition is preserved by the conjugation action. In our new coordinate system, the space $B_\infty^0$ is identified with $i\mathbb{R}\times\mathbb{C}$ and the projection map becomes $(a,b)\mapsto (a^0,b)$ where $a^0=(a-\tensor[^\sigma]{a}{})/2=i \Im a$ is the traceless part of $a\in\mathbb{C}$. The space $B_\infty^0$ is equipped with an inner-product constructed as the direct sum of the standard inner-product on $\mathbb{R}$ and $\mathbb{C}$, i.e. $|(a^0,b)|^2=|a^0|^2+|b|^2$. Let $\mathcal{R}^0$ be the projection of $\mathcal{R}$ to
$B_\infty^0$. Then,  $\mathcal{R}^0<B_\infty^0$ is a lattice of co-volume $\asymp 1$.

In this new coordinate system, we have for a quaternion of positive norm\footnote{Note that the transpose operation with respect to $K_E$ is $\sm \alpha & D_B \beta \\ \tensor[^\sigma]{\beta}{} & \tensor[^\sigma]{\alpha}{}\esm\mapsto
\sm \tensor[^\sigma]{\alpha}{} & D_B \beta \\ \tensor[^\sigma]{\beta}{} & \alpha\esm$.}
\begin{equation*}
u\left(h\begin{pmatrix}
a & D_B b \\ \tensor[^\sigma]{b}{} & \tensor[^\sigma]{a}{}
\end{pmatrix}h^{-1}\right)=\frac{1}{2}\left[\frac{\Nr a+ D_B \Nr b}{\Nr a - D_B\Nr b}-1\right]=\frac{D_B \Nr b}{\Nr a - D_B\Nr b}\;.
\end{equation*}
Hence, if we write in coordinates
\begin{equation*}
(gh)^{-1} \xi (gh)=\begin{pmatrix}
a & D_B b \\ \tensor[^\sigma]{b}{} & \tensor[^\sigma]{a}{}
\end{pmatrix}
\end{equation*}
with $\Nr \xi>0$, then the conditions $u(g^{-1}\xi g)<\delta$, $\Nr \xi<N$ imply
\begin{align}
\label{eq:u-det-1}
\Nr a + D_B \Nr b&<N(2\delta+1)\;,\\
\label{eq:u-det-2}
D_B \Nr b &< \delta N\;.
\end{align}
The traceless part of $(gh)^{-1} \xi (gh)$ is
\begin{equation*}
(gh)^{-1} \xi^0 (gh)=
\begin{pmatrix}
a^0 & D_B b \\ \tensor[^\sigma]{b}{} & \tensor[^\sigma]{a}{^0}
\end{pmatrix}\;,
\end{equation*}
where $a^0=i \Im a$ is the traceless part of $a$.
Equation \eqref{eq:u-det-1} implies that $\Nr a^0 \leq N(2\delta+1)$. Motivated by these inequalities, we denote
\begin{align*}
B_\infty^0\supset X&\coloneq \left\{x=\begin{pmatrix}
a^0 & D_B b \\ \tensor[^\sigma]{b}{} & \tensor[^\sigma]{a}{^0}
\end{pmatrix}\colon D_B \Nr b\leq \delta N\;, \Nr a^0\leq N(2\delta+1)
\right\}\;,\\
B_\infty^0\supset X^g&\coloneq (gh)X(gh)^{-1}\;.
\end{align*}

We decompose $gh$ according to a Cartan decomposition in $K_E A_E K_E$ where $A_E$ is the orthogonal group preserving the quadratic form $(\Im a)^{2}-D_B(\Im b)^2$. Equivalently, the Lie algebra of $A_E$ is
\begin{equation*}
\Lie A_E= \mathbb{R}\cdot \begin{pmatrix}
0 & D_B  \\ 1 & 0
\end{pmatrix}< B_\infty^0\;.
\end{equation*}
Write $gh=k_2 a_E k_1$ with $k_1,k_2\in K_E$, $a_E\in A_E$ and denote by $\sqrt{\lambda}\geq (\sqrt{\lambda})^{-1}>0$ the eigenvalues of $a_E$. Then, $\sqrt{\lambda},(\sqrt{\lambda})^{-1}$ are also the eigenvalue of the diagonal part of the regular $K_\infty AK_\infty$ Cartan decomposition of $hg=h(gh)h^{-1}$, i.e.\ the singular values.

The set $X$ is invariant under conjugation by $K_E$, hence $X^g=(k_2 a_E)X (k_2 a_E)^{-1}$. We can write the equations defining the set $a_E X a_E^{-1}$ explicitly by decomposing the Lie algebra $B_\infty^0$ into the weight spaces of $A_E$. The result of the computation is that every $x=\sm
a^0 & D_B b \\ \tensor[^\sigma]{b}{} & \tensor[^\sigma]{a}{^0}
\esm\in a_E X a_E^{-1}$ satisfies
\begin{align}
\label{eq:X^g-1}
\left(\frac{\lambda+\lambda^{-1}}{2}\Im a^0+\sqrt{D_B}\frac{\lambda-\lambda^{-1}}{2}\Im b\right)^2&\leq N(2\delta+1)\;,
\\
\label{eq:X^g-2}
\left(\frac{\lambda-\lambda^{-1}}{2}\Im a^0+\sqrt{D_B}\frac{\lambda+\lambda^{-1}}{2}\Im b\right)^2 +D_B (\Re b)^2&\leq N\delta\;.
\end{align}

The set $X^g=k_2 (a_E X a_E^{-1}) k_2^{-1}$ is a rotation of $a_E X a_E^{-1}$ around the $\Im a^0$ axis. Hence, the equations defining $X^g$ are derived from \eqref{eq:X^g-1}, \eqref{eq:X^g-2} by a rotation in the $b$-plane.
Notice that equations \eqref{eq:X^g-1} and \eqref{eq:X^g-2} imply that $|\Im a^0| \ll (\lambda+\lambda^{-1})N^{1/2}(1+\delta^{1/2})$. Because the axis $\Im a^0$ is invariant under conjugation by $k_2$ this inequality holds also for $X^g$.

\begin{lem}\label{lem:traceless-count}
Assume $0<\delta\leq 1$. Then,
\begin{equation*}
\#\left(X^{g}\cap \mathcal{R}^0 \right)
\ll |D_E|\left(N^{3/2}\delta+(\lambda+\lambda^{-1})(N\delta^{1/2}+N^{1/2})\right)\;.
\end{equation*}
\end{lem}
\begin{proof}
From $|\Im a^0| \ll (\lambda+\lambda^{-1})N^{1/2}$ we deduce that there are $\ll (\lambda+\lambda^{-1})(N|D_E|)^{1/2}$ possibilities for $a^0=i\Im a^0$.
The second equation \eqref{eq:X^g-2} implies that for any fixed $a^0=i\Im a^0$ the element $b$ belong to an ellipse with radii $\sqrt{N\delta/D_B},\frac{2}{\lambda+\lambda^{-1}}\sqrt{N\delta/D_B}$. Conjugation by $k_2$ amounts to rotating the set around the $\Im a^0$ axis. Hence, this observation remains valid for $X^g$. We deduce that for any fixed $\Im a^0$ we have
\begin{equation*}
\ll \frac{N\delta\sqrt{|D_E|}}{D_B(\lambda+\lambda^{-1})}+\left(\frac{N\delta|D_E|}{D_B}\right)^{1/2}+1\;.
\end{equation*}
possibilities for $b$. The claim follows by multiplying the number of possibilities for $a^0$ by the number of possible $b$'s for each $a^0$.
\end{proof}

\begin{lem}\label{lem:equidisc-count}
Assume $0<\delta\leq 1$. Then, 
\begin{equation*}
\#\left\{ (\xi_1^0,\xi_2^0)\in\left(\mathcal{R}^0 \cap X^g\right)^{2} \colon \Nr \xi_1^0=\Nr \xi_2^0
\right\}\ll_\varepsilon |D_E|^{1+\varepsilon} (\lambda+\lambda^{-1})^{2+\varepsilon}N^\varepsilon
\left(N^2\delta^{3/2}+N^{1/2}\right)
\;.
\end{equation*}
\end{lem}
\begin{proof}
Write
\begin{equation*}
\xi_{1,2}^0=\begin{pmatrix}
a_{1,2}^0 & D_B b_{1,2} \\ \tensor[^\sigma]{b}{_{1,2}} & \tensor[^\sigma]{a}{_{1,2}^0}
\end{pmatrix}
\end{equation*}
and assume $\xi_1^0,\xi_2^0\in X^g$ and $\Nr\xi_1^0=\Nr \xi_2^0$. Our goal is to count the number of possible pairs $(\xi_1^0,\xi_2^0)$.

For every $\xi\in B_\infty^0$ let $\tilde{a}^0$ be the $a^0$ coordinate of $a_E^{-1}\xi a_E$. Then,
\begin{equation*}
\Im \tilde{a}^0=\frac{\lambda+\lambda^{-1}}{2}\Im a^0+\sqrt{D_B}\frac{\lambda-\lambda^{-1}}{2}\Im b\;.
\end{equation*}
Moreover, $\tilde{a}^0$ is also the $a^0$ coordinate of $(gh)^{-1}\xi (gh)$ because conjugation by $K_E$ acts trivially on the $a^0$-axis.
By substitution, we can rewrite equation \eqref{eq:X^g-2} as
\begin{equation}\label{eq:X^g-2-shifted}
\left(\frac{\lambda-\lambda^{-1}}{\lambda+\lambda^{-1}}\Im \tilde{a}^0+\sqrt{D_B}\frac{2}{\lambda+\lambda^{-1}}\Im b\right)^2+D_B(\Re b)^2 \leq N\delta \;.
\end{equation}
Assume $\xi\in X^g$.
Because equation \eqref{eq:u-det-1} implies that $|\Im \tilde{a}^0|\ll N^{1/2}$, we see that \eqref{eq:X^g-2-shifted} restricts $b\in X^g$ to an ellipse with radii $(N\delta/D_B)^{1/2},(\lambda+\lambda^{-1})(N\delta/D_B)^{1/2}$ and center in an interval of length $\ll N^{1/2}$. We deduce that there are at most 
\begin{equation}\label{eq:b-choices}
\ll (\lambda+\lambda^{-1})N\delta^{1/2}\sqrt{\frac{|D_E|}{D_B}}+(\lambda+\lambda^{-1})N^{1/2}\sqrt{ \frac{|D_E|}{D_B}}+1
\end{equation}
choices for $b$  if $\xi\in X^g$. Moreover, we see that necessarily $\Nr b \ll (\lambda+\lambda^{-1})N$

\paragraph*{Case I: $|a_1^0|=|a_2^0|$}
In this case, the condition $\Nr\xi_1=\Nr \xi_2$ implies that $\Nr b_1=\Nr b_2$. Because there are at most $\ll_\varepsilon (n|D_E|)^{\varepsilon}$ elements of norm $n$ in $\widehat{\mathcal{O}}_E$ and $\Nr b_2\ll (\lambda+\lambda^{-1})N$ we see that for any fixed $\xi_1^0$ there are at most $\ll_\varepsilon ((\lambda+\lambda^{-1})N|D_E|)^\varepsilon$ possibilities for $\xi_2^0$. We deduce from Lemma \ref{lem:traceless-count} that the number of possible pairs $(\xi_1,\xi_2)$ with $|a_1^0|=|a_2^0|$ satisfies
\begin{equation*}
\ll_\varepsilon  ((\lambda+\lambda^{-1})N|D_E|)^\varepsilon|D_E| \left(N^{3/2}\delta+(\lambda+\lambda^{-1})(N\delta^{1/2}+N^{1/2})\right)
\end{equation*}
and this bound is compatible with the claim.

\paragraph*{Case II: $|a_1^0|\neq|a_2^0|$}
In this case, we will first count the number of possibilities for $(b_1,b_2)$. We bound the number of choices for $b_1$ using \eqref{eq:b-choices} above.
If $\xi\in X^g$, then equation \eqref{eq:u-det-2} implies
\begin{equation*}
(\Im \tilde{a}^0)^2-\Nr \xi=(\Im \tilde{a}^0)^2-\Nr((gh)^{-1}\xi(gh))\in[0,N\delta]\;.
\end{equation*}
Thus, we deduce for $\xi_{1,2}^0$ that $\left|(\Im \tilde{a}_1^0)^2-(\Im \tilde{a}_2^0)^2
\right|\leq 2N\delta$ and
\begin{equation}\label{eq:Ima_1^0-Ima_2^0}
\left||\Im \tilde{a}_1^0|-|\Im \tilde{a}_2^0|\right|\ll \sqrt{N\delta}\;.
\end{equation}
Once $b_1$ is fixed, inequality \eqref{eq:X^g-2-shifted} restricts $\frac{\lambda-\lambda^{-1}}{\lambda+\lambda^{-1}}\Im \tilde{a}_1^0$ to an interval of length $\ll \sqrt{N\delta}$. Equation \eqref{eq:Ima_1^0-Ima_2^0} then restricts $\frac{\lambda-\lambda^{-1}}{\lambda+\lambda^{-1}}|\Im \tilde{a}_2^0|$ to a interval also of length $\ll \sqrt{N\delta}$.

This constraints the possibilities for the center of the ellipse in Inequality \eqref{eq:X^g-2-shifted} for $b_2$ into two intervals of length $\ll \sqrt{N\delta}$. Hence, given $b_1$, there at most
\begin{equation*}
\ll (\lambda+\lambda^{-1})N\delta\sqrt\frac{|D_E|}{|D_B|}+(\lambda+\lambda^{-1})(N\delta)^{1/2}\sqrt\frac{|D_E|}{|D_B|}+1
\end{equation*}
options for the $b_2$.

After fixing $b_1,b_2$, we use the condition $\Nr \xi_1=\Nr \xi_2$ to fix $(\Im a_1^0)^2-(\Im a_2^0)^2$. The divisor bound and the condition $|\Im a_1^0|\neq|\Im a_2^0|$ now implies there are at most $\ll_\varepsilon ((\lambda+\lambda^{-1})N|D_E|)^\varepsilon$ options for the pair $(a_1^0,a_2^0)$.

The total number of possible pairs $(\xi_1^0,\xi_2^0)$ in this case is thus bounded by
\begin{align*}
\ll_\varepsilon ((\lambda+\lambda^{-1})N|D_E|)^\varepsilon|D_E|&\left(
(\lambda+\lambda^{-1})N\delta^{1/2}+(\lambda+\lambda^{-1})N^{1/2}\right)\\
\cdot&
\left((\lambda+\lambda^{-1})N\delta+(\lambda+\lambda^{-1})(N\delta)^{1/2}+1\right)\;.
\end{align*}
This bound is also compatible with the claim.
\end{proof}

\begin{proof}[Proof of Proposition \ref{prop:second-moment-embedded}]
Assume first $\delta>1$. Then, we follow \cite{IwaniecSarnak} to establish the bound $M(g,n;\delta)\ll_{\varepsilon}  \left((\lambda+\lambda^{-1})n \delta |D_E|^{1/2} \right)^{1+\varepsilon}$. We have the bounds $\Tr a=2\Re a\ll (n\delta)^{1/2}$ and $\Im a=\Im a^0\ll (\lambda+\lambda^{-1})(n\delta)^{1/2}$. After fixing $a$ we can fix $b$ using the equality $n=\det\xi=\Nr a- D_B\Nr b$. The divisor bound and the inequality $\Nr b \ll (\lambda+\lambda^{-1})^2n \delta$ imply we have at most $\ll_\varepsilon (\lambda+\lambda^{-1})^\varepsilon n^\varepsilon \delta^{\varepsilon} |D_E|^{\varepsilon}$ choices for $b$.

Assume next $\delta\leq 1$.
Once again, an important role is reserved for the simple formula
\begin{equation}\label{eq:nr-tr-disc}
\Nr x=\frac{(\Tr x)^2}{4}+\Nr x^0
\end{equation}
that holds for all $x\in B_\infty$ with $x^0\in B_\infty^0$ the traceless part of $x$.
Our goal is to bound the number of pairs $(\xi_1,\xi_2)\in\mathcal{R}$ such that $0\leq \Nr\xi_1=\Nr\xi_2 \leq N$ and $u(g^{-1}\xi_1g)=u(g^{-1}\xi_2 g)<\delta$.

\paragraph*{Case I: $|\Tr \xi_1|\neq |\Tr \xi_2|$}
Lemma \ref{lem:traceless-count} implies that the number of possibilities for the pair $(\xi_1^0,\xi_2^0)$ is bounded by
\begin{equation*}
\ll |D_E|^2 \left(N^3\delta^2+(\lambda+\lambda^{-1})^2(N^{5/2}\delta^{3/2}+N)\right)\;.
\end{equation*}
For any pair $(\xi_1^0,\xi_2^0)\in B_\infty^0\times B_\infty^0$, the lifts to $B_\infty\times B_\infty$ are determined by $(\Tr \xi_1,\Tr \xi_2)$.

From the Formula \eqref{eq:nr-tr-disc}, we derive
$(\Tr \xi_1)^2-(\Tr \xi_2)^2=4\left(\Nr \xi_2^0-\Nr\xi_1^0\right)$. The right hand side is bounded in absolute value $\ll N$.
The divisor bound and the assumption $|\Tr \xi_1|\neq |\Tr \xi_2|$ imply that for every $(\xi_1^0,\xi_2^0)$ the number of possible pairs  $(\Tr \xi_1,\Tr \xi_2)$ is bounded by $\ll_\varepsilon N^\varepsilon$. The cumulative bound is consistent with the claim.

\paragraph*{Case II: $|\Tr \xi_1|= |\Tr \xi_2|$}
In this case, Formula \eqref{eq:nr-tr-disc} implies that $\Nr \xi_1^0=\Nr \xi_2^0$ and we can bound the total number of pairs $(\xi_1^0,\xi_2^0)$ using Lemma \ref{lem:equidisc-count}. The number of pairs $(\Tr \xi_1,\Tr \xi_2)$ is trivially bounded by $\ll \sqrt{N}$ because $|\Tr \xi_1|= |\Tr \xi_2|$. The resulting bound on the pairs $(\xi_2,\xi_2)$ is consistent with the claim.
\end{proof}

\section{Proof of Main Theorem}\label{sec:proof}
This section is dedicated to establishing our main result, Theorem \ref{thm:main}. Recall that $\mathcal{B}_m^{\mathrm{new}}$ is an orthonormal basis of Hecke newforms of weight $m>2$. We can combine Corollary \ref{cor:spectral-lower-bound} and Corollary \ref{cor:geometric-upper-bound} to deduce
\begin{align}
\label{eq:spectral-geometric-inequality}
\sum_{f \in \mathcal{B}_m^\mathrm{new}}  |f^\sharp(g)|^4
&\ll_{\varepsilon} (qD_B)^{3+\varepsilon} \frac{m^2}{m-2}
\Bigg\{\frac{m}{2}\int_0^\infty \left[\sum_{n=1}^{(qD_B)^2m} \frac{1}{n}M(g,n;\delta)^2\right]^{1/2}\\
&+\left[\sum_{n>(qD_B)^2m} \frac{\exp(-n/(qD_B)^2)}{n} M(g,n;\delta)^2 \right]^{1/2}
\frac{\dif\delta}{(1+\delta)^{m/2+1}}\Bigg\}^2\;.
\nonumber
\end{align}
On the right-hand-side, we've denoted by $M(g,n;\delta)$ the counting function associated to a maximal order containing $\mathcal{R}$. 

\subsection{Proof of Main Theorem for the Split Matrix Algebra}
Let $\mathcal{F}$ a fundamental domain for the action $\mathbf{SL}_2(\mathbb{Z})$ on $\mathbb{H}$. Recall that in this case $\Gamma=\Gamma_0(q)<\mathbf{SL}_2(\mathbb{Z})$. 
For $g\in\mathbf{G}(\mathbb{R})=\mathbf{SL}_2(\mathbb{R})$, we denote 
\begin{equation*}
\mathrm{ht}_\Gamma(g)=\min\left\{ y\mid \exists \gamma\in\mathbf{SL}_2(\mathbb{Z})\colon (\gamma g).i=x+iy\in\mathcal{F}\right\}\;.
\end{equation*} 

We first bound the sum $\sum_{n=1}^{(qD_B)^2m}$. Because $M(g,n;\delta)$ is the count associated to the maximal order $\operatorname{Mat}_{2\times 2}(\mathbb{Z})$, the sum is invariant under the operation of replacing $g$ by $\gamma g$ for any $\gamma\in\mathbf{SL}_2(\mathbb{Z})$. In particular, we can arrange $g.i=x+iy$ with $y=\mathrm{ht}_\Gamma(g)$.
We need to convert the logarithmic sum $\sum \frac{1}{n}M(g,n;\delta)^2$ to an unweighted sum. We achieve this using the general, integration-by-parts, identity
\begin{equation*}
\sum_{n=1}^N \frac{1}{n}f(n)=\frac{1}{N}\sum_{n=1}^N f(n)+\int_1^N\frac{1}{t}\sum_{n=1}^t f(n) \frac{\dif t}{t}\;,
\end{equation*}
which holds for any $f\colon\mathbb{N}\to\mathbb{C}$. This identity and Proposition \ref{prop:second-moment-sp} imply
\begin{equation}
\label{eq:second-moment-applied}
\sum_{n=1}^{N} \frac{1}{n}M(g,n;\delta)^2\ll_{\varepsilon} N^{2+\varepsilon}\delta^2+N^\varepsilon+\mathrm{ht}_\Gamma(g)^2\begin{cases}
N^{1+\varepsilon}\delta^{3/2}+N^{1/2+\varepsilon}\delta & \delta<1\\
N^{1+\varepsilon}\delta & \delta\geq 1
\end{cases}\;,
\end{equation}
where $N=(qD_B)^2m$. We next need to compute the integral $\int_0^\infty \sqrt{\cdots}\frac{\dif\delta}{(1+\delta)^{m/2+1}}$. We use the $l^2$--$l^1$ inequality to separate the terms in \eqref{eq:second-moment-applied} under the square root.
We first compute the contribution of the first two terms in \eqref{eq:second-moment-applied}
\begin{equation*}
N^{1+\varepsilon}\int_0^\infty \frac{\delta}{(1+\delta)^{m/2+1}}\dif \delta +N^\varepsilon\int_0^\infty \frac{1}{(1+\delta)^{m/2+1}}\dif\delta\ll \frac{N^{1+\varepsilon}}{m(m-2)}+\frac{N^{\varepsilon}}{m}\ll
\frac{(qD_B)^{2+\varepsilon}m^{\varepsilon}}{m-2}\;.
\end{equation*}

We will need the following estimate for $0<\kappa<1$, $\varepsilon>0$
\begin{align*}
\int_0^1\frac{\delta^\kappa}{(1+\delta)^{m/2+1}}\dif\delta&\leq m^{-\kappa(1-\varepsilon)}\int_0^{m^{-1+\varepsilon}}\frac{\dif\delta}{(1+\delta)^{m/2+1}}
+\int_{m^{-1+\varepsilon}}^1\frac{\dif\delta}{(1+\delta)^{m/2+1}}\\
&\ll\frac{m^{-\kappa(1-\varepsilon)}}{m^{1-\varepsilon}}+(1+m^{-1+\varepsilon})^{-m/2}
\ll_{\epsilon} \frac{m^{-\kappa(1-\varepsilon)}}{m^{1-\varepsilon}}+e^{-m^{\varepsilon}/4}\ll_{\varepsilon,\kappa} m^{-1-\kappa+\varepsilon}\;.
\end{align*}
To compute the contribution of the term proportional to $\mathrm{ht}_\Gamma(g)$ in \eqref{eq:second-moment-applied}, we split the integral over $\delta$ into $\int_0^1+\int_1^\infty$ and use the integral estimate above to arrive at
\begin{align*}
\int_0^\infty\begin{cases}
N^{1/2+\varepsilon}\delta^{3/4}+N^{1/4+\varepsilon}\delta^{1/2} & \delta<1\\
N^{1/2+\varepsilon}\delta^{1/2} & \delta\geq 1
\end{cases}\cdot\frac{\dif\delta}{(1+\delta)^{m/2+1}}
&\ll (Nm)^\varepsilon\left(\frac{N^{1/2}}{m^{1+3/4}}+\frac{N^{1/4}}{m^{1+1/2}}+2^{-m/2}\frac{N^{1/2}}{(m-2)}\right)\\
&\ll (qD_B)^{1+\varepsilon}\frac{m^{-1/4+\varepsilon}}{m-2}\;.
\end{align*}
In conclusion
\begin{equation*}
\int_0^\infty \left[
\sum_{n=1}^{(qD_B)^2m} \frac{1}{n}M(g,n;\delta)^2\right]^{1/2} 
\frac{\dif \delta}{(1+\delta)^{m/2+1}}\ll_\varepsilon (qD_B)^{2+\varepsilon}m^{\varepsilon}\frac{1+\mathrm{ht}_\Gamma(g)m^{-1/4}}{m-2}\;.
\end{equation*}

The computation of the bound for the integral $\int_0^{\infty} \left[\sum_{n>(qD_B)^2m} \frac{\exp(-n/(qD_B)^2)}{n} M(g,n;\delta)^2 \right]^{1/2}
\frac{\dif\delta}{(1+\delta)^{m/2+1}}$ uses a very similar argument, except that we need to apply the integration-by-parts identity
\begin{equation*}
\sum_{n>A m} \frac{\exp(-n/A)}{n}f(n)=-\frac{\exp(-m)}{Am}\sum_{n=1}^{Am} f(n)
+\frac{1}{A}\int_m^\infty \exp(-t) \sum_{n=1}^{At}f(n)\left(1+\frac{1}{t}\right)\frac{\dif t}{t}\;,
\end{equation*}
that holds for any function $f\colon \mathbb{N}\to\mathbb{C}$ satisfying $\log f(n)=o(n)$ and $A,m\geq 1$. The contributions of these terms is then easily seen to be negligible.

Combining these inequalities with \eqref{eq:spectral-geometric-inequality}, we arrive at
\begin{equation*}
\sum_{f \in \mathcal{B}_m^\mathrm{new}}  |f^\sharp(g)|^4
\ll_{\varepsilon} (qD_B)^{7+\varepsilon}\frac{m^{4+\varepsilon}}{(m-2)^3}\left(1+\mathrm{ht}_\Gamma(g)^2m^{-1/2}\right)\;.
\end{equation*} 
This is consistent with the first claim in Theorem \ref{thm:main} for $m>2$. As mentioned in the introduction, the second claim requires the additional input of \cite[Theorem 1.8]{BKY}, which says that most of the $L^4$-mass is concentrated on $\mathrm{ht}_{\Gamma}(g) \ll m^{\frac{1}{4}}$. Since the extension of said Theorem to include a polynomial level dependence follows their proof almost verbatim, we leave it to the reader.
\qed
\subsection{Proof of Main Theorem for Division Algebras}
In this section, we use the notations of \S \ref{sec:second-moment-cartan}.
We follow the same arguments as for the split algebra replacing Proposition \ref{prop:second-moment-sp} by Proposition \ref{prop:second-moment-embedded} to arrive at
\begin{equation*}
\sum_{f \in \mathcal{B}_m^\mathrm{new}}  |f^\sharp(g)|^4
\ll_{\varepsilon} |D_E|^{2+\varepsilon}(\lambda+\lambda^{-1})^{2+\varepsilon}(qD_B)^{7+\varepsilon}\frac{m^{4+\varepsilon}}{(m-2)^3}\;.
\end{equation*}
Recall that $E\hookrightarrow\mathcal{R}$ is any optimal embedding of an imaginary quadratic field into the fixed maximal order. By \cite{ChinburgFriedman}, this is always possible if any prime dividing $D_B$ is inert in $E$. Using the chinese remainder theorem  one deduces that such a discriminant $D_E$ exists satisfying $|D_E|\ll D_B$.

Lastly, we can replace $g$ by any $\gamma g$ for any $\gamma\in \Gamma$, hence $(\lambda+\lambda^{-1})$ is polynomially bounded by the volume of $\lfaktor{\Gamma}{\mathbf{G}(\mathbb{R})}$, as follows from \cite{ChuLi}. The latter has been recently improved by the second named author in
 \cite{SteinerSmallDiameter}.

\bibliographystyle{myalpha}
\bibliography{theta_kernel_bib}
\end{document}

%% file: Intro_arxiv_v3.tex
\section{Introduction}

The study of distributional aspects of automorphic forms has enjoyed ample consideration in the past couple of decades, in particular questions related to the Quantum Unique Ergodicity Conjecture, various bounds for $L^p$-norms, and restriction problems. In this paper, we are mainly concerned with the $L^{\infty}$-norm of holomorphic Hecke eigenforms on arithmetic hyperbolic surfaces in the large weight limit, though our method also gives essentially sharp results for moments of $L^4$-norms.

The sup-norm problem asks to provide the best possible bound on the sup-norm of a Hecke eigenform in terms of the analytic conductor. Specifically, one often seeks a non-trivial bound on the sup-norm separately with respect to the weight, Laplace eigenvalue or level aspect. It is analogous and closely related to the Lindel\"of Hypothesis for automorphic $L$-functions.
The go-to method for the majority of previous work on this problem is amplification. It was first used in this context by Iwaniec--Sarnak in the pioneering paper \cite{IwaniecSarnak}, though the idea of an amplifier goes back to Selberg \cite{selberg1942zeros}. Iwaniec and Sarnak showed the bound 
\begin{equation}
\|\varphi\|_{\infty} \ll_{\Gamma, \varepsilon} (1+|\lambda_{\varphi}|)^{\frac{5}{24}+\varepsilon} \|\varphi\|_2
\label{eq:IS95}
\end{equation}
for a Hecke--Maass form $\varphi\colon\lfaktor{\Gamma}{\mathbb{H}}\to\mathbb{C}$, where the lattice $\Gamma<\mathbf{SL}_2(\mathbb{R})$ is the unit norm elements of an Eichler order in a quadratic division algebra. Here and henceforth, we've adopted Vinogradov's notation. Their result marked the first time a power of $1+|\lambda_{\varphi}|$ was saved over what holds for a general Riemannian surface. Indeed, \eqref{eq:IS95} has been known to hold with exponent $\frac{1}{4}$ for a general compact Riemannian surface, without any further assumptions of arithmetic nature (cf. \cite{Sogge}). 
The amplifying technique has been used heavily due to its versatility. In the context of autmorphic forms on arithmetic hyperbolic surfaces, Blomer--Holowinsky \cite{BH10}, Templier \cite{HT1, Thybrid}, Harcos--Templier \cite{HT2,HT3}, Saha \cite{saha2017hybrid, Saha14, saha2019sup}, Hu--Saha \cite{HuSaha} and K{\i}ral \cite{halflevel} have used it to show subconvex bounds in various level aspects; Das--Sengupta \cite{DasSengupta}, Steiner \cite{halfsup} have used it to show subconvex bounds in the weight aspect. Blomer--Harcos--Mili\'cevi\'c \cite{blomer2016bounds}, Blomer--Harcos--Maga--Mili\'{c}evi\'{c} \cite{BHMM} applied it to a more general setting over number fields, which corresponds to products of hyperbolic 2- and 3-spaces. The most general $\mathbf{PGL}_2$ result is due to Assing \cite{AssingSup}. Moreover, the technique has also been adopted to arithmetic 2-spheres by Vanderkam \cite{Vanderkam} and products of 2- and 3-spheres by Blomer--Michel \cite{BM2011,BM2013}, and generalized to higher rank, e.g. Blomer--Pohl \cite{BlomerPohl} for $\mathbf{Sp}_4$, Blomer--Maga \cite{BlomerMagaPGL4,BlomerMagaPGLn} for $\mathbf{PGL}_n$ ($n\ge 4$), and Marshall \cite{Marshallsup} for semisimple split Lie groups over totally real fields and their totally imaginary quadratic extensions, to name a few.

In this paper, we employ a different tool, namely the theta correspondence. The theta correspondence was first used by the second named author \cite{S3sup} to tackle sup-norm problems. It has been previously used by Nelson to answer questions regarding quantum unique ergodicity and quantum variance \cite{NelsonQV1,NelsonQV2,NelsonQV3,NelsonSymmL}, and give Fourier-like expansions for forms living on compact spaces \cite{NelsonComp}. The main advantage of this approach is that instead of looking at an amplified second moment, we are able to bound a fourth moment sharply. Another advantage is that it works for co-compact lattices equally well as it does for non-co-compact ones. Our main theorem and its corollary read as follows.

\begin{thm}
\label{thm:main}
Let the arithmetic lattice $\Gamma < \mathbf{SL}_{2}(\mathbb{R}) $ be the unit norm elements of an Eichler order in an indefinite quaternion algebra over $\mathbb{Q}$ and $\{f_j\}_j \subset S_m^{\mathrm{new}}(\Gamma) $ be an orthonormal\footnote{with respect to the probability measure} basis of Hecke newforms of weight $m >2$. Then, there is a constant $A \ge 1$, such that for any $\varepsilon>0$, there is a constant $C_{\varepsilon}$ for which we have
\begin{equation}
\sum_j y^{2m}|f_j(z)|^4 \le C_{\varepsilon} \covol(\Gamma)^A  m^{1+\varepsilon} \left(1+m^{-\frac{1}{2}} \mathrm{ht}_{\Gamma}(z)^2 \right)\;,
\label{eq:thmmain1}
\end{equation}
where $\mathrm{ht}_{\Gamma}(z)=1$ if $\Gamma$ is co-compact and $$\mathrm{ht}_{\Gamma}(z)=\max_{\gamma \in \mathbf{SL}_2(\mathbb{Z})}  \ \Im(\gamma z)$$
if $\Gamma < \mathbf{SL}_2(\mathbb{Z})$. Furthermore, we have
\begin{equation}
\sum_j \|f_j\|_4^4 \le C_{\epsilon} \covol(\Gamma)^A  m^{1+\epsilon}\;.
\label{eq:thmmain2}
\end{equation}
\end{thm}

\begin{cor} Let $\Gamma < \mathbf{SL}_{2}(\mathbb{R}) $ be as above with the additional assumption of being co-compact ($\Leftrightarrow B(\mathbb{Q})$ is non-split) and $f \in S_m^{\mathrm{new}}(\Gamma)$ a Hecke newform of weight $m$. Then, there is a constant $A \ge 1$, such that for any $\varepsilon>0$, there is a constant $C_{\varepsilon}$ for which we have
\begin{equation}
\sup_{z \in \mathbb{H}} y^{\frac{m}{2}}|f(z)| \le C_{\varepsilon} \covol(\Gamma)^{A}  m^{\frac{1}{4}+\varepsilon} \|f\|_2\;.
\label{eq:cormain}
\end{equation}
\label{cor:main}
\end{cor}

The first half of Theorem \ref{thm:main} marks a significant improvement over what has been known previously. It shows that the $L^{\infty}$-norm of the fourth moment of holomorphic newforms of weight $m$ is, essentially, as small as it can be, meaning that they enjoy a stronger `orthogonality' relation than what was previously known. Remarkably, our proof does not rely on any deep results from arithmetic geometry such as Deligne's bound for the Hecke eigenvalues, but rather a sharp bound for a second moment matrix count as we shall explain in further detail in Section \ref{sec:method}. The second half of Theorem \ref{thm:main} is a simple consequence of the first half if $\Gamma$ is co-compact and otherwise it follows in conjunction with \cite[Theorem 1.8]{BKY}, which says that the mass of the fourth norm is concentrated in the domain $\{z \in \lfaktor{\Gamma}{\mathbb{H}} : \mathrm{ht}_{\Gamma}(z) \le m^{\frac{1}{4}} \}$. Following Sarnak and Watson \cite{SarnakWatsonL4}, Inequality \eqref{eq:thmmain2}, through the use of Watson's formula \cite[Thm. 3]{Watsontriple} or more generally Ichino's formula \cite{Ichino} (cf. \cite[Section 4]{Nelsonlevelequidistribution}), may be reformulated as a Lindel\"of on average statement about degree eight $L$-functions. In particular, assuming that the product of the reduced discriminant $D_B$ of $B$ and the level $q$ of $\Gamma$ is square-free, one deduces
\begin{equation}
\frac{1}{m}\sum_f  \frac{1}{2m}\sum_g L(f \times f \times g, \tfrac{1}{2}) \le C_{\varepsilon} \covol(\Gamma)^A m^{\varepsilon}\;,
\label{eq:lindlavg}
\end{equation}
where $f \in S_m^{\mathrm{new}}(\Gamma)$ runs through a basis of newforms of weight $m$ for $\Gamma$ and $g \in S_{2m}(\Gamma)$ runs through an orthonormal set of newforms of weight $2m$ for $\Gamma$ with Hecke eigenvalues equal to $1$ for all primes $p \mid D_B$ and Atkin--Lehner eigenvalues equal to $-1$ for all primes $p \mid q$. This should be compared to the result of Sun--Ye \cite{sun2019double} who considered the double average of the degree six $L$-function $L(\Sym^2 f \times g, \frac{1}{2})$, where $f,g$ are Hecke eigenforms of weight $m$, respectively $2m$, for $\mathbf{SL}_2(\mathbb{Z})$. Note that $L(f \times f \times g, \tfrac{1}{2})=L(\Sym^2 f \times g, \tfrac{1}{2}) L(g,\frac{1}{2})$. One should also mention a result of Khan, who managed to show an asymptotic formula for the left-hand-side of \eqref{eq:thmmain2} for $\Gamma=\mathbf{SL}_2(\mathbb{Z})$ with an extra (smooth) average over the weight $m$. Khan's result matches up with conjectures concerning the asymptotics of the $L^4$-norm in the large weight aspect. We refer to \cite{BKY} for details regarding these conjectures. In the future, we plan to address the question whether one can upgrade the second half of Theorem \ref{thm:main} to an asymptotic without any extra average over the weight. We shall also mention the strongest individual bound for the $L^4$-norm of a Hecke eigenform $f$ of weight $m$ on $\mathbf{SL}_2(\mathbb{Z})$ which is due to Blomer--Khan--Young \cite{BKY}. They managed to show $\|f\|_4 \ll_{\varepsilon} m^{\frac{1}{12}+\varepsilon} \|f\|_2$.

The convex or trivial bound in the context of Corollary \ref{cor:main} is $ \ll \covol(\Gamma)^{\frac{1}{2}}m^{\frac{1}{2}}$ and the first non-trivial bound in the weight aspect $\ll_\varepsilon m^{\frac{1}{2}-\delta+\varepsilon}$ for a small $\delta>0$ was achieved by Das--Sengupta\footnote{$\delta=1/32$ appears in the published version, though this has been corrected to $\delta=1/64$ in a recent revision on the arXiv.} \cite{DasSengupta} through the use of an amplifier. The previous best  bound in the weight aspect is due to Ramacher--Wakatsuki \cite{RamacherWakatsuki} who establish a subconvex bound for the sup-norm in significant generality.

The analogue of Corollary \ref{cor:main} for non-uniform lattices is much easier to establish, because one can use the Fourier expansion at a cusp and then apply Deligne's bound for the Fourier coefficients. This was observed by Xia \cite{Supnormintweight}, who worked out the case $\Gamma=\mathbf{SL}_2(\mathbb{Z})$. In the same fashion, a sharp hybrid bound for holomorphic forms of minimal type was derived by Hu--Nelson--Saha \cite{hu2019some}. We would also like to thank Paul Nelson for pointing out to us the relation between our technique and \cite[Theorem 3.1.]{NelsonComp}. Nelson uses  an explicit (non-holomorphic) version of Shimizu's theta kernel \cite{Shimizu} to construct an expansion of $y^m|f(z)|^2$, where $f$ is an arithmetically normalized newform on a compact arithmetic surface, that resembles a Fourier expansion.

Finally, we shall mention that we did not attempt to optimize the dependence on the co-volume or level in Theorem \ref{thm:main} in this first paper. Due to our method requiring sharp bounds for a second moment matrix count of length comparable to the conductor, any such undertaking must necessarily address the inability of pre-existing matrix counting techniques in the non-split case, such as \cite{HT1}, to deal with large determinants. Furthermore, a strategy needs to be devised to incorporate the dependence on the reduced discriminant of the indefinite quaternion algebra. All of this shall be addressed in a sequel joint with P. Nelson \cite{Masstheta4moment}.

As far as the structure of this paper goes, in the subsequent section, we shall briefly explain the main concept of the proof as well as mentioning an alternative approach using $L$-functions instead of a theta kernel. Sections \ref{sec:weilreptheta} and \ref{sec:actionEichler} deal with local and global properties of the Weil representation and their consequences to the associated theta series. The action of the Hecke algebra on the theta kernel is computed in Section \ref{sec:lift}. In Sections \ref{sec:Bergman} and \ref{sec:specexp}, we show that the Bergman kernel satisfies the required assumptions in the construction of the theta kernel and compute its spectral expansion. In Section \ref{sec:L2reduction}, we reduce a bound on the $L^2$-norm of the theta kernel to matrix counts. In Section \ref{sec:matrix}, we prove the essentially sharp second moment matrix count. The main theorem is then established in Section \ref{sec:proof}.

\begin{acknowledgements} 
It is a pleasure to thank Paul Nelson for many enlightening discussions on this project and very useful comments on an earlier draft of this manuscript. We would like to deeply thank Peter Sarnak for his continuous encouragement and numerous fruitful discussions on the topic. We thank Jared Wunsch for his help with the analytic aspects of this work. We are also grateful to Valentin Blomer, Simon Marshall, and the referees for their comments on a previous version of the manuscript. 

The majority of this work was conducted during a stay of the second named author at Northwestern University, and subsequently completed at the respective home institutions: Northwestern University and Institute for Advanced Study / ETH Z\"urich.
The first named author has been supported by an AMS Centennial Fellowship and he would like to thank the AMS  for its generosity. The second named author would like to thank Northwestern University for their hospitality, the Institute for Advanced Study, where he was supported by the National Science Foundation Grant No. DMS -- 1638352 and the Giorgio and Elena Petronio Fellowship Fund II, and the Institute for Mathematical Research (FIM) at ETH Z\"urich.
\end{acknowledgements}

\section{General Method}
\label{sec:method}

In this section, we shall briefly explain two essentially equivalent strategies that lead to Theorem \ref{thm:main}. We shall first lay out the approach which is conceptually closer to that of an amplifier. For simplicity, we shall assume everything is unramified, i.e. $\Gamma= \mathbf{SL}_2(\mathbb{Z})$, which is the set of determinant one elements of the maximal order $\mathcal{R}=\operatorname{Mat}_{2x2}(\mathbb{Z})$ inside the quaternion algebra $\operatorname{Mat}_{2x2}(\mathbb{Q})$. Let $\mathcal{R}_n$ denote the elements of $\mathcal{R}$ of norm $n$, such that $\Gamma=\mathcal{R}_1$. We begin with a Bergman kernel (also known as a reproducing kernel) on $S_m(\Gamma)$, the space of weight $m$ holomorphic cusp forms on $\Gamma$,
\begin{equation}
B(z,w) = \sum_j \Im(z)^{\frac{m}{2}} f_j(z) \Im(w)^{\frac{m}{2}} \overline{f_j(w)}\;,
\label{eq:methodBerg}
\end{equation}
where $\{f_j\}_j$ is an orthonormal basis of Hecke eigenforms of the space $S_m(\Gamma)$. The amplified counterparts to the Bergman kernel are
\begin{equation}
B_n(z,w) = \sum_j \lambda_j(n) \Im(z)^{\frac{m}{2}} f_j(z) \Im(w)^{\frac{m}{2}} \overline{f_j(w)}\;,
\label{eq:methodBergamp}
\end{equation}
where $\lambda_j(n)$ is the $n$-th Hecke eigenvalue of the newform $f_j$. We normalize the Hecke operators so that Deligne's bound reads $|\lambda_j(n)| \le d(n)$, $d(n)$ is the divisor function. The kernels $B_n$ are roughly of the shape
\begin{equation}
B_n(z,w) \approx \frac{m}{\sqrt{n}} \sum_{\substack{\alpha \in \mathcal{R}_n \\ u(\alpha z, w) \le \frac{1}{m}}} 1\;,
\label{eq:methodBergaprox}
\end{equation}
where $u(z,w)=|z-w|^2/(4 \Im(z) \Im(w))$. Instead of taking a suitable linear combination of \eqref{eq:methodBergamp} as one would do for an amplifier, we consider
\begin{equation}
\int_0^1  \left| \sum_{n \le m} B_n(z,z) e(nt) \right|^2 dt = \sum_{i,j} \Im(z)^{2m} |f_j(z)|^2|f_i(z)|^2 \sum_{n \le m} \lambda_j(n) \overline{\lambda_i(n)}\;.
\label{eq:methodlsym}
\end{equation}
To the latter, or more precisely a smooth version thereof, one may apply Vorono\"i summation. If we set aside any intricacies stemming from Riemann zeta factors and smoothing, we pick up main terms for $i=j$ corresponding to the poles of $L(f_i \times f_j,s)$ at $s=1$ for $i=j$ and a dual sum of length $m^2/m$. Thus, we find that \eqref{eq:methodlsym} is approximately
\begin{equation}
 m \sum_{j} \Im(z)^{2m} |f_j(z)|^4  + \sum_{i,j} \Im(z)^{2m} |f_j(z)|^2|f_i(z)|^2 \sum_{n \leq  m^2/m} \lambda_j(n) \overline{\lambda_i(n)}\;.
\label{eq:methodlsym2}
\end{equation}
We see that the \emph{new} dual sum is once again of the shape \eqref{eq:methodlsym} and we may replace it with its geometric counterpart. Through rearranging and the use of the approximation of the Bergman kernel \eqref{eq:methodBergaprox}, one arrives at
\begin{equation}
\sum_{j} \Im(z)^{2m} |f_j(z)|^4 \ll  m \sum_{n \le m} \frac{1}{n} \sum_{\substack{\alpha_1,\alpha_2 \in \mathcal{R}_n \\ u(\alpha_i z,z) \le \frac{1}{m}, i=1,2}} 1\;.
\label{eq:methodmainineq}
\end{equation}
We see that we end up with a second moment matrix count. Before we discuss the latter further, we shall describe how to arrive at the same inequality in an alternate fashion by using a theta kernel.

At its core, one wishes to find a kernel\footnote{Here, $\Gamma<\mathbf{G}(\mathbb{Q})$ is a congruence lattice in an indefinite inner-form $\mathbf{G}$ of $\mathbf{SL}_2$ and $\Lambda<\mathbf{SL}_2(\mathbb{Q})$ is a lattice in the split form $\mathbf{SL}_2$ that arises in Shimizu's explicit Jacquet-Langlands transfer of modular forms on $\lfaktor{\Gamma}{\mathbf{G}(\mathbb{R})}$.} $\vartheta: \lfaktor{\Gamma}{\mathbf{SL}_2(\mathbb{R})} \times \lfaktor{\Gamma}{\mathbf{SL}_2(\mathbb{R})} \times \lfaktor{\Lambda}{\mathbf{SL}_2(\mathbb{R})}\to\mathbb{C}$, such that
\begin{equation}
\left \langle \vartheta(z,w;\bullet), (\Im\bullet)^{m/2}\tilde{f} \right\rangle = \Im(z)^{\frac{m}{2}} f(z) \Im(w)^{\frac{m}{2}} \overline{f(w)} \cdot \|\tilde{f}\|_2^2\;,
\label{eq:introthetaproperty}
\end{equation}
for an $L^2$-normalized newform $f$ and $\tilde{f}$ an arithmetically normalized newform in the Jaquet--Langlands transfer to $\mathbf{GL}_2$ of the automorphic representation generated by $f$. It immediately follows that
\begin{equation}
\sum_{f} \|\tilde{f} \|_2^2 \cdot \Im(z)^{2m} |f(z)|^4 \le \|\vartheta(z,z;\bullet)\|_2^2
\label{eq:methodbessel}
\end{equation}
by Bessel's inequality. For $\Gamma = \mathbf{SL}_2(\mathbb{Z})$, such a kernel may be given by
\begin{equation}
\vartheta(z,w;\zeta) = \sum_{n=1}^{\infty} B_n(z,w) n^{\frac{m-1}{2}} e(n\zeta)\;.
\label{eq:simplethetasl2Z}
\end{equation} 
This may be used to recover \eqref{eq:methodmainineq} upon using the Hoffstein--Lockhart bound for $\|\tilde{f}\|_2$ \cite{HoffsteinLockhart} and standard bounds for the incomplete Gamma function. 

We prefer to employ the latter approach as it avoids translating spectral data back into geometric terms. Specifically, in equation \eqref{eq:methodlsym2} we have been able to replace the dual sum by the integral of the same amplified Bergman kernel on the left hand side of \eqref{eq:methodlsym}. This step cannot be reproduced verbatim in the ramified cases. Instead, one would need to express the dual sum in terms of Fourier expansions of amplified Bergman kernels associated to various levels and different cusps. The approach using the theta correspondence avoids these issues altogether.

Whilst the constructions of theta kernels in great generality have been known for a while, see \cite{Shimizu} or \cite[Section 5 \& Appendix B]{NelsonComp} for an explicit example, they are unfortunately generally not in $L^2$. An attempt to rectify this, would be to project such a theta kernel to $S_m(\Gamma)$. Formul{\ae} for such projections are given in Gross--Zagier \cite[Section IV.5]{grosszagier}. However, we follow a different path. Motivated by the simplicity of the kernel $\vartheta$ in the case $\Gamma=\mathbf{SL}_2(\mathbb{Z})$ \eqref{eq:simplethetasl2Z}, we modify the general construction of a theta kernel to mirror a classical Bergman kernel of weight $m$. In order to show that the novel theta kernel behaves in the prescribed fashion, we use a method of Vign\'eras \cite{Vigneras} at the infinite place and compute the Fourier--Whittaker expansion in the $\zeta$-variable. We compare the latter with Shimizu's explicit form of the Jacquet-Langlands correspondence \cite{Shimizu}. As a corollary, we derive a new elementary theta series for indefinite quadratic forms of signature $(2,2)$.

\begin{thm} Let $\mathcal{R}$ be an Eichler order of level $q$ in an indefinite division quaternion algebra over $\mathbb{Q}$ of reduced discriminant $D_B$. Denote by $\mathcal{R}^{+}$ the subset of elements of positive norm and by $\Gamma$ the subset of elements of norm equal to one. Furthermore, let $f \in S_m(\Gamma)$ be a cusp form of weight $m>2$. Then, for each $z \in \mathbb{H}$, the function $\mathcal{F}_{f}(z;\bullet)$, given by
\begin{equation}
\mathcal{F}_{f}(z;\zeta) = \sum_{\alpha \in \lfaktor{\Gamma}{\mathcal{R}^+}} \Nr(\alpha)^{\frac{m}{2}-1} (f|_m\alpha)(z) e(\Nr(\alpha)\zeta)\;,
\label{eq:individualtheta}
\end{equation}
is a cusp form of weight $m$ for $\Gamma_0(qD_B)$. Moreover, we have $\mathcal{F}_{T_n f}(z;\bullet) = (T_n\mathcal{F}_{f})(z;\bullet)$ for $(n,qD_B)=1$.
\label{thm:individualtheta}
\end{thm}

Returning to the second moment matrix count, we see that upon using partial summation we need to bound the number of solutions to
\begin{equation}
\alpha_1,\alpha_2 \in \mathcal{R}: 1 \le \Nr(\alpha_1)=\Nr(\alpha_2) \le N, \quad u(\alpha_i z,z) \le \tfrac{1}{m}, \ i=1,2\;.
\label{eq:methodmatrix}
\end{equation}
Consider $z$ fixed for the moment. Then, we are given a quadratic equation in eight variables all of size $N^{\frac{1}{2}}$ with four additional linear inequalities of density $m^{-\frac{1}{2}}$. Heuristics suggest that we should have on the order of $(N^{\frac{1}{2}})^8\cdot N^{-1} \cdot (m^{-\frac{1}{2}})^4=N^3m^{-2}$ solutions for $N$ large. We see that for $N \le m$, $N^3m^{-2} \le N$, which is the bound we are aiming for. Moreover, by considering the order $\operatorname{Mat}_{2 \times 2}(\mathbb{Z})$ and the special point $z=i$, we see that the matrices of the shape $\sm a & -b \\ b & a  \esm$ with $1 \le a^2+b^2 \le N$ satisfy the conditions in \eqref{eq:methodmatrix} and give rise to a lower bound of size $N$. Likewise, we should expect that such subvarieties with exceptionally many solutions exist also for other special points and orders under consideration. Hence, the general estimate we seek is at the cusp of what is achievable. This is in stark contrast to the classical approach of an amplifier, where one may consider matrices of reduced norm up to only a small power of $m$ in order to get a non-trivial result. However, the difficulty of the task at hand is rewarded with a sharp fourth moment estimate. In order to achieve the required bound, we rely on geometry of numbers arguments, which have been successful in the past for first moments (cf. \cite{HT3}), in particular with regards to uniformity in the varying point $z$. To account for the additional quadratic equation, we decompose each matrix $\alpha_i$ into two parts: a multiple of the identity and a traceless part $\alpha_i^0$. To the traceless parts $\alpha_i^0$ we apply the geometry of numbers arguments. The quadratic equation now reads
$$
\Nr \alpha_1 = \frac{(\Tr \alpha_1)^2}{4}+\Nr \alpha_1^0 = \Nr \alpha_2 = \frac{(\Tr \alpha_2)^2}{4} + \Nr \alpha_2^0
$$
and we may use the divisor bound to bound the number of possibilities for the traces. This gives the required bound if and only if the traces are not equal in absolute value. The latter case needs to be dealt with separately. We do so by showing that there are essentially only a constant number of matrices $\alpha \in \mathcal{R}$ satisfying $u(\alpha z,z) \le \frac{1}{m}$ of a given trace and reduced norm $\le m$.

As a final remark, we address the natural question, whether the method lends itself to further amplification. Albeit it being straightforward to produce amplified versions of \eqref{eq:methodmainineq}, the problem lies within the matrix count, where there is no further space for savings as all of the savings stemming from $u(\alpha z,z)\le \frac{1}{m}$ are used up by the fact that we already have to consider matrices of determinant $\le m$. Any additional increase in the size of the determinant will thus automatically increase the bound on the matrix count and subsequently the geometric side of \eqref{eq:methodmainineq} by a considerable amount.

%% file: Fourth_moments_I_arxiv_v3.bbl
\begin{thebibliography}{BHMM20}

\bibitem[AL70]{AtkinLehner}
A.~O.~L. Atkin and J.~Lehner.
\newblock Hecke operators on {$\Gamma \sb{0}(m)$}.
\newblock {\em Math. Ann.}, 185:134--160, 1970.

\bibitem[Ass17]{AssingSup}
E.~Assing.
\newblock {O}n sup-norm bounds part {I}: ramified {M}aa{\ss} newforms over
  number fields.
\newblock {\em {P}reprint}, 2017.
\newblock {\tt arXiv:1710.00362}.

\bibitem[BH10]{BH10}
V.~Blomer and R.~Holowinsky.
\newblock Bounding sup-norms of cusp forms of large level.
\newblock {\em Invent. Math.}, 179(3):645--681, 2010.

\bibitem[BHM16]{blomer2016bounds}
V.~Blomer, G.~Harcos, and D.~Mili\'{c}evi\'{c}.
\newblock Bounds for eigenforms on arithmetic hyperbolic 3-manifolds.
\newblock {\em Duke Math. J.}, 165(4):625--659, 2016.

\bibitem[BHMM20]{BHMM}
V.~Blomer, G.~Harcos, P.~Maga, and D.~Mili\'{c}evi\'{c}.
\newblock The sup-norm problem for {GL}(2) over number fields.
\newblock {\em J. Eur. Math. Soc. (JEMS)}, 22(1):1--53, 2020.

\bibitem[BKY13]{BKY}
V.~Blomer, R.~Khan, and M.~Young.
\newblock Distribution of mass of holomorphic cusp forms.
\newblock {\em Duke Math. J.}, 162(14):2609--2644, 2013.

\bibitem[BMi11]{BM2011}
V.~Blomer and P.~Michel.
\newblock Sup-norms of eigenfunctions on arithmetic ellipsoids.
\newblock {\em Int. Math. Res. Not. IMRN}, (21):4934--4966, 2011.

\bibitem[BMi13]{BM2013}
V.~Blomer and P.~Michel.
\newblock Hybrid bounds for automorphic forms on ellipsoids over number fields.
\newblock {\em J. Inst. Math. Jussieu}, 12(4):727--758, 2013.

\bibitem[BMa15]{BlomerMagaPGL4}
V.~Blomer and P.~Maga.
\newblock The sup-norm problem for {PGL}(4).
\newblock {\em Int. Math. Res. Not. IMRN}, (14):5311--5332, 2015.

\bibitem[BMa16]{BlomerMagaPGLn}
V.~Blomer and P.~Maga.
\newblock Subconvexity for sup-norms of cusp forms on {$\rm {PGL}(n)$}.
\newblock {\em Selecta Math. (N.S.)}, 22(3):1269--1287, 2016.

\bibitem[Bor63]{BorelFiniteness}
Armand Borel.
\newblock Some finiteness properties of adele groups over number fields.
\newblock {\em Inst. Hautes \'{E}tudes Sci. Publ. Math.}, (16):5--30, 1963.

\bibitem[BP16]{BlomerPohl}
V.~Blomer and A.~Pohl.
\newblock The sup-norm problem on the {S}iegel modular space of rank two.
\newblock {\em Amer. J. Math.}, 138(4):999--1027, 2016.

\bibitem[CF99]{ChinburgFriedman}
T.~Chinburg and E.~Friedman.
\newblock An embedding theorem for quaternion algebras.
\newblock {\em J. London Math. Soc. (2)}, 60(1):33--44, 1999.

\bibitem[CL16]{ChuLi}
M.~Chu and H.~Li.
\newblock Small generators of cocompact arithmetic {F}uchsian groups.
\newblock {\em Proc. Amer. Math. Soc.}, 144(12):5121--5127, 2016.

\bibitem[DS15]{DasSengupta}
S.~Das and J.~Sengupta.
\newblock {$L^\infty$} norms of holomorphic modular forms in the case of
  compact quotient.
\newblock {\em Forum Math.}, 27(4):1987--2001, 2015.

\bibitem[GZ86]{grosszagier}
B.~H. Gross and D.~B. Zagier.
\newblock Heegner points and derivatives of {$L$}-series.
\newblock {\em Invent. Math.}, 84(2):225--320, 1986.

\bibitem[HL94]{HoffsteinLockhart}
J.~Hoffstein and P.~Lockhart.
\newblock Coefficients of {M}aass forms and the {S}iegel zero.
\newblock {\em Ann. of Math. (2)}, 140(1):161--181, 1994.
\newblock With an appendix by Dorian Goldfeld, Hoffstein and Daniel Lieman.

\bibitem[HNS19]{hu2019some}
Y.~Hu, P.~D. Nelson, and A.~Saha.
\newblock Some analytic aspects of automorphic forms on {$\rm GL(2)$} of
  minimal type.
\newblock {\em Comment. Math. Helv.}, 94(4):767--801, 2019.

\bibitem[HS20]{HuSaha}
Yueke Hu and Abhishek Saha.
\newblock Sup-norms of eigenfunctions in the level aspect for compact
  arithmetic surfaces, {II}: newforms and subconvexity.
\newblock {\em Compos. Math.}, 156(11):2368--2398, 2020.

\bibitem[HT12]{HT2}
G.~Harcos and N.~Templier.
\newblock On the sup-norm of {M}aass cusp forms of large level: {II}.
\newblock {\em Int. Math. Res. Not. IMRN}, (20):4764--4774, 2012.

\bibitem[HT13]{HT3}
G.~Harcos and N.~Templier.
\newblock On the sup-norm of {M}aass cusp forms of large level. {III}.
\newblock {\em Math. Ann.}, 356(1):209--216, 2013.

\bibitem[Ich08]{Ichino}
Atsushi Ichino.
\newblock Trilinear forms and the central values of triple product
  {$L$}-functions.
\newblock {\em Duke Math. J.}, 145(2):281--307, 2008.

\bibitem[IS95]{IwaniecSarnak}
H.~Iwaniec and P.~Sarnak.
\newblock {$L^\infty$} norms of eigenfunctions of arithmetic surfaces.
\newblock {\em Ann. of Math. (2)}, 141(2):301--320, 1995.

\bibitem[JL70]{JacquetLanglands}
H.~Jacquet and R.~P. Langlands.
\newblock {\em Automorphic forms on {${\rm GL}(2)$}}.
\newblock Lecture Notes in Mathematics, Vol. 114. Springer-Verlag, Berlin-New
  York, 1970.

\bibitem[K{\i}r14]{halflevel}
E.~M. K{\i}ral.
\newblock Bounds on sup-norms of half-integral weight modular forms.
\newblock {\em Acta Arith.}, 165(4):385--399, 2014.

\bibitem[KNS22]{Masstheta4moment}
I.~Khayutin, P.~D. Nelson, and R.~S. Steiner.
\newblock {T}heta functions, fourth moments of eigenforms, and the sup-norm
  problem {II}.
\newblock {\em {P}reprint}, 2022.
\newblock {\tt arXiv:2207.12351}.

\bibitem[Mar14]{Marshallsup}
S.~Marshall.
\newblock Upper bounds for {M}aass forms on semisimple groups.
\newblock {\em {P}reprint}, 2014.
\newblock {\tt arXiv:1405.7033}.

\bibitem[Nel11]{Nelsonlevelequidistribution}
Paul~D. Nelson.
\newblock Equidistribution of cusp forms in the level aspect.
\newblock {\em Duke Math. J.}, 160(3):467--501, 2011.

\bibitem[Nel15]{NelsonComp}
P.~D. Nelson.
\newblock Evaluating modular forms on {S}himura curves.
\newblock {\em Math. Comp.}, 84(295):2471--2503, 2015.

\bibitem[Nel16]{NelsonQV1}
P.~D. Nelson.
\newblock Quantum variance on quaternion algebras, {I}.
\newblock {\em {P}reprint}, 2016.
\newblock {\tt arXiv:1601.02526}.

\bibitem[Nel17]{NelsonQV2}
P.~D. Nelson.
\newblock Quantum variance on quaternion algebras, {II}.
\newblock {\em {P}reprint}, 2017.
\newblock {\tt arXiv:1702.02669}.

\bibitem[Nel19]{NelsonQV3}
P.~D. Nelson.
\newblock Quantum variance on quaternion algebras, {III}.
\newblock {\em {P}reprint}, 2019.
\newblock {\tt arXiv:1903.08686}.

\bibitem[Nel20]{NelsonSymmL}
Paul~D. Nelson.
\newblock Bounds for twisted symmetric square {$L$}-functions via half-integral
  weight periods.
\newblock {\em Forum Math. Sigma}, 8:Paper No. e44, 21, 2020.

\bibitem[Pet40]{Pet1}
H.~Petersson.
\newblock \"{U}ber eine {M}etrisierung der automorphen {F}ormen und die
  {T}heorie der {P}oincar\'{e}schen {R}eihen.
\newblock {\em Math. Ann.}, 117:453--537, 1940.

\bibitem[Pet41]{Pet2}
H.~Petersson.
\newblock Einheitliche {B}egr\"{u}ndung der {V}ollst\"{a}ndigkeitss\"{a}tze
  f\"{u}r die {P}oincar\'{e}schen {R}eihen von reeller {D}imension bei
  beliebigen {G}renzkreisgruppen von erster {A}rt.
\newblock {\em Abh. Math. Sem. Hansischen Univ.}, 14:22--60, 1941.

\bibitem[RW21]{RamacherWakatsuki}
Pablo Ramacher and Satoshi Wakatsuki.
\newblock Subconvex bounds for {H}ecke-{M}aass forms on compact arithmetic
  quotients of semisimple {L}ie groups.
\newblock {\em Math. Z.}, 298(3-4):1383--1424, 2021.

\bibitem[Sah17a]{saha2017hybrid}
A.~Saha.
\newblock Hybrid sup-norm bounds for {M}aass newforms of powerful level.
\newblock {\em Algebra Number Theory}, 11(5):1009--1045, 2017.

\bibitem[Sah17b]{Saha14}
A.~Saha.
\newblock On sup-norms of cusp forms of powerful level.
\newblock {\em J. Eur. Math. Soc. (JEMS)}, 19(11):3549--3573, 2017.

\bibitem[Sah20]{saha2019sup}
A.~Saha.
\newblock Sup-norms of eigenfunctions in the level aspect for compact
  arithmetic surfaces.
\newblock {\em Math. Ann.}, 376(1-2):609--644, 2020.

\bibitem[Sar03]{SarnakWatsonL4}
P.~Sarnak.
\newblock Spectra of hyperbolic surfaces.
\newblock {\em Bull. Amer. Math. Soc. (N.S.)}, 40(4):441--478, 2003.

\bibitem[Sel42]{selberg1942zeros}
A.~Selberg.
\newblock On the zeros of {R}iemann's zeta-function.
\newblock {\em Skr. Norske Vid.-Akad. Oslo I}, 1942(10):59, 1942.

\bibitem[Shi72]{Shimizu}
H.~Shimizu.
\newblock Theta series and automorphic forms on {${\rm GL}_{2}$}.
\newblock {\em J. Math. Soc. Japan}, 24:638--683, 1972.

\bibitem[Sog88]{Sogge}
C.~D. Sogge.
\newblock Concerning the {$L^p$} norm of spectral clusters for second-order
  elliptic operators on compact manifolds.
\newblock {\em J. Funct. Anal.}, 77(1):123--138, 1988.

\bibitem[Ste16]{Steiner16}
R.~S. Steiner.
\newblock Uniform bounds on sup-norms of holomorphic forms of real weight.
\newblock {\em Int. J. Number Theory}, 12(5):1163--1185, 2016.

\bibitem[Ste17]{halfsup}
R.~S. Steiner.
\newblock Supnorm of modular forms of half-integral weight in the weight
  aspect.
\newblock {\em Acta Arith.}, 177(3):201--218, 2017.

\bibitem[Ste20]{S3sup}
R.~S. Steiner.
\newblock Sup-norm of {H}ecke-{L}aplace eigenforms on {$S^3$}.
\newblock {\em Math. Ann.}, 377(1-2):543--553, 2020.

\bibitem[Ste22]{SteinerSmallDiameter}
R.~S. Steiner.
\newblock Small diameters and generators for arithmetic lattices in
  {$\mathrm{SL}_2(\mathbb{R})$} and certain {R}amanujan graphs.
\newblock {\em {P}reprint}, 2022.
\newblock {\tt arXiv:2207.12684}.

\bibitem[SW71]{Stein-Weiss-Fourier}
Elias~M. Stein and Guido Weiss.
\newblock {\em Introduction to {F}ourier analysis on {E}uclidean spaces}.
\newblock Princeton Mathematical Series, No. 32. Princeton University Press,
  Princeton, N.J., 1971.

\bibitem[SY19]{sun2019double}
H.~Sun and Y.~Ye.
\newblock Double first moment for {$L(\frac 12,{\rm Sym}^2f\times g)$} by
  applying {P}etersson's formula twice.
\newblock {\em J. Number Theory}, 202:141--159, 2019.

\bibitem[Tem10]{HT1}
N.~Templier.
\newblock On the sup-norm of {M}aass cusp forms of large level.
\newblock {\em Selecta Math. (N.S.)}, 16(3):501--531, 2010.

\bibitem[Tem15]{Thybrid}
N.~Templier.
\newblock Hybrid sup-norm bounds for {H}ecke--{M}aass cusp forms.
\newblock {\em J. Eur. Math. Soc. (JEMS)}, 17(8):2069--2082, 2015.

\bibitem[Van97]{Vanderkam}
J.~M. VanderKam.
\newblock {$L^\infty$} norms and quantum ergodicity on the sphere.
\newblock {\em Internat. Math. Res. Notices}, (7):329--347, 1997.

\bibitem[Vig77]{Vigneras}
M.-F. Vign\'{e}ras.
\newblock S\'{e}ries th\^{e}ta des formes quadratiques ind\'{e}finies.
\newblock In {\em S\'{e}minaire {D}elange-{P}isot-{P}oitou, 17e ann\'{e}e
  (1975/76), {T}h\'{e}orie des nombres: {F}asc. 1, {E}xp. {N}o. 20}, page~3.
  1977.

\bibitem[Voi18]{VoightQA}
John Voight.
\newblock Quaternion algebras, 2018.
\newblock Available at https://math.dartmouth.edu/~jvoight/quat.html.

\bibitem[Wal85]{Waldspurger}
J.-L. Waldspurger.
\newblock Sur les valeurs de certaines fonctions {$L$} automorphes en leur
  centre de sym\'{e}trie.
\newblock {\em Compositio Math.}, 54(2):173--242, 1985.

\bibitem[Wat08]{Watsontriple}
T.~C. Watson.
\newblock {R}ankin {T}riple {P}roducts and {Q}uantum {C}haos.
\newblock {\em Ph.D. dissertation, Princeton University}, 2008.
\newblock {\tt arXiv:0810.0425}.

\bibitem[Wei64]{Weil}
A.~Weil.
\newblock Sur certains groupes d'op\'{e}rateurs unitaires.
\newblock {\em Acta Math.}, 111:143--211, 1964.

\bibitem[Xia07]{Supnormintweight}
H.~Xia.
\newblock On {$L^\infty$} norms of holomorphic cusp forms.
\newblock {\em J. Number Theory}, 124(2):325--327, 2007.

\bibitem[Zag76]{EichSelbZagier}
D.~Zagier.
\newblock The {E}ichler--{S}elberg {T}race {F}ormula on
  {$\mathrm{SL}_2(\mathbb{Z})$}.
\newblock {\em An appendix to the book "{I}ntroduction to modular forms" by
  {S}. Lang, Springer-Verlag, Berlin-New York}, pages 44--55, 1976.

\bibitem[Zag77]{ZagierDoiLecture}
D.~Zagier.
\newblock Modular forms whose {F}ourier coefficients involve zeta-functions of
  quadratic fields.
\newblock In {\em Modular functions of one variable, {VI} ({P}roc. {S}econd
  {I}nternat. {C}onf., {U}niv. {B}onn, {B}onn, 1976)}, pages 105--169. Lecture
  Notes in Math., Vol. 627, 1977.

\end{thebibliography}
